\documentclass{article}

\usepackage[english]{babel}
\usepackage[utf8]{inputenc}
\usepackage[T1]{fontenc}
\usepackage{pifont} 
\usepackage[a4paper,top=2cm,bottom=2cm,left=3cm,right=3cm,marginparwidth=1.75cm]{geometry}

\usepackage{framed}
\usepackage{amsmath}
\usepackage{amsfonts,amssymb,amsthm,dsfont,mathrsfs,stmaryrd,bbm}
\usepackage[new]{old-arrows}
\usepackage{graphicx}
\usepackage{tikz}
\usepackage[colorlinks=true, allcolors=blue]{hyperref}
\usepackage{tabularx}
\usepackage{booktabs}
\usepackage[labelfont=bf,format=plain,justification=raggedright,singlelinecheck=false]{caption}

\newtheorem{thm}{Theorem}[section]
\newtheorem{cor}[thm]{Corollary}
\newtheorem{prop}[thm]{Proposition}
\newtheorem{lem}[thm]{Lemma}
\newtheorem{defi}[thm]{Definition}

\theoremstyle{remark}
\newtheorem*{rmk}{Remark}

\newcommand{\R}{\mathds R}
\newcommand{\Q}{\mathds Q}

\newcommand{\N}{\mathds N}
\newcommand{\Zintff}[1]{\left\llbracket #1 \right\rrbracket}
\newcommand{\Zintfo}[1]{\left\llbracket #1 \right\llbracket}

\newcommand{\Zintof}[1]{\left\rrbracket #1 \right\rrbracket}
\newcommand{\id}{\textup{Id}}

\newcommand{\abs}[1]{\left\lvert#1\right\rvert}
\newcommand{\norm}[1]{\left\lVert#1\right\rVert}
\DeclareMathOperator{\sgn}{sign}

\newcommand{\st}{\text{ such that }}
\newcommand{\ie}[1][\,]{\textit{i.e.#1}}
\newcommand{\eg}[1][\, ]{\textit{e.g.#1}}

\newcommand{\tend}[2][+\infty]{\scriptscriptstyle #2\rightarrow #1 }

\newcommand{\cv}[2][+\infty]{\underset{\tend[#1]{#2}}{\longrightarrow}} 

\renewcommand{\d}{\text{\small d}}

\newcommand{\leb}{\textup{Leb}}
\newcommand{\supp}{\textup{supp}}

\newcommand{\evt}[1]{\!\left( #1 \right)}
\newcommand{\cro}[1]{\!\left[ #1 \right]}
\newcommand{\1}{\mathbbm 1}
\newcommand{\as}[1][-]{\text{#1a.s.}}

\renewcommand{\ae}[1][-]{\text{#1a.e.}}

\newcommand{\iid}{\text{i.i.d.\,}}

\newcommand{\prob}{(\mathbb P)}
\newcommand{\dstb}{(d)}

\renewcommand{\P}{\mathbb P}
\newcommand{\E}{\mathbb E}

\newcommand{\Unif}{\textsc{Unif}}

\newcommand{\Geom}{{\textsc{Geom}}_{\N}}

\newcommand{\BGW}[1][\mu]{\textup{BGW}_{\!#1}}

\newcommand{\Luka}[1][ ]{\textsc{\L{}ukasiewicz#1}}

\newcommand{\Ul}{\mathcal U}

\newcommand{\Tree}{\mathcal T}

\newcommand{\dGH}{d_{\textup{\tiny GH}}}
\newcommand{\dGHP}{d_{\textup{\tiny GHP}}}
\newcommand{\dis}{\textup{dis}}

\newcommand{\disc}{\textup{Disc}}
\newcommand{\PR}{\mathscr P}
\DeclareMathOperator{\rot}{Rot}
\DeclareMathOperator{\corot}{Tor}
\DeclareMathOperator{\Loop}{Loop}

\title{Rotating random trees with \textsc{Skorokhod}'s $M_1$ topology}
\date{}
\author{Antoine \textsc{Aurillard}\footnote{CMAP, CNRS, École polytechnique, Institut Polytechnique de Paris, 91120 Palaiseau, France. \newline Institut Camille Jordan, CNRS UMR 5208, Universite Claude Bernard Lyon 1, 43
Boulevard du 11 novembre 1918, 69622 Villeurbanne Cedex, France. \newline \textit{email:} antoine.aurillard@polytechnique.edu}}

\begin{document}

\maketitle
\begin{abstract}
    We extend the classical coding of measured $\R$-trees by continuous excursion-type functions to càdlàg excursion-type functions through the notion of parametric representations. The main feature of this extension is its continuity properties with respect to the \textsc{Gromov-Hausdorff-Prokhorov} topology for $\R$-trees and \textsc{Skorokhod}'s $M_1$ topology for càdlàg functions. As a first application, we study the $\R$-trees $\mathscr T_{\mathbbm x^{(\alpha)}}$ encoded by excursions of spectrally positive $\alpha$-stable \textsc{Lévy} processes for $\alpha \in (1,2]$. In a second time, we use this setting to study the large-scale effects of a well-known bijection between plane trees and binary trees, the so-called rotation. \textsc{Marckert} has proved that the rotation acts as a dilation on large uniform trees, and we show that this remains true when the rotation is applied to large critical \textsc{Bienaymé} trees with offspring distribution attracted to a Gaussian distribution. However, this does not hold anymore when the offspring distribution falls in the domain of attraction of an $\alpha$-stable law with $\alpha \in (1,2)$, and instead we prove that the scaling limit of the rotated trees is $\mathscr T_{\mathbbm x^{(\alpha)}}$.

\end{abstract}

\section{Introduction}\label{sct:Intro}

In studies on random trees, a classical approach is to encode trees by real-valued functions. It can be used to establish scaling limits of trees and to derive some of their geometric properties (see \eg \cite{Legall2005RandomTreesApplications,DuquesneLegall2005LevyTrees}) but also to define and understand related objects such as random snakes \cite{DuquesneLegall2002RandomTreeLevyProcessesBranching,Marzouk2020ScalingLimitStableSnake} and looptrees \cite{CurienKortchemski2014StableLooptrees,KhanfirPreprintConvergenceLooptrees}. In the present paper, we broaden this approach by defining measured $\R$-trees encoded by a \textit{càdlàg} (\ie right-continuous with left limits) excursion-type function in a way that is continuous with respect to \textsc{Skorokhod}'s $M_1$ topology on the space of càdlàg functions $D([0,1])$. This construction generalizes the encoding of $\R$-trees by a continuous \textit{contour function} in a way that preserves some tools already used to study the geometric properties of the encoded $\R$-trees. Moreover, since \textsc{Skorokhod}'s $M_1$ topology is weaker than the uniform topology and \textsc{Skorokhod}'s $J_1$ topology which are usually used in the studies on random trees, this extension also enables to get scaling limits of trees by means of a weak functional convergence of their encoding processes.

Based on this, we are able to study the large-scale effects of the so-called \textit{rotation correspondence} on random trees. This has already been investigated in \cite{Marckert2004Rotation} in the case of uniform trees, but here we consider the wider setting of large size-conditioned \textsc{Bienaymé-Galton-Watson} trees (which will simply be called \textsc{Bienaymé} trees in the following) with a critical offspring distribution that lies in the domain of attraction of a stable distribution with index $\alpha \in (1,2]$. It appears that radically different behaviours occur depending on $\alpha$. In the Gaussian case $\alpha=2$, the rotation asymptotically acts as a dilation, which means that up to a dilation of distances the rotation does not affect the scaling limits of these \textsc{Bienaymé} trees. This contrasts with the case $\alpha \in (1,2)$, where the scaling limits of rotated trees form a new family of random $\R$-trees which are encoded by excursions of spectrally positive $\alpha$-stable \textsc{Lévy} processes. We also investigate some geometric properties of these trees, such as their \textsc{Hausdorff} dimension, and establish a link with the stable looptrees introduced in \cite{CurienKortchemski2014StableLooptrees}.

Let us present our main results. We assume some familiarity with plane trees and their encoding processes (contour process, height process and \Luka walk), with scaling limits of trees and with critical \textsc{Bienaymé} trees whose offspring distribution lies in the domain of attraction of a stable distribution. These notions will be detailed in Sections~\ref{sct:planeTrees},~\ref{sct:ScalingLimitTrees} and ~\ref{sct:Encoding Processes}.

\paragraph{\textsc{Skorokhod}'s $M_1$ topology and (measured) $\R$-trees.}

Our idea underlying the encoding of $\R$-trees by a càdlàg function $x$ is heavily based on a core notion of \textsc{Skorokhod}'s $M_1$ topology, namely the notion of parametric representations of $x$. In short, a parametric representation consists of an increasing and continuous parametrization of the graph of $x$ where the discontinuities are filled in by segments. Informally, this notion provides continuous functions that are \textit{stretched} versions of $x$. Combining this with the usual encoding of $\R$-trees by functions in $C_0([0,1],\R_+)$, which is invariant by \textit{stretching}, gives a uniquely defined $\R$-tree $\mathscr T_x$ associated with $x \in D_0([0,1],\R_+)$ (see Definition~\ref{def:Discontinuous Contour and Real Tree}). Here $C_0$ (resp. $D_0$) refers to the space of continuous (resp. càdlàg) functions that vanish at $0$ and $1$. The tree $\mathscr T_x$ may also be endowed in a natural way with a measure $\mu_x$ encoded by $x$ (Definition~\ref{def:Discontinuous Contour and Measured Real Tree}). We will give some basic properties of this construction, but its main point is that it also extends the continuity properties of $f\in C_0([0,1],\R_+) \mapsto \mathscr T_f$ and $f\in C_0([0,1],\R_+) \mapsto (\mathscr T_f,\mu_f)$ to the setting of $D_0([0,1],\R_+)$ endowed with the $M_1$ distance. More precisely, in the context of the \textsc{Gromov-Hausdorff} topology where we deal with $\R$-trees without considering measures,  the \textsc{Lipschitz} property of the usual setting immediately extends as follows
\begin{prop}\label{prop:GHlipschitzM1}
    For all $x,y\in D_0([0,1],\R_+)$, 
    \begin{equation*}
        \dGH(\mathscr T_x,\mathscr T_y) \leq 2d_{M_1}(x,y).
    \end{equation*} 
 \end{prop}
 In the context of measured $\R$-trees, Proposition~\ref{prop:GHP M1} gives a slightly weaker form of continuity with respect to the \textsc{Gromov-Hausdorff-Prokhorov} topology.

\begin{prop}\label{prop:GHP M1}
    The map $x \in D_0([0,1],\R_+) \mapsto \bigl(\mathscr T_x,\mu_x\bigr)$ is continuous for $d_{M_1}$ and $\dGHP$.
\end{prop}


\paragraph{Large-scale effects of the rotation on \textsc{Bienaymé} trees.}

\begin{figure}[h!]
    \centering
    \includegraphics[width=0.75\linewidth]{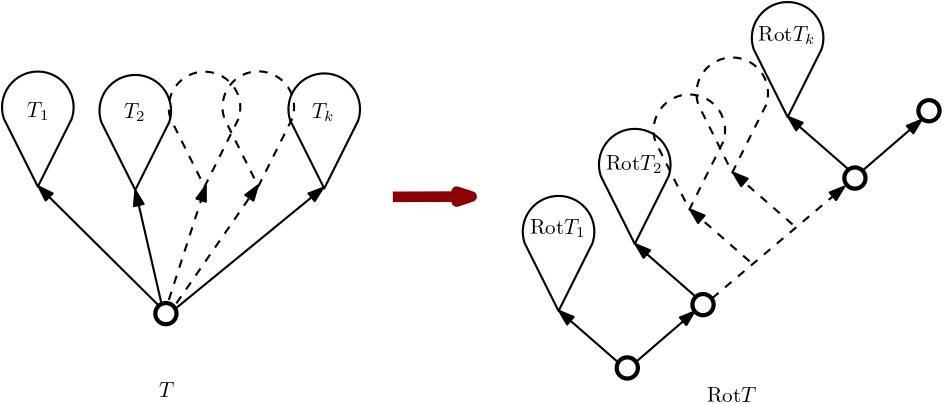}
        \caption{\centering Recursive definition of $\rot$}

    \label{fig:DefRecRotation}
\end{figure}

Our main application of the previous extension to càdlàg functions is the study of the rotation correspondence, denoted by $\rot$ in the following, applied to large critical \textsc{Bienaymé} trees. The rotation is a bijection between plane trees with $n$ vertices and binary plane trees with $n$ leaves which may be understood in a recursive way:
\begin{itemize}
    \item[] \textbf{Base case:} $\rot \{\varnothing\} =\{\varnothing\}$, where $\{\varnothing\}$ is the plane tree with a single vertex.
    \item[] \textbf{Inductive step:} Consider a plane tree $T$ with $n\geq 2$ vertices. Denote by $T_1,\ldots,T_k$ the $k\geq 1$ subtree(s) grafted on the root. Then $\rot T$ consists of a spine of $k+1$ vertices with $\rot T_1$ grafted on  the left of the $1$st one, \ldots, $\rot T_k$ grafted on  the left of the $k$th one (see Figure~\ref{fig:DefRecRotation}).
\end{itemize}

To understand how the rotation affects a critical \textsc{Bienaymé} tree conditioned to have $n$ vertices, say $\Tree_n$, for large $n$,  we determine the joint scaling limit (in distribution) of $\bigl( \Tree_n,\rot \Tree_n\bigr)$, in the setting of the \textsc{Gromov-Hausdorff} topology. This question has already been studied by \textsc{Marckert} \cite{Marckert2004Rotation} in the special case of $T_n$ a uniform plane tree with $n$ vertices, which corresponds to a \textsc{Bienaymé} tree with critical geometric offspring distribution. \textsc{Marckert}'s result, recast in the setting of the \textsc{Gromov-Hausdorff} topology, is

\begin{equation*}
    \Bigl(\frac{1}{\sqrt{n}}T_n,\frac{1}{\sqrt{n}}\rot T_n\Bigr) \xrightarrow[\tend{n}]{\dstb} \Bigl(\sqrt 2 \mathscr T_{\mathbbm e},2\sqrt 2\mathscr T_{\mathbbm e}\Bigr),
\end{equation*} where $\mathscr T_{\mathbbm e}$ is a Brownian Continuum Random Tree (in short, Brownian CRT), that is an $\R$-tree encoded by normalized Brownian excursion $\mathbbm e$. Note that it is known since the pioneering work of \textsc{Aldous} \cite{Aldous1991Continuum1,Aldous1993Continuum3} that the scaling limit of $\Tree_n$ alone, when its critical offspring distribution $\mu$ has a finite variance $\sigma^2$ and we rescale its graph distance by $\sqrt n$, is a Brownian CRT dilated by $2/\sigma$. Thus, the main part of \textsc{Marckert}'s result is the fact that $T_n$ and $\rot T_n$ jointly converge toward the \textit{same} Brownian CRT, up to a dilation by a factor $2$. Given this context, it is natural to ask whether a similar result holds for all \textsc{Bienaymé} trees whose critical offspring  distribution has a finite variance. Moreover, \textsc{Aldous}' Theorem has been extended by \textsc{Duquesne} \cite{Duquesne2003Contour} to a wider class of critical \textsc{Bienaymé} trees whose scaling limits form a family of $\R$-trees (the so-called stable trees) which contains the Brownian CRT as a special case, so the same question may be asked in this more general context. In order to answer this question, we first give more details on \textsc{Duquesne}'s Theorem: when $\mu$ is attracted to a stable distribution with index $\alpha \in (1,2]$, the (rescaled) encoding processes of $\Tree_n$ jointly converge towards a normalized excursion $\mathbbm x=\mathbbm x^{(\alpha)}$ of a spectrally positive $\alpha$-stable \textsc{Lévy} process together with the continuous-time height excursion $\mathbbm h$ associated with $\mathbbm x$, and the scaling limit of $\Tree_n$ itself is $\mathscr T_{\mathbbm h}$ \ie the  $\R$-tree encoded by this continuous excursion $\mathbbm h$. Based on these results and our extension exposed in the previous part, we prove that the distributional scaling limit of $\rot \Tree_n$ is $\mathscr T_{\mathbbm x}$, the $\R$-tree encoded by the càdlàg excursion $\mathbbm x$.


\begin{thm}\label{thm:MainResult}
Let $\Tree_n$ be distributed as $\BGW\evt{\,\cdot\,\vert\textup{\#vertices = $n$}}$ where the offspring distribution $\mu$ is critical. 

\begin{itemize}
    \item Assume that $\mu$ has variance $\sigma^2<+\infty$. We have the convergence in distribution
\begin{equation*}
    \Bigl(\frac{1}{\sqrt{n}}\Tree_n,\frac{1}{\sqrt{n}}\rot \Tree_n\Bigr) \xrightarrow[\tend{n}]{\dstb} \Bigl(\frac{2}{\sigma}\mathscr T_{\mathbbm e},\frac{2+\sigma^2}{\sigma}\mathscr T_{\mathbbm e}\Bigr),
\end{equation*}
with respect to the \textsc{Gromov-Hausdorff} topology. 
\item Assume that $\mu$ has variance $\sigma^2=+\infty$, and is attracted to a stable distribution of index $\alpha\in (1,2]$. There is an increasing slowly varying function\footnote{In this paper, one only needs to know that $\ell$ is such that for all $\varepsilon >0$, for $x$ large enough we have $x^{-\varepsilon} \leq \ell(x)\leq x^{\varepsilon}$. See \cite{BinghamGoldieTeugels1987RegularVariation} for a proper definition and a general account of slowly varying functions.} $\ell$ such that
\begin{equation*}
    \Bigl(\frac{\ell(n)}{n^{1-1/\alpha}}\Tree_n,\frac{1}{\ell(n)n^{1/\alpha}}\rot \Tree_n\Bigr) \xrightarrow[\tend{n}]{\dstb} \bigl(\mathscr T_{\mathbbm h},\mathscr T_{\mathbbm x}\bigr),
\end{equation*}
with respect to the \textsc{Gromov-Hausdorff} topology. 
\end{itemize}
\end{thm}

In the Gaussian case $\alpha=2$, we actually have that $\mathbbm h= \mathbbm x$ with this process distributed as $\sqrt 2 \mathbbm e$, and additionally when $\sigma^2=+\infty$ then $\ell(n)\rightarrow+\infty$. To unify the two subcases $\sigma^2<+\infty$ and $\sigma^2=+\infty$, set $\ell(n) = \sigma/\sqrt 2$  when $\sigma^2<+\infty$, then we may rewrite the result in the Gaussian case as
\begin{equation*}
        \dGH\!\left(\frac{1}{\ell(n)\sqrt{n}}\rot \Tree_n,(1+\ell(n)^2)\frac{1}{\ell(n)\sqrt{n}}\Tree_n\right) \xrightarrow[\tend{n}]{\prob} 0.
    \end{equation*}
In particular, whenever $\sigma^2 < +\infty$ the rotation acts as a dilation by a constant factor $1+\sigma^2/2$.

In the case $\alpha \in (1,2)$, not only does the rotation change the scale of the tree, but it also affects its geometry since we will prove that the scaling limit of rotated trees, $\mathscr T_{\mathbbm x}$, has geometric properties different from those of the stable tree $\mathscr T_{\mathbbm h}$.

\paragraph{Main ideas of the proof.} Thanks to the continuity results (Propositions~\ref{prop:GHlipschitzM1} and~\ref{prop:GHP M1}), we prove Theorem~\ref{thm:MainResult} as a corollary of the joint convergence of the encoding processes of $\Tree_n$ and $\rot \Tree_n$, as \textsc{Marckert} did in the case of uniform trees. The complete statement for the joint scaling limits of encoding processes is a bit more technical, it is given in Theorem~\ref{thm:DetailedResult}, we only stress that it directly implies a refinement of Theorem~\ref{thm:MainResult} with $\dGHP$ instead of $\dGH$, and it relies on a comparison of the encoding processes of $\rot \Tree_n$ with those of $\Tree_n$. This comparison is also based on the handy identification  introduced by \textsc{Marckert} between the non-root vertices of a tree $T$ and the internal vertices of $\rot T$. However, we use quite different methods to prove that the relevant encoding processes are asymptotically close to each other. Indeed, \textsc{Marckert} heavily relies on the fact that in his case  $\rot \Tree_n$ is a uniform binary tree and thus may be seen as a size-conditioned \textsc{Bienaymé} tree whose structure is already well understood (see the comment ending Section~\ref{ssct:mainResults} for a more detailed discussion of \textsc{Marckert}'s approach). As this property is specific to the uniform case, we use a different, more general approach here:
\begin{enumerate}

    \item First we express the height of internal vertices of $\rot \Tree_n$ by means of the encoding processes of $\Tree_n$. This requires considering the mirrored version of $\Tree_n$ and its mirrored enumeration.
    \item Then we introduce an intermediate process, consisting of the height of those internal vertices of $\rot \Tree_n$ enumerated in a way corresponding to the mirrored enumeration of $\Tree_n$. The study of the underlying enumeration shows that there is a simple combinatorial relation between this process and the contour process of the internal vertices of $\rot \Tree_n$. As there are also combinatorial relations between this last process and the contour process of $\rot \Tree_n$, we get a link between this contour process of $\rot \Tree_n$ and the encoding processes of $\Tree_n$.
    \item Finally, we combine this with known results on encoding processes of $\Tree_n$ to get the desired joint scaling limits with respect to some adequate topology, here we need \textsc{Skorokhod}'s $M_1$ topology.
 
\end{enumerate}

In short, our approach is to express as many quantities of interest as possible by means of the encoding processes of $\Tree_n$ (rather than $\rot \Tree_n$), and then to establish combinatorial links between the processes easily expressed and the desired encoding processes in order to compare those processes with respect to the $M_1$ distance. Note that these combinatorial relations hold for all plane trees, we only make use of the fact that $\Tree_n$ is a \textsc{Bienaymé} tree when we use the joint convergence of its encoding processes.

\paragraph{Properties of $\mathscr T_{\mathbbm x}$ and links with transformations related to the rotation.}

 The tree $\mathscr T_{\mathbbm x}$ obtained as a scaling limit of rotated trees can also be studied in the setting of $\R$-trees encoded by càdlàg functions. In particular, even though the set $\disc(\mathbbm x)$ of discontinuities of $\mathbbm x$ is dense in $[0,1]$, standard arguments existing for the usual coding will be adapted and applied here in order to deduce, from properties of $\mathbbm x$, that $\mathscr T_{\mathbbm x}$ is  binary and that $(\mathscr T_{\mathbbm x},\mu_{\mathbbm x})$ is a continuum random tree as defined in \cite{Aldous1993Continuum3}.

\begin{prop}\label{prop:Continuum random tree}
    Assume that $\alpha \in (1,2)$. Almost surely, 
    \begin{itemize}
        \item $\mathscr T_{\mathbbm x}$ is binary \ie $\forall v\in \mathscr T_{\mathbbm x},\ \deg(v)\in\{1,2,3\}$;
        \item Let $\mathcal L(\mathscr T_{\mathbbm x})$ denotes the set of leaves $\{v\in \mathscr T_{\mathbbm x}: \deg(v)=1\}$, then $\overline{\mathcal L(\mathscr T_{\mathbbm x})}=\mathscr T_{\mathbbm x}$ ;
        \item $\mu_{\mathbbm x}$ is non-atomic and charges the leaves of $\mathscr T_{\mathbbm x}$, \ie $\mu_{\mathbbm x}\!\left(\mathcal L(\mathscr T_{\mathbbm x})\right)=1$.
    \end{itemize}
    \end{prop}

Recalling that $\mathscr T_{\mathbbm h}$ has almost surely some vertices with infinite degree (see \cite{DuquesneLegall2005LevyTrees}), we immediately get that $\mathscr T_{\mathbbm x}$ cannot be distributed as $ \mathscr T_{\mathbbm h}$. Another striking difference between these trees is their \textsc{Hausdorff} dimensions, as we will prove that $\dim_H\left( \mathscr T_{\mathbbm x}\right) = \alpha$ while it is known that $\dim_H\left( \mathscr T_{\mathbbm h}\right) = \alpha/(\alpha-1)$ (see \cite{DuquesneLegall2005LevyTrees,HaasMiermont2004FragmentationTrees}). This also shows that $\mathscr T_{\mathbbm x}$ is not a Brownian CRT as soon as $\alpha < 2$.

Our study of $\dim_H\left( \mathscr T_{\mathbbm x}\right)$ relies on an interesting link between this tree and another metric space encoded by $\mathbbm x$, the stable looptree $\mathscr L_{\mathbbm x}$ introduced by \textsc{Curien \& Kortchemski} in \cite{CurienKortchemski2014StableLooptrees}. In short, $\mathscr L_{\mathbbm x}$ appears as the scaling limit of $\Loop(\Tree_n)$, where $\Loop(T)$ is the medial graph of the plane tree $T$ and is called the looptree associated with $T$. In the discrete setting, for every plane tree $T$ we will explain that $\rot T$ (without its leaves) can actually be seen as a specific spanning tree of $\Loop(T)$. We will then argue that in the setting of metric spaces, $\mathscr T_{\mathbbm x}$ may also be seen as a \textit{spanning $\R$-tree} of $\mathscr L_{\mathbbm x}$. Note that \textsc{Khanfir} \cite{KhanfirPreprintConvergenceLooptrees} also studied some relations between a tree encoded by a discontinuous contour function and the looptree encoded by the same function, but his notion of $\R$-tree encoded by a discontinuous excursion is different from ours. Here the relation between $\mathscr T_{\mathbbm x}$ and $\mathscr L_{\mathbbm x}$ is the following:

\begin{prop}\label{prop:Looptree and dimension}
 Assume that $\alpha \in (1,2)$. Consider the canonical projections $\pi_{\mathbbm x}^{\text{tree}}:[0,1] \mapsto \mathscr T_{\mathbbm x}$ and $\pi_{\mathbbm x}^{\text{loop}}:[0,1] \mapsto \mathscr L_{\mathbbm x}$. There exists a unique map $p:\mathscr T_{\mathbbm x} \mapsto \mathscr L_{\mathbbm x}$ which satisfies
 \begin{equation*}
 p\circ \pi_{\mathbbm x}^{\text{tree}}= \pi_{\mathbbm x}^{\text{loop}}.
 \end{equation*}
 Moreover, $p$ is $1$-\textsc{Lipschitz} and $p$ restricted to $\mathscr T_{\mathbbm x}^*=\mathscr T_{\mathbbm x}\!\setminus\!\{\pi_{\mathbbm x}^{\text{tree}}(s) \text{ for } s\in \disc(\mathbbm x)\}$ is one-to-one and onto.

 Finally, these metric spaces have the same \textsc{Hausdorff} dimension:
    \begin{equation*}
        \dim_H\left( \mathscr T_{\mathbbm x}\right) = \dim_H\left( \mathscr L_{\mathbbm x}\right) = \alpha.
    \end{equation*}
\end{prop}

\textit{Remarks.} \hfill
\begin{itemize}
    \item We believe that the first assertion also holds in a sense for a general càdlàg excursion $x$. To be more precise, we do not have in general that $\mathscr T_x$ (as defined in this paper) and $\mathscr L_x$ (as defined in \cite{KhanfirPreprintConvergenceLooptrees} in the case where $x$ has no negative jump) are related in a similar way, but $\mathscr T_x$ might be a spanning $\R$-tree of $\mathscr V_x$ the vernation tree encoded by $x$. This metric space is defined by \textsc{Khanfir} in \cite{KhanfirPreprintConvergenceLooptrees} to unify looptrees and their limits. However, we have not investigated this point further.
    \item  In the case $\alpha=2$, the scaling limit of $\frac{\ell(n)}{\sqrt n}\Loop(\Tree_n)$ is  $C(\mu)\mathscr T_{\sqrt 2\mathbbm e}$ where $C(\mu)$ is a constant depending on $\mu$  (see \cite{CurienHaasKortchemski2015CRTDissections,KortchemskiRichier2020BoundaryMapsLooptrees}). Note that this limit is not a looptree, but as an $\R$-tree it is a vernation tree, and one could clearly consider this tree a spanning $\R$-tree of itself. Thus in this particular case we do have that the scaling limit of $\rot \Tree_n$ is distributed as a spanning $\R$-tree of the scaling limit of $\Loop(\Tree_n)$, but surprisingly when $\sigma^2 < +\infty$ the constant $C(\mu)$ depends on $\mu$ in a more complicated fashion than the constant $1+\sigma^2/2$ arising for $\rot\Tree_n$.
    \item This simple relation between $\mathscr T_{\mathbbm x}$ and $\mathscr L_{\mathbbm x}$ contrasts with the fact that we do not know if $\mathscr T_{\mathbbm x}$ may be expressed as a simple measurable function of $\mathscr T_{\mathbbm h}$ and vice versa.
\end{itemize}

The setting of $\R$-trees encoded by càdlàg functions is also convenient to study the limits in distribution of $\alpha \mapsto \mathscr T_{\mathbbm x^{(\alpha)}}$ and its measured version.  Indeed, by our results of continuity, Propositions\ref{prop:GHlipschitzM1} and~\ref{prop:GHP M1}, convergences for $\alpha \mapsto \mathbbm x^{(\alpha)}$ directly implies the corresponding convergences for trees and measured trees. For instance, the continuity in distribution of $\alpha \in (1,2)  \mapsto \mathbbm x^{(\alpha)}$ (discussed, with the stronger $J_1$ topology, in \cite[Lemma~5.1]{KortchemskiMarzouk2024LevyLooptreesMaps}) is transferred to their measured trees. The most interesting fact obtained this way is that $\alpha \mapsto \mathscr T_{\mathbbm x^{(\alpha)}}$ interpolates between the line segment and the Brownian CRT (see Figure~\ref{fig:Simulation rotated stable trees}) and this also holds for their measured version.

\begin{prop}\label{prop:interpolation}
We have the following convergences in distribution with respect to the \textsc{Gromov-Hausdorff-Prokhorov} topology
    \begin{equation*}
    \bigl(\mathscr T_{\mathbbm x^{(\alpha)}},\mu_{\mathbbm x^{(\alpha)}}\bigr) \xrightarrow[{\tend[1]{\alpha}}]{\dstb} \bigl(([0,1],\abs{\cdot}),\leb\bigr), \hspace{1.5cm}  \bigl(\mathscr T_{\mathbbm x^{(\alpha)}},\mu_{\mathbbm x^{(\alpha)}}\bigr) \xrightarrow[{\tend[2]{\alpha}}]{\dstb} \bigl(\mathscr T_{\mathbbm x^{(2)}},\mu_{\mathbbm x^{(2)}}\bigr)=\sqrt 2\bigl(\mathscr T_{\mathbbm e},\mu_{\mathbbm e}\bigr).
\end{equation*}
\end{prop}
\begin{figure}[h]
    \begin{center}
    \includegraphics[width=0.3\linewidth]{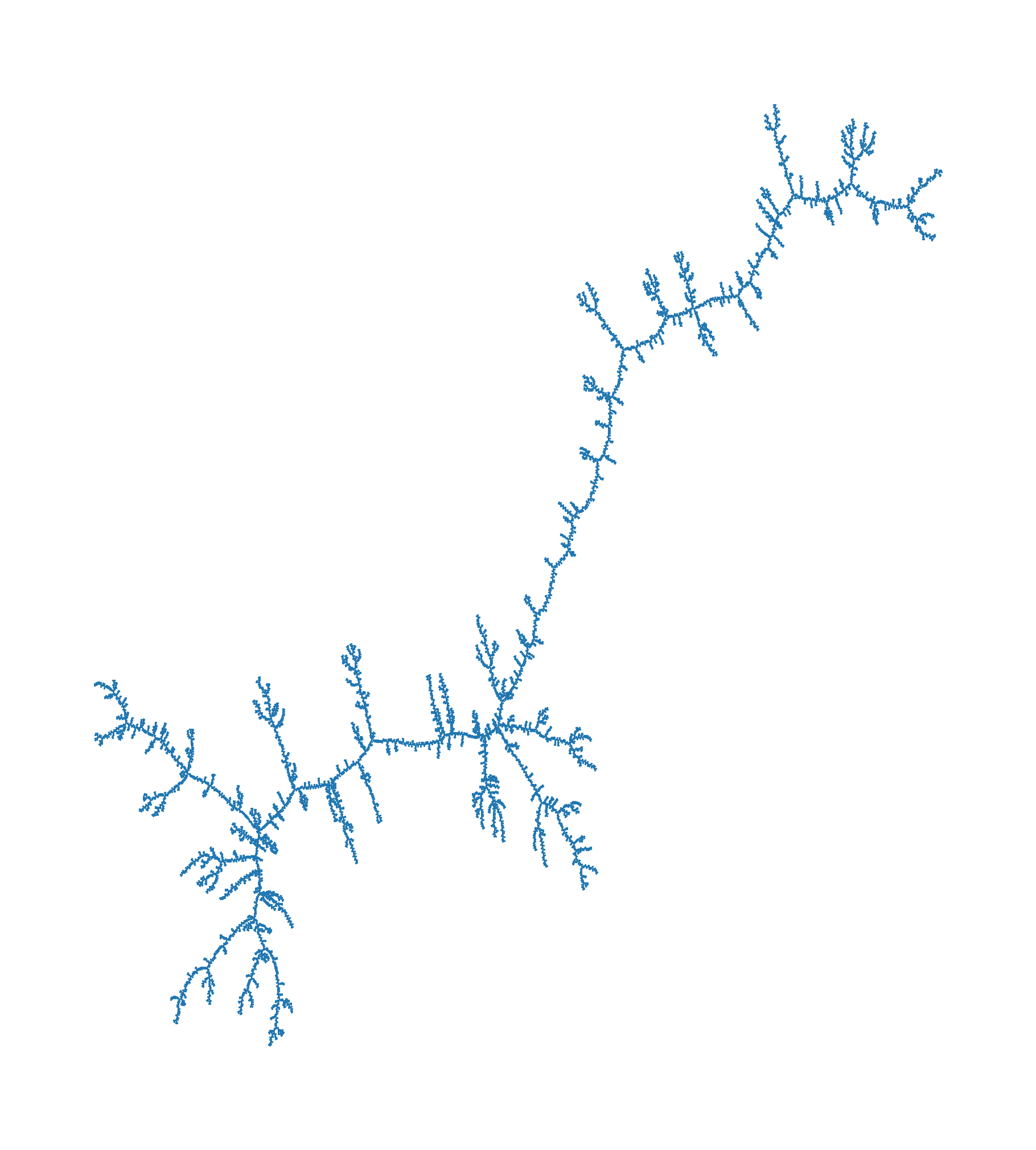}
    \includegraphics[width=0.3\linewidth]{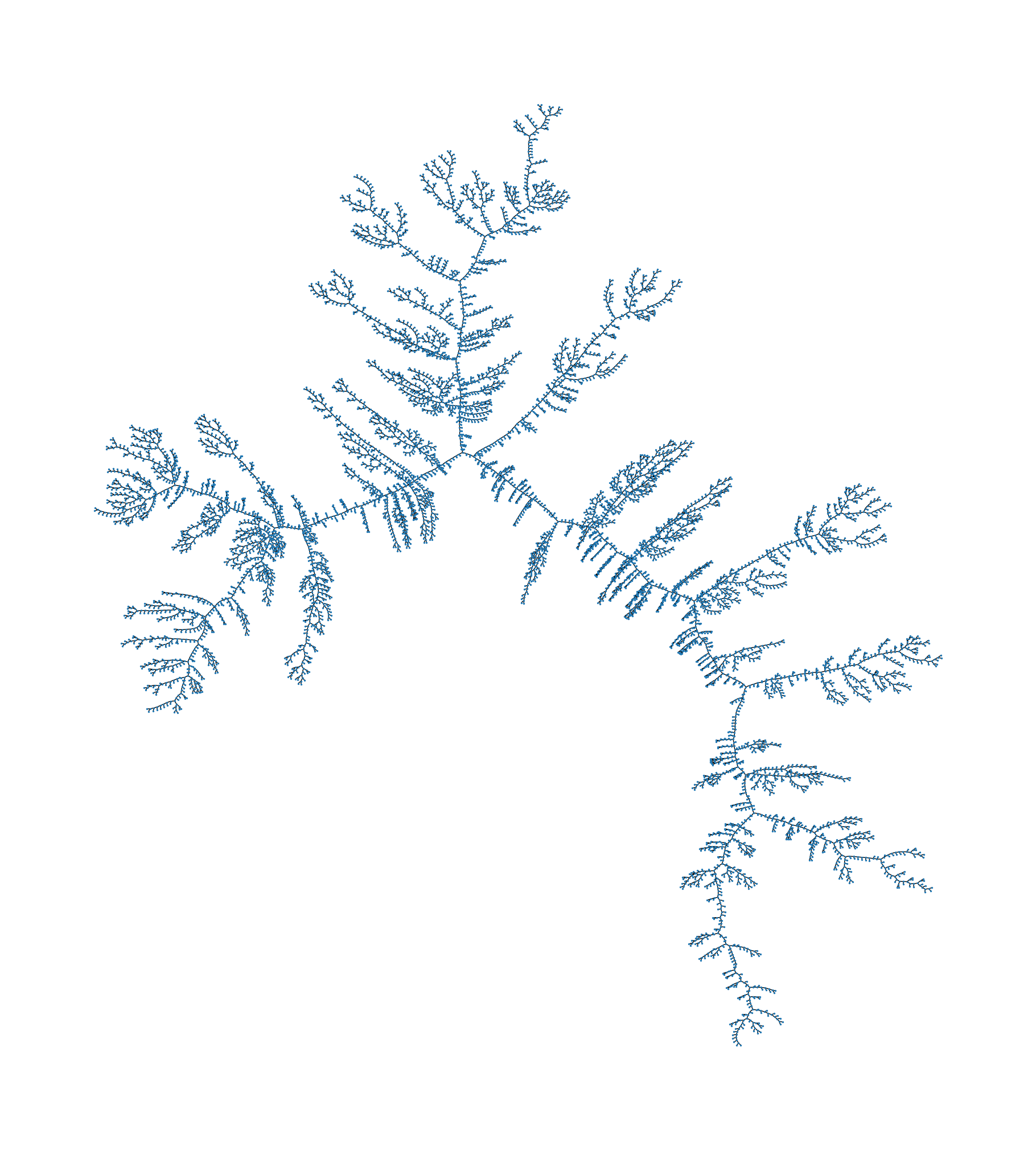}
    \includegraphics[width=0.3\linewidth]{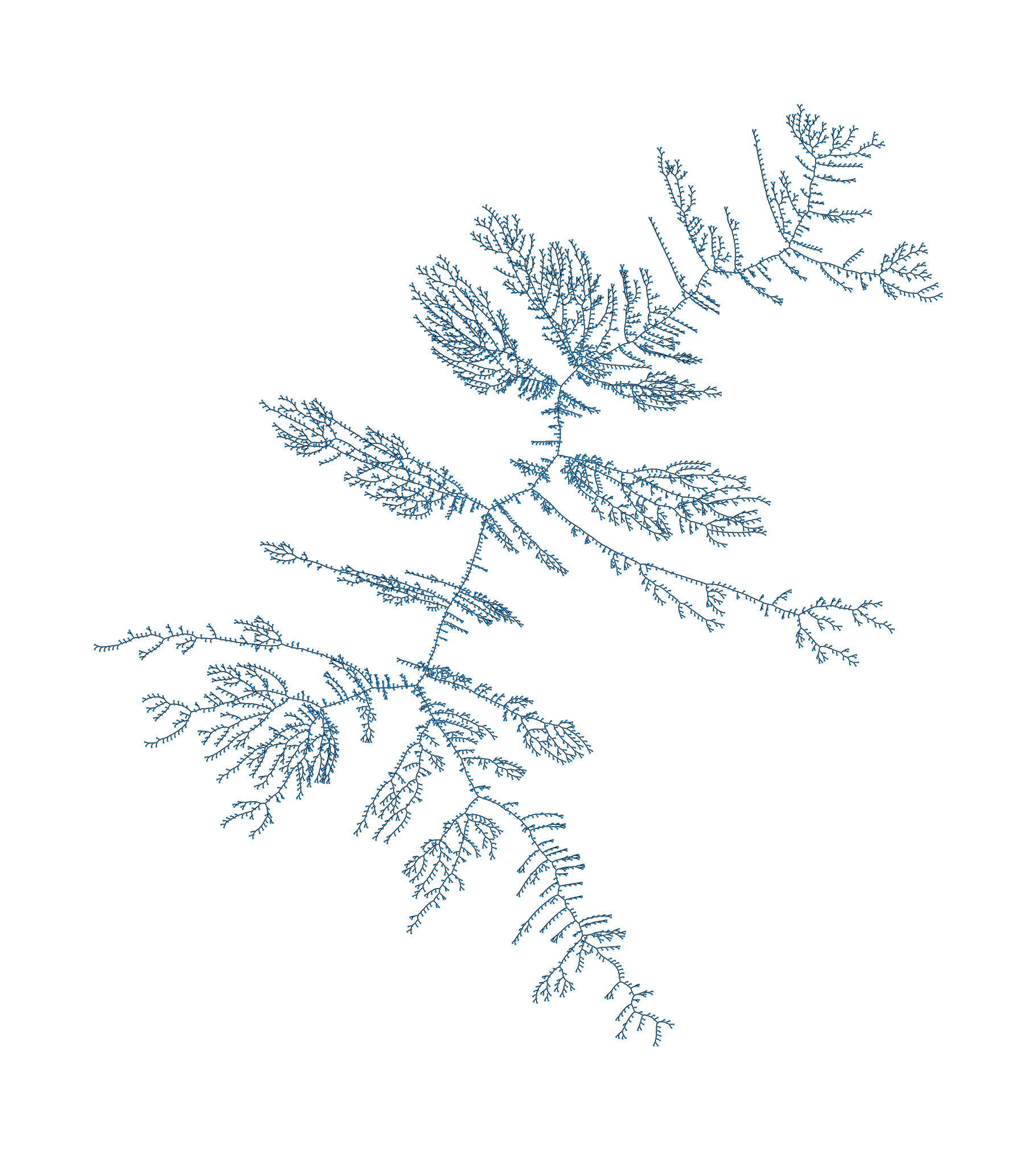}
    \caption{\centering Realisations of $\rot \Tree_{5000}$ for several values of $\alpha$: on the left $\alpha=1.01$, in the middle $\alpha=1.5$, on the right $\alpha=1.8$. }
    \label{fig:Simulation rotated stable trees}
   \end{center}
\end{figure}

Let us finally discuss some properties related to $\mathscr T_{\mathbbm x}$ \textit{seen as a limit}.  We introduced $\mathscr T_{\mathbbm x}$ as the distributional scaling limit of $\rot \Tree_n$, but it is also the scaling limit of another sequence of trees. Indeed, the rotation comes with a symmetric counterpart, the \textit{co-rotation}, denoted here by $\corot$ (see Figure~\ref{fig:DefRecCorotation} and compare it with Figure~\ref{fig:DefRecRotation}). It also maps a plane tree $T$ into a binary tree which is another spanning tree of $\Loop(T)$, and it appears that the scaling limit of $\corot \Tree_n$ is again $\mathscr T_{\mathbbm x}$ (see Corollary~\ref{cor:Comparison Rot Corot}). We have chosen to say a few words about the large-scale effects of $\corot$ on \textsc{Bienaymé} trees compared to the effects of $\rot$, because the encoding processes of $\corot \Tree_n$ surprisingly behave in a simpler way than those of $\rot \Tree_n$. This also provides interesting examples concerning the relation between the scaling limits of some trees and their \Luka walks: the \Luka walks of $\rot \Tree_n$ and $\corot \Tree_n$ have radically different scaling limits despite the fact that the trees themselves shared the same distributional scaling limit (see Theorem~\ref{thm:Detailed Result Corotation} and its comments). On the contrary, the \Luka walk of $\corot \Tree_n$ has the same distributional scaling limit as the \Luka walk of $\Tree_n$ even though these trees have different scaling limits.

\paragraph{Acknowledgement.} I wish to thank Bruno~\textsc{Schapira} and Igor~\textsc{Kortchemski} for their references, pieces of advice, and numerous rereadings that helped me throughout this work. They also provided guidance for organizing this article, and the very idea of studying the large-scale effects of the rotation was proposed by Igor. I am also indebted to Cyril~\textsc{Marzouk} for an enlightening discussion about the proof of Proposition~\ref{prop:LukaMirror}.
\bigskip

\paragraph{Outline.} First we present our notation for finite plane trees and their encoding processes in Section~\ref{sct:planeTrees}. In Section~\ref{sct:ScalingLimitTrees} we introduce all the usual notions needed to consider scaling limits of trees and measured trees, and then we briefly present the notion of parametric representations underlying \textsc{Skorokhod}'s $M_1$ topology in order to define and study measured $\R$-trees encoded by discontinuous excursions. We end this section by investigating the particular case of $\R$-trees encoded by an $\alpha$-stable excursion. The rest of the paper is dedicated to the study of the rotation correspondence: Section~\ref{sct:rotation} contains definitions of this correspondence and related objects, then in Section~\ref{sct:Encoding Processes} we give some existing results of \textsc{Duquesne} for \textsc{Bienaymé} trees before establishing the scaling limits of all encoding processes of rotated \textsc{Bienaymé} trees as well as those of co-rotated \textsc{Bienaymé} trees. This section contains both statements and proofs of these results, except four lemmas which compare some encoding processes with respect to the $M_1$ distance. We prove these lemmas in Section~\ref{sct:ProofLemmas}, which is dedicated to the use of combinatorial relations to control $M_1$ distances. Finally, Appendix~\ref{app:MirrorLuka} is devoted to the proof of a minor extension of \textsc{Duquesne}'s result needed here.



\section{Basics on finite plane trees}\label{sct:planeTrees}

\paragraph{\textsc{Ulam-Harris-Neveu} formalism.}
All trees here will be plane trees, and we will use standard notation for these trees (see e.g. \cite{Legall2005RandomTreesApplications}): plane trees will be subtrees of \textsc{Ulam}'s tree $\Ul= \bigcup_{n \in \N} (\N^*)^n$. Here, $(\N^*)^0=\{\varnothing\}$. $\varnothing$ is the root of $\Ul$, and for any $u=u_1\ldots u_n\in \Ul$ and $j \in \N^*$ we say that $uj=u_1\ldots u_nj$ is the $j$th child of $u$. We thus define parent$(uj)=u$, and we denote by $|u|=n$ the \textit{height} (or generation) of $u$ in the tree $\Ul$. The notion of parent is well-defined for every $u$ such that $|u|\geq 1$, as well as the notion of ancestors. 

\medskip

A finite plane tree $T$ is a finite subset of $\Ul$ such that:
\begin{itemize}
    \item[(a)] $\varnothing \in T$.
     \item[(b)] For $u\neq \varnothing$,  $u\in T$ implies parent$(u)\in T$.
     \item[(c)] For every $u \in T$, there is $d_u(T)\in \N^*$ such that for any $j\in \N^*$, $uj \in T$ if and only if $1 \leq j \leq d_u(T)$. 
\end{itemize}

We interpret $d_u(T)$ as the number of children of $u$ in $T$. When $d_u(T)=0$ we say that $u$ is a leaf of $T$, else it is an internal vertex of $T$. We equip every plane tree $T$ with the lexicographical order $\preccurlyeq$ (induced on it by the lexicographical order of $\Ul$), and the sequence $u_0=\varnothing,\ldots,u_{\#T-1}$ of the vertices of $T$ in lexicographical order will be called \textit{the lexicographical enumeration of $T$} (see Figure~\ref{fig:lexicoEnumeration}).

\begin{figure}[h]
    \begin{center}
\begin{tikzpicture}[scale=0.12]
\tikzstyle{every node}+=[inner sep=0pt]
\draw [black] (40.1,-46.5) circle (3);
\draw (40.1,-46.5) node {$\varnothing$};
\draw [black] (25.8,-32.4) circle (3);
\draw (25.8,-32.4) node {$1$};
\draw [black] (40.1,-32.4) circle (3);
\draw (40.1,-32.4) node {$2$};
\draw [black] (53.5,-32.4) circle (3);
\draw (53.5,-32.4) node {$3$};
\draw [black] (18.7,-16.7) circle (3);
\draw (18.7,-16.7) node {$11$};
\draw [black] (30.6,-16.7) circle (3);
\draw (30.6,-16.7) node {$12$};
\draw [black] (53.5,-16.7) circle (3);
\draw (53.5,-16.7) node {$31$};
\draw [black] (42.17,-44.33) -- (51.43,-34.57);
\fill [black] (51.43,-34.57) -- (50.52,-34.81) -- (51.24,-35.5);
\draw [black] (53.5,-29.4) -- (53.5,-19.7);
\fill [black] (53.5,-19.7) -- (53,-20.5) -- (54,-20.5);
\draw [black] (40.1,-43.5) -- (40.1,-35.4);
\fill [black] (40.1,-35.4) -- (39.6,-36.2) -- (40.6,-36.2);
\draw [black] (37.96,-44.39) -- (27.94,-34.51);
\fill [black] (27.94,-34.51) -- (28.15,-35.42) -- (28.86,-34.71);
\draw [black] (26.68,-29.53) -- (29.72,-19.57);
\fill [black] (29.72,-19.57) -- (29.01,-20.19) -- (29.97,-20.48);
\draw [black] (24.56,-29.67) -- (19.94,-19.43);
\fill [black] (19.94,-19.43) -- (19.81,-20.37) -- (20.72,-19.96);
\end{tikzpicture}
\end{center}
    
    \caption{\centering A plane tree (as a subtree of $\Ul$). Its lexicographical enumeration is $u_0=\varnothing,u_1=1,u_2=11,u_3=12, u_4=2, u_5=3,u_6=31$.}\label{fig:lexicoEnumeration}
\end{figure}
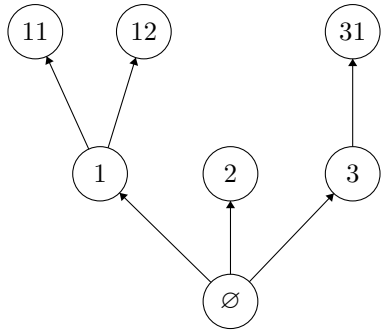

\medskip

\paragraph{Large critical \textsc{Bienaymé} trees.}
We focus on some specific probability measures on the set of finite plane trees. To define them, let $\mu$ be a probability measure on $\N$, called the offspring distribution, and assume that $\mu$ is \textit{critical}, which means it has mean $1$ but is not the \textsc{Dirac} mass $\delta_1$. Let $(k_u)_{u \in \Ul}$ be \iid~ random variables with distribution $\mu$, then almost surely there exists a unique finite plane tree $\Tree$ such that $\varnothing \in \Tree$ and $\forall u \in \Tree,\ d_u(\Tree)=k_u$. It corresponds to a family tree where each individual reproduces (on its own) independently of the rest and its offspring is distributed according to $\mu$. We denote by $\BGW$ its distribution, and any random tree with this distribution is called a \textsc{Bienaymé} tree with offspring distribution $\mu$. We will work with such trees but conditioned to have exactly $n$ vertices. Obviously, we implicitly restrict our attention to those integers $n$ such that $\BGW\evt{\text{\#vertices = $n$}}>0$. Note that there are infinitely many such $n$, hence we can choose $n$ arbitrarily large.



\medskip

\paragraph{Encoding processes.} 
A standard way to study plane trees is to use one-to-one correspondences with some integer-valued processes, called encoding processes in the following. Here we will mostly use those presented in \cite{Legall2005RandomTreesApplications}, namely the (lexicographic) height process, the contour process, and the (lexicographic) \Luka walk, and we will use them in two forms: discrete sequences and their associated time-scaled functions (see Figure~\ref{fig:encodingProcesses}). To introduce our notation, let $u_0=\varnothing,\ldots,u_{\#T-1}$ be the lexicographical enumeration of $T$.
\begin{itemize}
\item \textbf{The height process} $H_T=(H_T(k))_{0 \leq k \leq \#T}$ is defined by $H_T(k)=|u_k|$ for $k < \#T$ and $H_T(\#T)=0$ by convention.
\item \textbf{The contour process} $C_T=(C_T(k))_{0 \leq k \leq 2(\#T-1)}$ is informally constructed as follow: Imagine a particle living on the tree $T$ and initially located at the root. This particle can move in a discrete fashion, its elementary move simply is to go from its current vertex to one of its neighbours in $T$. Let $x_0=\varnothing, x_1,\ldots,x_{2(\#T-1)}=\varnothing$ be the sequence of vertices visited (in this order) by the particle when it goes straight from $u_0=\varnothing$ to $u_1$, then from $u_1$ to $u_2$,\ldots, then from $u_{\#T-1}$ to $\varnothing$ (\textit{i.e.} the particle follows the contour of $T$ from the left to the right). The contour process records the height of the particle at each step: $\forall 0\leq k\leq 2(\#T-1),\ C_T(k)=|x_k|$.
\item \textbf{The \Luka walk} $S_T=(S_T(k))_{0 \leq k \leq \#T}$ is defined by: 
\begin{equation*}
    S_T(0)=0 \text{ and } \forall k < \#T,\ S_T(k+1)-S_T(k)=d_{u_k}(T)-1.
\end{equation*}
One can see the \Luka walk as a record of the out-degrees (in lexicographical order), but there is another way to understand it. For any $u\in T$, we denote by $\Zintfo{\varnothing,u}$ the set of its ancestors (which does not contain $u$). Observe that for $k<\#T$
\begin{equation*}
\begin{split}
    S_T(k) &=\sum_{\ell < k} d_{u_{\ell}}(T) - k \\
    &= \#\{v \in T\!\setminus\!\{\varnothing\}\mid\text{parent}(v) \prec u_k\} - \#\{v \in T\!\setminus\!\{\varnothing\}\mid v \preccurlyeq u_k\} \\
    &= \#\{v \in T\!\setminus\!\{\varnothing\}\mid\text{parent}(v)\in\Zintfo{\varnothing,u_k} \text{ and } v \succ u_k\}.
\end{split}
\end{equation*}
Visually, $\Zintfo{\varnothing,u}$ forms a spine in $T$ and $\#\{v \in T\!\setminus\!\{\varnothing\}\mid\text{parent}(v)\in\Zintfo{\varnothing,u} \text{ and } v \succ u\}$ is the number of edges grafted on $\Zintfo{\varnothing,u}$ on its \textit{right} side. We call this quantity $R(u)$, which implicitly depends on $T$, and we also define its \textit{left} equivalent $L(u)=\#\{w \in T\!\setminus\!\{\varnothing\}\mid\text{parent}(w)\in\Zintfo{\varnothing,u}, w \not\in \Zintfo{\varnothing,u} \text{ and } w \prec u\}$, \ie the number of edges grafted on $\Zintfo{\varnothing,u}$ on its \textit{left} side. We thus get the useful alternative definition: 
\begin{equation*}
    \forall k < \#T\ S_T(k)=R(u_k), \text{ and } \ S_T(\#T)=-1.
\end{equation*}
\end{itemize}

\begin{figure}[h]
    \hspace{-1.5cm} \includegraphics[width=1.2\linewidth]{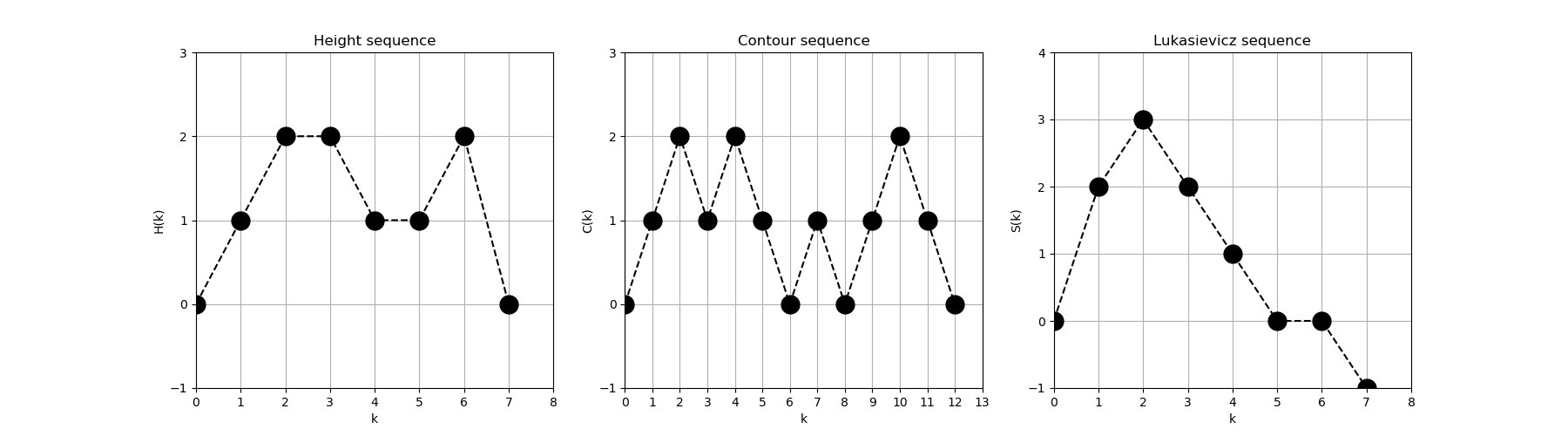}
    \caption{\centering Height, contour and \Luka processes of the tree depicted in Figure~\ref{fig:lexicoEnumeration}. Each of them fully characterizes this tree.}
    \label{fig:encodingProcesses}
\end{figure}

With any sequence $A=(A(k))_{0\leq k\leq p}$ we associate a continuous function $a\in C([0,1])$ that will be called its \textit{time-scaled function}: $a$ is the unique function affine on each segment $[k/p,(k+1)/p]$ and such that $\forall k\in \Zintff{0,p}\ a(k/p)=A(k)$. Let us stress that we only rescale \textit{time} while \textit{space} is unaffected. 

As we will make an extensive use of them, we make an explicit definition of the time-scaled height, contour, and \Luka functions of $T$ obtained with this procedure:
\begin{itemize}
\item \textbf{The time-scaled height function} is $h_T \in C([0,1])$ associated with $H_T$, hence 
\begin{equation*}
    \forall k\in \Zintff{0,\#T-1},\ h_T\!\left(\frac{k}{\#T}\right)=H_T(k)=|u_k| \text{ and } h_T(1)=H_T(\#T)=0.
\end{equation*}
\item \textbf{The time-scaled contour function} is $c_T \in C([0,1])$ associated with $C_T$, hence 
\begin{equation*}
    \forall k\in \Zintff{0,2(\#T-1)},\ c_T\!\left(\frac{k}{2(\#T-1)}\right)=C_T(k)=|x_k|.
\end{equation*}
\item \textbf{The time-scaled \Luka function} is $s_T \in C([0,1])$ associated with $S_T$, hence 
\begin{equation*}
    \forall k\in \Zintff{0,\#T-1},\ s_T\!\left(\frac{k}{\#T}\right)=S_T(k)=R(u_k) \text{ and } s_T(1)=S_T(\#T)=-1.
\end{equation*}
\end{itemize}
Note that we will also consider time-scaled functions associated with other sequences. 

\section{$\R$-trees and contour functions}\label{sct:ScalingLimitTrees}

In this section, we first explain the standard setting for studying scaling limits of random plane trees, in particular we emphasize how to understand trees, $\R$-trees and their convergence through their \textit{contour functions}. Then we introduce our extension and prove that a contour function remains a useful tool in this new setting, and finally we apply our results to study the $\R$-trees encoded by spectrally positive $\alpha$-stable excursions.

\subsection{Standard setting for scaling limits of trees}\label{ssct:CompactMetricSpace}

In the context of scaling limits, a plane tree $T$ is seen as a rooted compact metric space: we identify it with the finite set of its vertices equipped with the graph distance in $T$, and the root is a distinguished point of this space. Note that it is not a one-to-one identification as we forget the planar order of $T$. The point of this setting is that we can rescale trees, \ie we can consider $\lambda T$ obtained from $T$ by multiplying its metric by $\lambda$, in order to study the asymptotic geometry of some suitably rescaled trees by means of the \textsc{Gromov-Hausdorff} distance $\dGH$. Based on \cite{Legall2005RandomTreesApplications,Miermont2009Tesselations}, we will give a quick presentation of this distance $\dGH$, but also of the \textsc{Gromov-Hausdorff-Prokhorov} distance $\dGHP$. The interest of this distance on rooted compact metric spaces endowed with \textsc{Borel} probability measure is that it enables us to describe both the asymptotic geometry of some large finite trees and the asymptotic distribution of vertices picked uniformly at random in these trees. We will then introduce the metric spaces obtained as scaling limits of large finite trees, namely $\R$-trees, and explain the standard tools and techniques (such as \textit{contour functions}) used here to deal with both $\R$-trees and measured $\R$-trees.

\paragraph{The \textsc{Gromov-Hausdorff} and \textsc{Gromov-Hausdorff-Prokhorov} distances.}

At first, let us consider $K_1,K_2$ two compact subsets of some metric space $(E,d)$. One can compare them with the classical \textsc{Hausdorff} distance relative to  $(E,d)$: for $A \subset E$ and $r>0$, let $A^r$ be the $r$-neighbourhood of $A$, $\ie A^r=\{x \in E \st d(x,A)<r\}$. The distance between $K_1$ and $K_2$ is
\begin{equation*}
    d_{\text{haus}}(K_1,K_2)=\inf \! \left\{r >0 \st K_1 \subset K_2^r \text{ and } K_2 \subset K_1^r\right\}.
\end{equation*}
To compare two general compact metric spaces $(K_1,d_1)$ and $(K_2,d_2)$  with distinguished vertices $x_1 \in K_1$ and $x_2 \in K_2$, the idea behind the \textsc{Gromov-Hausdorff} distance simply is to embed them in a common metric space in order to make them as close as possible in the sense of the previous distance. Here we deal with rooted metric spaces, so we also take the distinguished point into account and the distance takes the following form:
\begin{equation}\label{eq:defGH}
    \dGH \bigl((K_1,d_1,x_1),(K_2,d_2,x_2)\bigr)= \inf_{\phi_1,\phi_2} \!\left\{ d_{\text{haus}}\bigl(\phi_1(K_1),\phi_2(K_2)\bigr)\vee d\bigl(\phi_1(x_1),\phi_2(x_2)\bigr)  \right\}
\end{equation}
where $(\phi_1,\phi_2)$ can be any pair of isometric embeddings of $K_1$ and $K_2$ into a common metric space and, for simplicity, $d$ always denotes the distance on this associated common space.

For practical use, we can rely on an alternative definition that does not require embedding $K_1$ and $K_2$. First, a \textit{correspondence} between $K_1$ and $K_2$ is a measurable subset $\mathcal C$ of $K_1\times K_2$ such that for every $y_1 \in K_1$ there is $y_2\in K_2$ with $(y_1,y_2)\in \mathcal C$ and conversely for every $y_2 \in K_2$ there is $y_1 \in K_1$ with $(y_1,y_2)\in \mathcal C$. The \textit{distortion} of the correspondence $\mathcal C$ is defined by
\begin{equation*}
    \dis(\mathcal C)=\sup_{(y_1,y_2),(z_1,z_2)\in \mathcal C} \abs{d_1(y_1,z_1)-d_2(y_2,z_2)}
\end{equation*}
As mentioned in \cite{Legall2005RandomTreesApplications}, the \textsc{gromov-hausdorff} distance $\dGH$ between $K_1$ and $K_2$ can be reformulated as half the infimum of distortion of correspondences that contain the pair of distinguished vertices $(x_1,x_2)$:
\begin{equation}\label{eq:GH via distortion}
    \dGH \bigl((K_1,d_1,x_1),(K_2,d_2,x_2)\bigr) = \frac{1}{2} \inf_{\mathcal C \ni (x_1,x_2)}\dis(\mathcal C).
\end{equation}

Now let us endow $K_1$ and $K_2$ with some \textsc{Borel} probability measures $\mu_1$ and $\mu_2$. The idea behind the \textsc{Gromov-Hausdorff-Prokhorov} distance $\dGHP$ is quite similar to the one that led to \eqref{eq:defGH} to define $\dGH$: one embeds $K_1$ and $K_2$ in a common metric space in order to make them as close as possible in the sense of the \textsc{Hausdorff} distance (and the distance between the distinguished vertices) but also in the sense of the \textsc{Lévy-prokhorov} distance between the \textsc{Borel} probability measures obtained by pushing forward $\mu_1$ and $\mu_2$. However, there is again an equivalent definition based on the notion of correspondence and more convenient to work with:
\begin{equation}\label{eq:GHP via distortion and coupling}
    \dGHP \bigl((K_1,d_1,x_1,\mu_1),(K_2,d_2,x_2,\mu_2)\bigr) =  \inf_{\mathcal C, \nu}\frac{1}{2}\dis(\mathcal C)\vee (1-\nu(\mathcal C))
\end{equation}
where the infimum is taken over all correspondences $\mathcal C$ between $K_1$ and $K_2$ containing the pair of distinguished vertices $(x_1,x_2)$ and all probability measures $\nu$ on $K_1\times K_2$ that form a coupling of $\mu_1$ and $\mu_2$. Note that $\dGH \bigl((K_1,d_1,x_1),(K_2,d_2,x_2)\bigr) \leq \dGHP \bigl((K_1,d_1,x_1,\mu_1),(K_2,d_2,x_2,\mu_2)\bigr)$, whatever $\mu_1$ and $\mu_2$ are.

In these settings, a map $f:K_1 \mapsto K_2$ is called an \textit{isometry} when it is one-to-one and preserves the distances and distinguished points in the case of rooted compact metric spaces, and additionally the measures when the spaces are endowed with measures. The notion of \textit{isometry} must not be confused with the notion of \textit{isometric map}, which denotes a map that preserves the distances (without any other requirement, in particular it does not have to be onto). We finally say that $K_1$ and $K_2$ are isometric when there is an isometry $f:K_1 \mapsto K_2$.

Actually, $\dGH$ and $\dGHP$ are pseudo-distances: two spaces are at distance $0$ if and only if they are isometric, and $\dGH$ (resp. $\dGHP$) truly defines a distance between \textit{isometry classes} of (resp. measured) rooted compact metric spaces. In both cases, this set of isometry classes equipped with the corresponding distance turns out to be complete and separable (see \cite{Miermont2009Tesselations,Petersen2016Geometry} for proofs).

\paragraph{$\R$-trees.}

The compact metric spaces obtained as scaling limits of large but finite plane trees always have some specific properties: they are geodesic metric spaces that are acyclic in a sense. A compact metric space with these properties is called an $\R$-tree (or a real tree). However we will not give any further details here on these metric properties, and instead of this we define $\R$-trees with their representations through contour functions. We refer to \cite{Legall2005RandomTreesApplications} for a more complete introduction.

Consider $f\in C_0([0,1],\R_+)$, \ie a positive continuous function such that $f(0)=f(1)=0$. By mimicking the way that the contour process of a plane tree can be used to compute distances in this tree, we define a pseudo-distance $d_f$ on $[0,1]$: The most recent common ancestor of two points $s,t$ would be a point $a\in {[s\wedge t, s \vee t]}$ such that $f(a)=\inf_{[s\wedge t, s \vee t]}f$ and thus we set
\begin{equation}\label{eq:pseudometric}
    d_f(s,t)=f(s)+f(t)-2\inf_{[s\wedge t, s \vee t]}f.
\end{equation}
It satisfies the triangular inequality, but $d_f(s,t)=0$ can occur even though $s\neq t$. However, once we define an equivalence relation by $s \sim_f t$ if and only if $d_f(s,t)=0$, we get that $d_f$ becomes a distance on the quotient set $[0,1]/\!\!\sim_f$. We let $\mathscr{T}_f$ denote the resulting metric space $\bigl([0,1]/\!\!\sim_f,d_f\bigr)$ and  $\pi_f$ denote the canonical projection from $[0,1]$ onto $\mathscr T_f$. This projection is continuous and onto, hence $\mathscr T_f$ is compact, and we consider this metric space as rooted at $\pi_f(0)$. 

The metric spaces called $\R$-trees are those which can be represented as $\mathscr T_g$ for some $g$:

\begin{defi}\label{def:RealTree}
A rooted compact metric space $K$ is an $\R$-tree if and only if it is isometric to $\mathscr{T}_f$ for some $f\in C_0([0,1],\R_+)$, which is then called a contour of this $\R$-tree.
\end{defi}

\begin{rmk}
    This definition is equivalent to \cite[Definition~2.1]{Legall2005RandomTreesApplications} since $\mathscr T_f$ always satisfies the requirements of \cite[Definition~2.1]{Legall2005RandomTreesApplications} and conversely all $\R$-trees (in the sense of \cite[Definition~2.1]{Legall2005RandomTreesApplications}) can be represented by a contour function (see \cite[Theorem~2.2]{Legall2005RandomTreesApplications} and \cite[Corollary~1.2]{DuquesnePreprintCoding}).
\end{rmk}

\begin{rmk}
In this definition, $f$ is far from unique. In particular we stress that any monotone and onto map $\phi:[0,1]\mapsto[0,1]$ induces an isometry $\tilde\phi: \mathscr{T}_{f\circ\phi} \mapsto \mathscr{T}_f$ such that $\tilde\phi\circ\pi_{f\circ\phi}=\pi_f\circ\phi$. We also mention that a dilation of $f$ implies a dilation of the corresponding tree, \ie $\forall \lambda>0, \mathscr{T}_{\lambda f}=\lambda\mathscr{T}_f$.
\end{rmk}

Finally, an $\R$-tree is said to be \textit{measured} when it is endowed with a \textsc{Borel} probability measure. A convenient way to define such a measure on an $\R$-tree is again to use a contour function: for $f\in C_0([0,1],\R_+)$, we let $\mu_f$ be the pushforward of $\leb$ by the projection $\pi_f:[0,1]\mapsto \mathscr T_f$, \ie
\begin{equation*}
     \mu_f \text{ is the law of } \pi_f(U) \text { where } U \text{  is uniformly distributed over }[0,1]. 
\end{equation*}

Note that for some contour functions $f,g$, we have that $\mathscr T_f$ and $\mathscr T_g$ are isometric while $(\mathscr T_f,\mu_f)$ and $(\mathscr T_g,\mu_g)$ are not.

A crucial example is given by $c_T$ the time-scaled contour function of some finite plane tree $T$. The connected metric space $\mathscr T_{c_T}$ cannot be the discrete metric space $T$ (except for $T=\{\varnothing\}$), but it is closely related to it: it is (isometric to) the metric space obtained from $T$ by considering edges as segments with unit length and vertices as the endpoints of these segments. To be precise, for each vertex, all the incident edges are unit segments that merge at their endpoint corresponding to this vertex. Moreover this space is rooted at the endpoint corresponding to $\varnothing$. The measure $\mu_{c_T}$ corresponds to the uniform measure on the union of the unit-length segments. Observe further that $T$ is isometric to the subset of endpoints of edges of $\mathscr T_{c_T}$ and that each $x \in \mathscr T_{c_T}$ belongs to an edge thus is at distance at most $1$ from its farthest away endpoint from the root.  Consequently,

\begin{equation}\label{eq:Tree approximation}
    \forall \lambda>0,\ \dGH(\lambda T, \mathscr T_{\lambda c_T}) \leq \dGHP \bigl((\lambda T,\Unif_T), (\mathscr T_{\lambda c_T},\mu_{\lambda c_T})\bigr) \leq \lambda \vee \frac{1}{\# T}.
\end{equation}
As we think of an $\R$-tree $\mathscr T$ as a tree, we say that $x\in \mathscr T$ is a vertex of $\mathscr T$ and its degree $\deg(x)$ is the number of connected components of $\mathscr T\!\setminus\!\{x\}$. A vertex $x$ is said to be a leaf when $\deg(x)=1$, and a branch-point when $\deg(x)\geq 3$.

\paragraph{Convergence through contour functions.} We have chosen to discuss $\R$-trees directly through their representation by contour functions since our goal is to extend this representation to càdlàg functions, but also because this is all we need here to study scaling limits of random trees. Indeed, convergence of contour functions directly implies convergence of trees for $\dGH$ and $\dGHP$. See \eg \cite{Legall2005RandomTreesApplications} for a proof of the fact that $f\mapsto \mathscr T_f$ is lipschitz for $\dGH$, and it is not difficult to adapt this proof to see that the same result holds for $\dGHP$: 

\begin{lem}[{\cite[Lemma~2.4]{Legall2005RandomTreesApplications}}]\label{lem:GHlipschitz}
    For $f,g\in C_0([0,1],\R_+)$, 
    \begin{equation*}
        \dGHP\bigr((\mathscr T_f,\mu_f),(\mathscr T_g,\mu_g)\bigl) \leq 2\norm{f-g}_{\infty}.
    \end{equation*}
\end{lem}

This gives a convenient way to get scaling limits for trees by means of functional scaling limits. As an illustration, let us briefly discuss the first historical example of scaling limit result for trees, namely \textsc{Aldous}' theorem (although \textsc{Aldous} stated and proved it in a different way in his original articles \cite{Aldous1991Continuum2,Aldous1993Continuum3}). Fix a critical offspring distribution $\mu$ with finite variance $\sigma^2$ and for all adequate $n\in \N^*$ let $\Tree_n$ be a random tree with law $\BGW\evt{\,\cdot\,\vert\text{\#vertices = $n$}}$. It has been shown that $\bigl(\frac{\sigma}{2\sqrt n}c_{\Tree_n}\bigr)_n$ converges in distribution (with respect to $\norm{\cdot}_{\infty}$) to a normalized Brownian excursion $\mathbbm e$, see \eg \cite{Legall2005RandomTreesApplications}. This entails that $\bigl(\frac{\sigma}{2\sqrt n}\mathscr T_{c_{\Tree_n}}\bigr)_n$ converges in distribution (with respect to $\dGH$) towards $\mathscr T_{\mathbbm e}$, which is called a Brownian Continuum Random Tree (or Brownian CRT). Moreover \eqref{eq:Tree approximation} yields
\begin{equation*}
    \dGH\left(\frac{\sigma}{2\sqrt n}\Tree_n,\frac{\sigma}{2\sqrt n}\mathscr T_{c_{\Tree_n}}\right)\leq \frac{\sigma}{2\sqrt n}\vee \frac{1}{n} \cv{n} 0.
\end{equation*}
Hence we actually have that $\bigl(\frac{\sigma}{2\sqrt n}\Tree_n\bigr)_n$ converges in distribution towards a Brownian CRT.  Moreover this argument works fine with $\dGHP$ too, so we also have the stronger convergence of these trees endowed with their uniform probability measure toward $(\mathscr T_{\mathbbm e},\mu_{\mathbbm e})$.

This illustration is important here for two reasons: 
\begin{itemize}
    \item This range of application of this method will be broadened by our extension of the coding of $\R$-trees by contour functions, and thus we will be able to apply this method to obtain the convergence of rotated trees (Theorem~\ref{thm:MainResult}).
    \item Our main tool to study the encoding processes of rotated trees will be a generalization of \textsc{Aldous}' theorem which establishes scaling limits of all encoding functions of $\Tree_n$ under weaker assumptions (see Section~\ref{ssct:KnownScalingLimits}).
\end{itemize}  

Finally, let us stress that the encoding processes of $\Tree_n$ not only encode the metric space $\Tree_n$ endowed with a measure but also contain additional structure such as the order on this tree, hence the convergence of these processes could be used to obtain reinforcement of the convergence of trees as measured metric spaces (see \eg \cite{DuquesnePreprintCoding} where a contour function $f$ also induces a linear order on the corresponding $\R$-tree).

\subsection{Skorokhod's $M_1$ topology and discontinuous contour functions}\label{ssct:M1DiscontinuousContour}

We now turn to an extension of the standard setting exposed above. We are going to explain how to define an $\R$-tree, possibly endowed with a measure, from a càdlàg contour function $x\in D_0([0,1],\R_+)$ (\ie a positive càdlàg function vanishing at $0$ and $1$). This extension is based on the notion of parametric representations underlying \textsc{Skorokhod}'s $M_1$ topology on $D([0,1])$, hence we will begin with this notion before discussing the encoding of $\R$-trees. Then we will end this section by deriving from this construction some properties of continuity with respect to this topology. As a consequence, we will get the convergence of trees as a direct corollary of the convergence of contour functions \textit{with respect to the $M_1$ topology}. This will be particularly useful to study the rotation correspondence in the second part of this paper, because some encoding processes of rotated trees cannot be handled with the classical \textsc{Skorokhod}'s $J_1$ topology on $D([0,1])$ while the $M_1$ topology provides an adequate framework to obtain their scaling limits (see the first comment in Section~\ref{ssct:comments} for more details).

Note that the presentation of  parametric representations and \textsc{Skorokhod}'s $M_1$ topology given in this section is mostly based on \cite[section~3.3, section~11.5.2, section~12.3]{Whitt2002StochasticProcessLimits}, which we refer to for more details.

\paragraph{Parametric representations of a discontinuous function.}   For $x\in D([0,1])$, we first introduce its \textit{completed graph} $\Gamma_x$:
\begin{equation*}
    \Gamma_x = \Bigl\{ \bigl(\lambda x(t-)+(1-\lambda) x(t),t\bigr)\in \R\times[0,1], \text{ for } (t,\lambda) \in [0,1]^2 \Bigr\}.
\end{equation*}
In words, $\Gamma_x$ is the graph of $x$ completed with vertical segments to fill in the discontinuities of $x$ so that it becomes close and connected. We endow it with a linear order based on how to draw $\Gamma_x$ from left to right in a continuous way: $(z_1,t_1) \leq (z_2,t_2)$ if either $t_1 < t_2$ or $t_1=t_2$ and $\abs{z_1-x(t_1-)}\leq \abs{z_2-x(t_2-)}$.

\begin{defi}\label{def:ParametricRepresentation}
A \textit{parametric representation} of $x$ is a non-decreasing (in the sense of the order on $\Gamma_x$) continuous and onto function $(\chi,\tau):[0,1]\rightarrow \Gamma_x \subset \R\times[0,1]$. Its first component $\chi:[0,1]\rightarrow \R$ is its spatial component while $\tau:[0,1]\rightarrow[0,1]$ is its temporal component.
\end{defi}

\begin{figure}[h]
    \centering
    \includegraphics[width=0.75\linewidth]{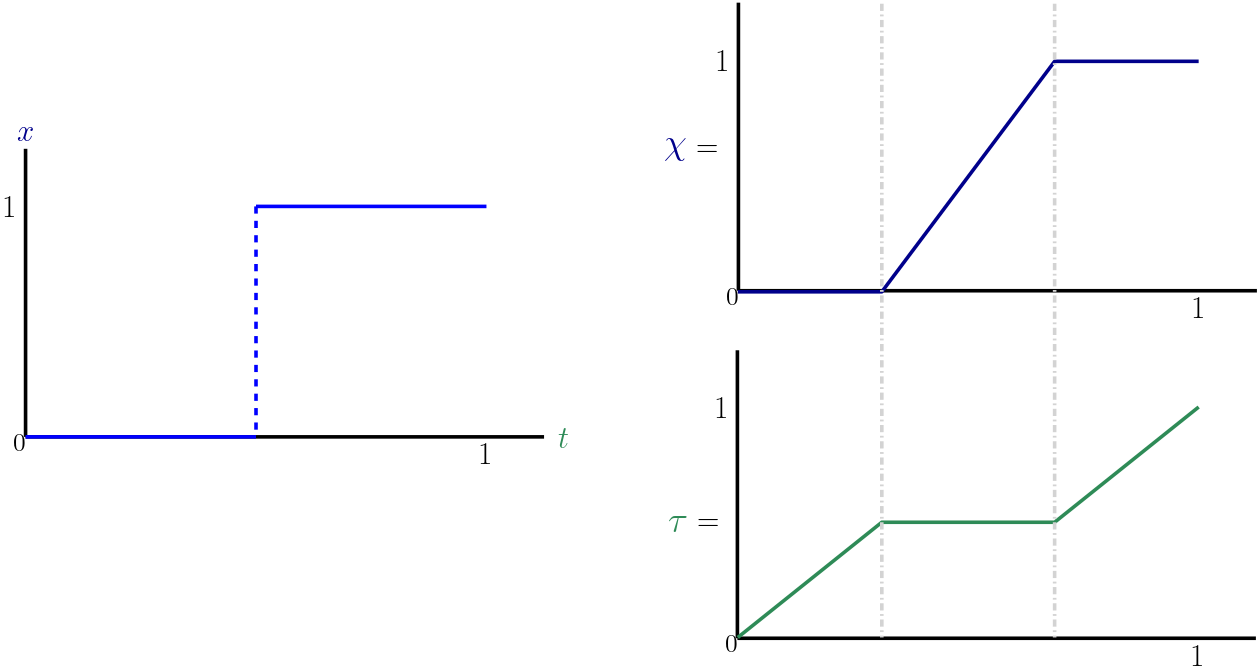}
    \caption{\centering The completed graph of $x=\1_{[1/2,1]}$ on the left, and one of its parametric representations on the right.}
    \label{fig:Parametric representation}
\end{figure}


Given a parametric representation $(\chi,\tau)$ of $x$, one can see that its temporal component $\tau$ is continuous non-decreasing (for the usual order on $[0,1]$) and  $x=\chi\circ\tau^{-1}$ with $\tau^{-1}$ the right-continuous inverse of $\tau$. The main technical difficulty with parametric representations is to prove their existence. \cite[section~12.3.3]{Whitt2002StochasticProcessLimits} indicates a way to construct a parametric representation (with the additional property of being one-to-one from $[0,1]$ to $\Gamma_x$) for any $x\in D([0,1])$, but a requirement is missing there to guarantee that this construction is well-defined. To explain this more precisely, let us assume the existence of a one-to-one parametric representation $(\chi,\tau)$ of $x$ and describe its properties: 
\begin{itemize}
    \item At extremal points we have $\bigl(\chi(0),\tau(0)\bigr)=\bigl(x(0),0\bigr)$ and $\bigl(\chi(1),\tau(1)\bigr)=\bigl(x(1),1\bigr)$.
    \item At $u\in(0,1)$, if we have $\tau(u)\not\in \disc(x)=\{t\in[0,1] \st x(t-)\neq x(t)\}$ then we must have $\chi(u)=x(\tau(u))$, hence as $(\chi,\tau)$ is one-to-one we see that $\tau$ must be strictly increasing at $u$: for all sufficiently small $\varepsilon>0,\ \tau(u-\varepsilon) < \tau(u) < \tau(u+\varepsilon)$.
    \item Conversely, for all $s\in \disc(x)$ there is a segment $[a_s,b_s] \subset [0,1]$ with $a_s < b_s$ such that $\tau(u)=s$ if and only if $u\in [a_s,b_s]$, and $\chi$ induces a continuous one-to-one function $[a_s,b_s]\rightarrow \bigl\{ \lambda x(s-)+(1-\lambda) x(s),\text{ for } \lambda \in [0,1] \bigr\}$ with $\chi(a_s)=x(s-)$ and $\chi(b_s)=x(s)$.  
\end{itemize}
  
    The construction proposed by \cite[section~12.3.3]{Whitt2002StochasticProcessLimits} is based on a collection of disjoint segments $[a_s,b_s]$ for $s \in \disc(x)$ that preserves order \ie for $s<s'\in \disc(x),\ b_s< a_{s'}$. From this, one would define a parametric representation on each of these intervals and in between so that it matches with the above description. However, one needs an additional property to ensure that this construction is well-defined and one-to-one at accumulation points of discontinuities: we must also require that for $s<s'\in \disc(x)$, there exists an open interval $I\subset(s,s')$ with $I\cap\disc(x)=\emptyset$ if and only if there exists an open interval $J\subset (b_s,a_{s'})$ with $J\cap[a_t,b_t]=\emptyset$ for all $t \in \disc(x)$. To avoid a detailed discussion about this requirement, we provide an alternative proof of existence in the next Lemma~\ref{lem:ParametricRepresentation}. We also prove that all parametric representations of $x$ are the same up to changes of time (which may not be one-to-one).

\begin{lem}\label{lem:ParametricRepresentation}
Let $\PR(x)$ be the set of all parametric representations of $x \in D([0,1])$. There exists $(\chi,\tau) \in\PR(x)$ that is one-to-one, and for all such one-to-one $(\chi,\tau) \in \PR(x)$ we have
    \begin{equation*}
        \PR(x)=\{(\chi\circ\phi,\tau\circ\phi), \text{ for } \phi: [0,1] \mapsto [0,1] \text{ non-decreasing and onto}\}.
    \end{equation*}
\end{lem}
\begin{proof}
Let us define a one-to-one parametric representation of $x$ based on the above description. A trick is to consider a probability measure on $[0,1]$ of the form $\leb/2+\sum_{s\in \disc(x)} m_s\delta_s$ (we just have to choose a family of strictly positive reals $(m_s)_{s \in \disc(x)}$ with sum $1/2$). Its cumulative distribution function $F$ is strictly increasing, with $F(0)=0, F(1)=1$ and $\disc(F)=\disc(x)$. Now we define $\tau:[0,1]\rightarrow[0,1]$ as its right-continuous inverse, which is non-decreasing but also continuous (because $F$ is strictly increasing) hence it is onto. For all $s\in \disc(x)$ there is a segment $[a_s,b_s] \subset [0,1]$ with $a_s=F(s-) < b_s=F(s)$ such that $\tau(u)=s$ if and only if $u\in [a_s,b_s]$, while for all $t \in [0,1]\!\setminus\!\disc(x)$ we have $\tau(u)=t$ if and only if $u=F(t)$. Next, we define $\chi$ by setting $\chi(u)=\frac{b_s-u}{b_s-a_s}x(s-)+\frac{u-a_s}{b_s-a_s}x(s)$ if $u \in [a_s,b_s]$ for some $s\in \disc(x)$ and $\chi(u)=x(\tau(u))$ otherwise. By construction, $(\chi,\tau)$ is a strictly increasing map from $[0,1]$ onto $\Gamma_x$. In addition, $\chi$ is continuous on any $]a_s,b_s[$ by definition, but also at $a_s$ and $b_s$ and at any $u$ such that $\tau(u)\not\in\disc(x)$ thanks to the regularity of $x$. As a consequence, $(\chi,\tau)$ is a one-to-one parametric representation of $x$.

We now prove the second assertion. Consider $(\chi,\tau)\in \PR(x)$ which is one-to-one, then for any non-decreasing onto (hence continuous) map $\phi:[0,1]\rightarrow[0,1]$, it is clear that $(\chi\circ\phi,\tau\circ\phi)\in\PR(x)$. Conversely, for any $(\chi',\tau')\in \PR(x)$, $\phi=(\chi,\tau)^{-1}\circ(\chi',\tau'):[0,1]\rightarrow[0,1]$ is onto and non-decreasing, and $(\chi',\tau')=(\chi\circ\phi,\tau\circ\phi)$.
\end{proof}

\paragraph{Measured $\R$-tree encoded by a discontinuous contour function.} To motivate the use of parametric representations, we first discuss why the ideas that led to the definition of $\mathscr T_f$ for a continuous excursion $f$  (in Section~\ref{ssct:CompactMetricSpace}) cannot be directly extended to a  càdlàg excursion $x$. 

For a given $x\in D_0([0,1],\R_+)$, it still makes sense to define the pseudo-metric induced by $x$ on $[0,1]$ as in \eqref{eq:pseudometric} and thus we can still consider the metric space $ \bigl([0,1]/\!\!\sim_x,d_x\bigr)$ defined as in Section~\ref{ssct:CompactMetricSpace}. However, the canonical projection $\pi_x:[0,1] \mapsto [0,1]/\!\!\sim_x$ is continuous if and only if $x$ is continuous, hence there is no guarantee that $([0,1]/\!\!\sim_x,d_x)$ is connected and compact. These properties are actually satisfied for some càdlàg excursions but not for all of them, for instance with $x=\1_{[1/3,2/3)}$ we simply get a discrete space consisting of two points. This means that $([0,1]/\!\!\sim_x,d_x)$ is not necessarily an $\R$-tree, and we will call it the \textit{quotient space associated with $x$}. In some cases this quotient space is actually an $\R$-tree, but proving it requires some additional work.

Parametric representations provide a simpler and more general way to define the $\R$-tree encoded by a càdlàg excursion: for any parametric representations $(\chi_1,\tau_1)$ and $(\chi_2,\tau_2)$ of $x\in D_0([0,1],\R_+)$, according to Lemma~\ref{lem:ParametricRepresentation} and the remark after Definition~\ref{def:RealTree} we have that  $\mathscr T_{\chi_1}$ is isometric to $\mathscr T_{\chi_2}$, hence we may define $\mathscr T_x$ as this $\R$-tree (seen up to isometry) and this definition is consistent with the case of a continuous function.

\begin{defi}\label{def:Discontinuous Contour and Real Tree}
For all  $x\in D_0([0,1],\R_+)$, the $\R$-tree associated with $x$ is (the isometry class of) $\mathscr T_{\chi}$, where $\chi$ is the spatial component of any parametric representation $(\chi,\tau)$ of $x$.
\end{defi}

Before endowing this tree with a measure, we compare this construction with the quotient space defined above. It appears that \textit{the quotient space $([0,1]/\!\!\sim_x,d_x)$ can always be seen as a subspace of the $\R$-tree $\mathscr T_x$}. Indeed, consider a parametric representation $(\chi,\tau)$ of $x$, then we have 
 \begin{equation*}
    \forall s,t \in [0,1]\ d_{\chi}(\tau^{-1}(s),\tau^{-1}(t))=d_x(s,t).
 \end{equation*}
As a consequence $\tau^{-1}$ induces an isometric map $\tau^*: \bigl([0,1]/\!\!\sim_x,d_x\bigr) \mapsto \bigl(\mathscr T_{\chi},d_{\chi}\bigr)$ which may or may not be onto but satisfies $\tau^*\circ\pi_x=\pi_{\chi}\circ\tau^{-1}$. Since $\mathscr T_x$ is defined as the isometry class of $\bigl(\mathscr T_{\chi},d_{\chi}\bigr)$, up to some identifications we may write that $ \bigl([0,1]/\!\!\sim_x,d_x\bigr) \subset \mathscr T_x$, with equality if and only if $\tau^*$ is onto. Thanks to this, the canonical projection $\pi_x$ on the quotient space may always be considered as taking value in the larger space $\mathscr T_x$. There are several sufficient conditions that ensure equality, for instance we have equality when $\bigl([x(s-)\wedge x(s),x(s-)\vee x(s)]\bigr)_{s \in \disc(x)}$ are pairwise disjoint segments, but in this paper we will only make use of the following criterion:

\begin{lem}\label{lem:Criterion quotient space}
    If $x\in D_0([0,1],\R_+)$ does not have any negative jump, then $ \bigl([0,1]/\!\!\sim_x,d_x\bigr)$ is (a representative of the isometry class of) the $\R$-tree $\mathscr T_x$.
\end{lem}

\begin{proof}
    Consider any parametric representation $(\chi,\tau)$ of $x$. We prove that for all $u \in [0,1]$ there is $s\in [0,1]$ such that $\pi_{\chi}(u) = \pi_{\chi}\bigl(\tau^{-1}(s)\bigr)$. As $\pi_{\chi}$ is onto and $\tau^*\circ\pi_x=\pi_{\chi}\circ\tau^{-1}$, it implies that $\tau^*$ is onto and the desired result follows.  
    
    For all $u \in [0,1]$ either $\tau(u) \not\in \disc(x)$ in which case $\pi_{\chi}(u) = \pi_{\chi}\bigl(\tau^{-1}(\tau(u))\bigr)$, or $\tau(u)\in \disc(x)$. In this latter case consider $s=\inf\{t>\tau(u): x(t)< \chi(u)\}$, with $\inf \emptyset = 1$. If $s>\tau(u)$ then $x(s-) \geq \chi(u)$ and $x(s)\leq \chi(u)$ so by positivity of the jumps $x(s-)=x(s)=\chi(u)$ and $\pi_{\chi}(u) = \pi_{\chi}\bigl(\tau^{-1}(s)\bigr)$, else if $s=\tau(u)$ we have that $u \leq \tau^{-1}(s)$ and by positivity of the jumps $\chi$ is non-decreasing on $[u,\tau^{-1}(s)]$ but at the same time $\chi(\tau^{-1}(s))=x(s)\leq \chi(u)$ so $\pi_{\chi}(u) = \pi_{\chi}\bigl(\tau^{-1}(s)\bigr)$ again.
\end{proof}

Let us now add a measure $\mu_x$ on the $\R$-tree $\mathscr T_x$. As for $\mathscr T_x$, the idea is to use a parametric representation to define a measured $\R$-tree which does not depend on our choice of parametric representation. We fix $x\in D_0([0,1],\R_+)$ and $(\chi,\tau),(\chi',\tau') \in \PR(x)$. We do not have in general that $(\mathscr T_{\chi},\mu_{\chi})$ is isometric to $(\mathscr T_{\chi'},\mu_{\chi'})$, so instead of $\mu_{\chi}$ we consider the measure $\mu_{(\chi,\tau)}$ on $\mathscr T_{\chi}$ defined as the distribution of $\pi_{\chi}\bigl(\tau^{-1}(U)\bigr)$ where $U$ is uniform over $[0,1]$. Now we have that $(\mathscr T_{\chi},\mu_{(\chi,\tau)})$ is isometric to $(\mathscr T_{\chi'},\mu_{(\chi',\tau')})$. Indeed we may assume that $(\chi,\tau)$ is a one-to-one parametric representation of $x$, so by Lemma~\ref{lem:ParametricRepresentation} there exists a non-decreasing onto map $\phi:[0,1]\mapsto [0,1]$ such that $(\chi',\tau')=(\chi,\tau)\circ\phi$ and we claim that the induced isometry $\tilde \phi:\mathscr T_{\chi'}\mapsto \mathscr T_{\chi} $ is actually an isometry between $(\mathscr T_{\chi'},\mu_{(\chi',\tau')})$ and $(\mathscr T_{\chi},\mu_{(\chi,\tau)})$: by construction $\tilde \phi \circ \pi_{\chi'}=\pi_{\chi}\circ \phi$ and $\phi\circ(\tau')^{-1}=\tau^{-1}$, so we have $\tilde \phi \circ \pi_{\chi'}\circ(\tau')^{-1}=\pi_{\chi}\circ\tau^{-1}$ and the push-forward of $\mu_{(\chi',\tau')}$ by $\tilde \phi$ is $\mu_{(\chi,\tau)}$. This leads to the following definition:

\begin{defi}\label{def:Discontinuous Contour and Measured Real Tree}
  For all  $x\in D_0([0,1],\R_+)$, the measured $\R$-tree $\bigl(\mathscr T_x,\mu_x\bigr)$ associated with $x$ is the isometry class of $\bigl(\mathscr T_{\chi},\mu_{(\chi,\tau)}\bigr)$, where $(\chi,\tau)$ is any parametric representation of $x$ and $\mu_{(\chi,\tau)}$ is the law of $\pi_{\chi}\bigl(\tau^{-1}(U)\bigr)$ with $U$ uniform over $[0,1]$.
\end{defi}

\begin{rmk}
    As every continuous function $f$ can be represented by $(f,\id_{[0,1]})$, this definition agrees with the previous construction of  $(\mathscr T_f,\mu_f)$ where  $\mu_f$ is the law of $\pi_f(U)$  (with $U$ uniform over $[0,1]$). Moreover, for all $x\in D_0([0,1],\R_+)$ we can also use the quotient space $\bigl([0,1]/\!\!\sim_x,d_x\bigr)$ and its identification as a subspace of $\mathscr T_x$ to understand the measure $\mu_x$ as the law of $\pi_x(U)$. Indeed recall that for every parametric representation $(\chi,\tau)$ of $x$ we have that $\tau^{-1}$ induces an isometric map $\tau^*: \bigl([0,1]/\!\!\sim_x,d_x\bigr) \mapsto \bigl(\mathscr T_{\chi},d_{\chi}\bigr)$ which may or may not be onto but satisfies $\tau^*\circ\pi_x=\pi_{\chi}\circ\tau^{-1}$. Up to the identifications $\tau^*=$ inclusion and $\bigl(\mathscr T_{\chi},\mu_{(\chi,\tau)}\bigr)=\bigl(\mathscr T_x,\mu_x\bigr)$, we thus have that  $\mu_x$ is the law of $\pi_x(U)$ where $U$ is uniform over $[0,1]$. As a consequence of this and the right-continuity of $\pi_x$, we also get that
    \begin{equation}\label{eq:support}
        \supp(\mu_x) = \overline{\bigl([0,1]/\!\!\sim_x,d_x\bigr)}^{\mathscr T_x}.
    \end{equation}
\end{rmk}

\smallskip

We end this part with a simple example of a measured $\R$-tree encoded by a càdlàg excursion illustrating the previous definitions. Interestingly, the measured $\R$-tree in this example has a measure supported on a strict closed subset of the tree and thus cannot be encoded by a continuous excursion.

Consider $T$ a finite plane tree and $c'_T$ its \textit{discontinuous} time-scaled contour function, given by
\begin{equation*}
    \forall t \in [0,1],\ c'_T(t) = C_T(\lfloor 2(\#T-1)t\rfloor).
\end{equation*}
It appears that $c'_T$ and the time-scaled contour function $c_T$  have some parametric representations with the same spatial component (we detail this in Section~\ref{ssct:framework}, see Figures~\ref{fig:Continuous and discontinuous interpolations} and~\ref{fig:Parametric representation of interpolations}), hence $\mathscr T_{c'_T}=\mathscr T_{c_T}$ \ie it is the $\R$-tree obtained from $T$ by considering edges as segments with unit length and vertices as the endpoints of these segments. But the subspace $\bigl([0,1]/\!\!\sim_{c'_T},d_{c'_T}\bigr)$ corresponds to the set of these endpoints, which means that it really is isometric to $T$ with the graph distance and that $\mu_{c'_T}$ can be seen as a measure on $T$. Since for each vertex in $T$ the contour process comes back to this vertex a number of times proportional to its number of incident edges, it appears that $\mu_{c'_T}$ is not uniform over $T$ but is the degree-biased distribution over $T$ (with the full degree, not the out-degree used in the definition of plane trees).

\paragraph{Continuity with respect to \textsc{Skorokhod}'s $M_1$ topology.}  Our main interest for the construction based on parametric representations is not its generality (in particular the only discontinuous contour functions we will use in this paper are $\alpha$-stable excursions with no negative jump hence according to Lemma~\ref{lem:Criterion quotient space} the corresponding $\R$-trees may be defined as quotient spaces), it is its compatibility with \textsc{Skorokhod}'s $M_1$ topology, as it really broadens the ways to get convergence of trees through convergence of contour functions. 

To understand this compatibility, one simply has to keep in mind a simple definition of \textsc{Skorokhod}'s $M_1$ topology. This topology on $D([0,1])$ is generated by the distance $d_{M_1}$ defined as follows: for any $x_1,x_2 \in D([0,1])$,
\begin{equation*}\label{eq:distance M1}
    d_{M_1}(x_1,x_2)=\inf_{(\chi_i,\tau_i)\in\PR(x_i)} \norm{\chi_1-\chi_2}_{\infty}\vee\norm{\tau_1-\tau_2}_{\infty}.
\end{equation*}
The underlying idea is to control the convergence of \textit{completed} graphs in a \textit{functional} sense. We refer to \cite[section~3.3, section~11.5.2, section~12.3]{Whitt2002StochasticProcessLimits} for more details on this distance, here we simply stress that although $d_{M_1}$ is not complete, it defines a Polish topology (see \cite[section~12.8]{Whitt2002StochasticProcessLimits}) hence it gives a standard framework for limit theorems. 
 
From this simple definition and the properties of $\R$-trees exposed above, we immediately deduce Proposition~\ref{prop:GHlipschitzM1} which gives the continuity of $x \in D_0([0,1],\R_+) \mapsto \mathscr T_x$ for $d_{M_1}$ and $\dGH$.
\begin{proof}[Proof of Proposition~\ref{prop:GHlipschitzM1}.]
    By Definition~\ref{def:Discontinuous Contour and Real Tree} and Lemma~\ref{lem:GHlipschitz} we have that
\begin{equation*}
     \dGH(\mathscr T_x,\mathscr T_y) \leq 2\inf_{(\chi_i,\tau_i)\in\PR(x_i)} \norm{\chi_1-\chi_2}_{\infty}\leq 2d_{M_1}(x,y). \vspace{-0.5cm}
\end{equation*}
\end{proof}

In the broader setting of measured $\R$-trees, it seems that the \textsc{Lipschitz} property does not hold anymore but, as stated in Proposition~\ref{prop:GHP M1}, we still have the continuity of $x \mapsto \bigl(\mathscr T_x,\mu_x\bigr)$. Our proof relies on the definitions of $\dGHP$ (Equation~\eqref{eq:GHP via distortion and coupling}) and of $\bigl(\mathscr T_x,\mu_x\bigr)$ (Definition~\ref{def:Discontinuous Contour and Measured Real Tree}) and also requires a slightly different characterization of the convergence with respect to $d_{M_1}$: according to \cite[Theorem~12.5.1]{Whitt2002StochasticProcessLimits}, a sequence $(x_n)_n$ converges to $x$ with respect to $d_{M_1}$ if and only if for all $(\chi,\tau) \in \PR(x)$ we have $(\chi_n,\tau_n)\in \PR(x_n)$ such that $(\chi_n,\tau_n) \rightarrow (\chi,\tau)$ with respect to ${\norm{\cdot}_{\infty}}$ as $n\shortrightarrow +\infty$.

\begin{proof}[Proof of Proposition~\ref{prop:GHP M1}.]
The continuity is a consequence of the following fact: for any $x_1,x_2 \in D_0([0,1],\R_+)$ and any $(\chi_1,\tau_1)\in\PR(x_1),(\chi_2,\tau_2)\in\PR(x_2)$ we have
    \begin{equation}\label{eq:Almost lipschitz}
        \dGHP\bigr( (\mathscr T_{x_1},\mu_{x_1}), (\mathscr T_{x_2},\mu_{x_2})\bigl) \leq \inf_{\delta>0} 2\bigl(\norm{\chi_1-\chi_2}_{\infty} + \omega(\chi_1;\delta)\bigr)\vee \frac{\norm{\tau_1-\tau_2}_{\infty}}{\delta},
    \end{equation}
    where $\omega(f;\delta)=\sup_{\abs{x-y}<\delta}\abs{f(x)-f(y)}$ is the modulus of continuity of a function $f$. Indeed, consider a sequence $(x_n)_n$ converging to $x$ with respect to $d_{M_1}$. By \cite[Theorem~12.5.1]{Whitt2002StochasticProcessLimits}, we have some parametric representations $(\chi_n,\tau_n)$ of $x_n$ and $(\chi,\tau)$ of $x$ such that $(\chi_n,\tau_n) \rightarrow (\chi,\tau)$ with respect to ${\norm{\cdot}_{\infty}}$. Set $\delta_n = \sqrt{\norm{\tau-\tau_n}_{\infty}}$ and apply \eqref{eq:Almost lipschitz} to get 
\begin{equation*}
     \dGHP\bigr( (\mathscr T_{x},\mu_{x}), (\mathscr T_{x_n},\mu_{x_n})\bigl) \leq 2\bigl(\norm{\chi-\chi_n}_{\infty} + \omega(\chi;\delta_n)\bigr)\vee \sqrt{\norm{\tau-\tau_n}_{\infty}}\cv{n} 0.
\end{equation*}

    Thus it only remains to prove \eqref{eq:Almost lipschitz}. We fix $\delta>0$ and define a correspondence between $\mathscr T_{\chi_1}$ and $\mathscr T_{\chi_2}$ by setting
    \begin{equation*}
        C_{\delta}=\{(\pi_{\chi_1}(u),\pi_{\chi_2}(v)):\ u,v \in [0,1] \st \abs{u-v}\leq \delta\}.
    \end{equation*}
    For $\bigl(\pi_{\chi_1}(u),\pi_{\chi_2}(v)\bigr),\bigl(\pi_{\chi_1}(u'),\pi_{\chi_2}(v')\bigr)\in C_{\delta}$, we have $ \abs{d_{\chi_1}(v,v')-d_{\chi_2}(v,v')} \leq 4\norm{\chi_1-\chi_2}_{\infty}$, and we also have $d_{\text{haus}}([u,u'],[v,v']) \leq \delta$ which implies  $\abs{d_{\chi_1}(u,u')-d_{\chi_1}(v,v')} \leq 4\omega(\chi_1;\delta)$. By summing these two parts we get
    \begin{equation}\label{eq:distortion Cdelta}
        \dis(C_{\delta})\leq 4\bigl(\norm{\chi_1-\chi_2}_{\infty} + \omega(\chi_1;\delta)\bigr).
    \end{equation}
    
    Next, let $\nu$ be the simple coupling between $\mu_{(\chi_1,\tau_1)}$ and $\mu_{(\chi_2,\tau_2)}$ obtained as the law of the couple $(V_1,V_2)=\bigl(\pi_{\chi_1}(\tau_1^{-1}(U)),\pi_{\chi_2}(\tau_2^{-1}(U))\bigr)$ where  $U$ is uniform over $[0,1]$. By construction we have
    \begin{equation*}
        \evt{(V_1,V_2) \not\in C_{\delta}} = \evt{U \in B_{\delta}} 
    \end{equation*}
    where $B_{\delta}$ is the set of \textit{bad points} $\{ t \in [0,1] \st \abs{\tau_1^{-1}(t)-\tau_2^{-1}(t)}>\delta\}$. We now control $\leb(B_{\delta})$ with the area between $\tau_1$ and $\tau_2$. First observe that by general properties of the inverse of an increasing function, we have 
    \begin{equation*}
    \{(u,t)\st \tau_1(u) \leq t \leq \tau_2(u)\}=\{(u,t)\st \tau_2^{-1}(t-) \leq u \leq \tau_1^{-1}(t)\}.
\end{equation*}
    As the left-continuous and right-continuous inverses are equal except on a countable set, it implies that $\int_0^1 \bigl(\tau_1^{-1}(t)-\tau_2^{-1}(t))_+\d t = \int_0^1 \bigl(\tau_1(u)-\tau_2(u)\bigr)_-\d u$, and this also holds when one exchanges $\tau_1$ and $\tau_2$. We thus have two ways to compute the area between $\tau_1$ and $\tau_2$ which provide
    \begin{align*}
       1-\nu(C_{\delta}) =\leb(B_{\delta}) \leq \frac{1}{\delta}\int_0^1 \abs{\tau_1^{-1}(t)-\tau_2^{-1}(t)}\d t = \frac{1}{\delta}\int_0^1 \abs{\tau_1(u)-\tau_2(u)}\d u \leq \frac{\norm{\tau_1-\tau_2}_{\infty}}{\delta}.
    \end{align*}
Combining this with \eqref{eq:distortion Cdelta} and \eqref{eq:GHP via distortion and coupling} and the fact that $\delta$ has been chosen arbitrarily proves \eqref{eq:Almost lipschitz}.

\end{proof}

\begin{rmk}
    It will be useful later to know some relations between \textsc{Skorokhod}'s $M_1$ topology and more standard topologies:
    \begin{itemize}
    \item $x_n \xrightarrow[\tend{n}]{J_1} x$ implies  $x_n \xrightarrow[\tend{n}]{M_1} x$ ;
    \item if $x_n \xrightarrow[\tend{n}]{M_1} x$ with $x$ continuous then actually $x_n \xrightarrow[\tend{n}]{\norm{\cdot}_{\infty}} x$.
\end{itemize}
In particular, Propositions~\ref{prop:GHlipschitzM1} and~\ref{prop:GHP M1} give also the continuity with respect to the usual \textsc{Skorokhod}'s $J_1$ topology.
\end{rmk}

\subsection{Properties of $\mathscr T_{\mathbbm x}$}\label{ssct:description limit tree}

We end this section on the encoding of $\R$-trees by studying those encoded by a normalized excursion $\mathbbm x \in D_0([0,1],\R_+)$ of a spectrally positive $\alpha$-stable \textsc{Lévy} process with $\alpha \in (1,2)$. Let us first describe $\mathbbm x$. Formally, consider a \textsc{Lévy} process $X$ whose \textsc{Laplace} transform is given, for all $t,\lambda >0$, by $\E\evt{\exp\cro{-\lambda X_t}} = e^{t\lambda^{\alpha}}$, then $\mathbbm x$ is defined as an excursion away from $0$ of $X$ conditioned to have a unit length (we can make sense of this conditioning by applying excursion theory or simply by the scaling property of $X$, see \eg \cite{Chaumont1994ExcursionStable}). This definition leads to the following path-properties holding almost surely for $\mathbbm x$:
\begin{enumerate}
    \item The set of discontinuities $\disc(\mathbbm x)$ is countable and dense in $(0,1)$;
    \item The jumps of $\mathbbm x$ are positive, \ie $\forall s \in [0,1]\ \mathbbm x(s) \geq \mathbbm x(s-)$;
    \item For all $s \in\disc(\mathbbm x)$, for all $\varepsilon$ small enough we have  $\inf_{[s-\varepsilon,s]} \mathbbm x < \mathbbm x(s-)$ and $\inf_{[s,s+\varepsilon]} \mathbbm x < \mathbbm x(s)$.
    \item For all $0<s<t<1$, there is at most one $r\in (s,t)$ such that $\mathbbm x(r) =\inf_{[s,t]}\mathbbm x$;
    \item For $\leb\ae\, t\in(0,1)$, there is no one-sided local minimum for $\mathbbm x$ at $t$, \ie  for $\varepsilon>0$ small enough, $\inf_{0 \leq s \leq \varepsilon} \mathbbm x(t-s)<\mathbbm x(t)$ and $\inf_{0 \leq s \leq \varepsilon} \mathbbm x(t+s)<\mathbbm x(t)$;
    \item For all $I$ non-trivial interval of $[0,1]$, $\mathbbm x$ is not monotone on $I$.

\end{enumerate}

\begin{rmk}
    These path-properties follow from the \textsc{Markov} property, the form of the \textsc{Lévy} measures associated with spectrally positive $\alpha$-stable \textsc{Lévy} processes and the duality property of these processes (see \cite[Proposition~2.10]{Kortchemski2014StableLaminations} for some proofs).
\end{rmk}

\paragraph{First properties and interpolation.} 
We prove Proposition~\ref{prop:Continuum random tree}, which gives that $\mathscr T_{\mathbbm x}$ is a binary continuous random tree, by means of a classical argument linking local minima in a contour function and degrees of the underlying tree. In short, given a continuous contour $f$ and a vertex $v\in \mathscr T_f$, the set $\pi_f^{-1}(\{v\})$ splits $f$ into $\deg(v)-1$ consecutive excursions above height $d_f(\pi_f(0),v)$, plus a remaining part. We refer to \cite[p.30]{Legall2005RandomTreesApplications} for a detailed discussion of this link.

\begin{proof}[Proof of Proposition~\ref{prop:Continuum random tree}.]
For the first claim, observe that thanks to properties $3$ and $4$ of $\mathbbm x$ we also have that for all one-to-one parametric representation $(\chi,\tau)$ of $\mathbbm x$ and for all $0<s<t<1$, there is at most one $r\in (s,t)$ such that $\chi(r) =\inf_{[s,t]}\chi$. Based on \cite[p.30]{Legall2005RandomTreesApplications} which links degree and contour function, we deduce that the degree of any vertex $ v \in \mathscr T_{\chi}$ is less than $3$, hence $\mathscr T_{\mathbbm x}$ is binary.

For the properties of $\mu_{\mathbbm x}$ and the set of leaves $\mathcal L(\mathscr T_{\mathbbm x})$, we use the fact that for $\leb\ae\, t\in(0,1)$, $\mathbbm x(t-)=\mathbbm x(t)$ and there is no one-sided local minimum for $\mathbbm x$ at $t$. It implies that for any one-to-one parametric representation $(\chi,\tau)$ of $\mathbbm x$, for $\leb\ae\, t\in(0,1)$ there is no one-sided local minimum for $\chi$ at $\tau^{-1}(t)$, hence $\pi_{\chi}\circ \tau^{-1}(t)$ is a leaf in $\mathscr T_{\chi}$ (again, see \cite[p.30]{Legall2005RandomTreesApplications}). Since $(\mathscr T_{\mathbbm x},\mu_{\mathbbm x})$ is isometric to $(\mathscr T_{\chi},\mu_{(\chi,\tau)})$, we directly get $\mu_{\mathbbm x}\!\left(\mathcal L(\mathscr T_{\mathbbm x}) \right)=1$. Moreover, by Lemma~\ref{lem:Criterion quotient space} and \eqref{eq:support} we get $\supp(\mu_{\mathbbm x})=\mathscr T_{\mathbbm x}$ and $\mathcal L(\mathscr T_{\mathbbm x})$ must be dense. Finally, we assume that $\mu_{\mathbbm x}$ has an atom $a$ to get a contradiction. The atom $a$ must be a leaf and it means that either $\pi_{\mathbbm x}^{-1}(\{a\})={\mathbbm x}^{-1}(\{0\})$ or $\pi_{\mathbbm x}^{-1}(\{a\})=[s,t]$ with ${\mathbbm x}$ constant on this segment. But $\mathbbm x$ vanishes only at $0$ and $1$ and cannot be monotone on a non-trivial interval, hence almost surely $\pi_{\mathbbm x}^{-1}(\{a\})$ is reduced to a singleton or $\{0,1\}$, which contradicts $\leb\bigr(\pi_{\mathbbm x}^{-1}(\{a\})\bigl)=\mu_{\mathbbm x}(\{a\})>0$. 
\end{proof}

The continuity of the encoding by a contour function also enables to transfer some properties of contour functions to their encoded trees. We use this to prove Proposition~\ref{prop:interpolation}, which states that $\alpha \mapsto  \bigl(\mathscr T_{\mathbbm x^{(\alpha)}},\mu_{\mathbbm x^{(\alpha)}}\bigr)$ interpolates between the unit segment (endowed with $\leb$) and the Brownian CRT. Note that there is a similar interpolation result for looptrees which is based on the convergence of encoding functions (\cite[Theorem~1.2]{CurienKortchemski2014StableLooptrees}), but here the continuity given by Proposition~\ref{prop:GHP M1} enables a simpler proof.

\begin{proof}[Proof of Proposition~\ref{prop:interpolation}]
    Thanks to the continuity of $x \mapsto \bigl(\mathscr T_x,\mu_x)$ established in Proposition~\ref{prop:GHP M1}, this is a direct consequence of the following facts:
    \begin{itemize}
        \item $(\mathbbm x^{(\alpha)})_{1 < \alpha \leq 2}$ converges in distribution toward $\mathbbm x^{(2)}=\sqrt 2 \mathbbm e$ when $\alpha \shortrightarrow 2$ (see \cite[Proposition~3.5]{CurienKortchemski2014StableLooptrees});
        \item $(\mathbbm x^{(\alpha)})_{1 < \alpha \leq 2}$ converges in distribution toward $t \mapsto (1-t)\1_{0 < t \leq 1}$ when $\alpha \shortrightarrow 1$, or more accurately their time-reversed versions converge toward the càdlàg function $t\mapsto t\1_{0 \leq t < 1}$ (see \cite[Proposition~3.6]{CurienKortchemski2014StableLooptrees}). According to Definition~\ref{def:Discontinuous Contour and Measured Real Tree}, the corresponding $\R$-tree is $\bigl(([0,1],\abs{\cdot}),\leb\bigr)$.
    \end{itemize} 
\vspace{-0.5cm}
\end{proof}
\paragraph{Stable looptrees and \textsc{Hausdorff} dimension.}

As mentioned in the introduction, another important property of $\mathscr T_{\mathbbm x}$ is that it may be seen as a \textit{spanning $\R$-tree} of the stable looptree $\mathscr L_{\mathbbm x}$. The intuition motivating this formulation comes in part from the relation between a rotated plane tree $\rot T$ and the looptree $\Loop(T)$ corresponding to $T$, which will be explored at the end of Section~\ref{sct:rotation}. Nevertheless, we try to motivate the idea of \textit{spanning $\R$-tree} to interpret Proposition~\ref{prop:Looptree and dimension} before proving this proposition in the rest of this section. Recall that $\mathscr L_{\mathbbm x}$ is a random compact metric space introduced in \cite{CurienKortchemski2014StableLooptrees} as the distributional scaling limit of $\Loop(\Tree_n)$. This space consists of \textit{continuous loops}, \ie subspaces isometric to circles, glued together by some additional limit points, but to simplify, we first give examples of $\R$-trees that could be seen as \textit{spanning $\R$-trees} of some simpler metric spaces made of loops:
\begin{itemize}
    \item The unit segment is a spanning $\R$-tree of the unit circle;
    \item Consider $T$ a finite plane tree, and let $\mathscr T$ (resp. $\mathscr L$) be the metric space obtained from $\rot T$ (resp. $\Loop(T)$) when edges are seen as unit-length segments. Based on the forthcoming Lemma~\ref{lem:discrete spanning tree} and Figure~\ref{fig:Rot in Looptree}, we will argue in Section~\ref{sct:rotation} that $\mathscr T$ is a spanning $\R$-tree of $\mathscr L$.
\end{itemize}
In both examples, a formal statement is that we have a mapping $p$ from an $\R$-tree $\mathscr T$ to a metric space $\mathscr L$ which is onto and \textsc{Lipschitz}, and there is a subset $\mathscr T^*$ of $\mathscr T$ such that $\mathscr T\!\setminus\!\mathcal L(\mathscr T) \subset \mathscr T^*$ and $p$ restricted to $\mathscr T^*$ is one-to-one. Proposition~\ref{prop:Looptree and dimension} gives the same statement for the $\R$-tree $\mathscr T_{\mathbbm x}$ and the metric space $\mathscr L_{\mathbbm x}$, with the additional result that they have the same \textsc{Hausdorff} dimension. To prove this proposition, we use the construction of $\mathscr L_{\mathbbm x}$ as a quotient space (see \cite[Section~2]{CurienKortchemski2014StableLooptrees}) obtained by considering another pseudo-distance on $[0,1]$ denoted here by $d_{\mathbbm x}^{\text{loop}}$ (defined by \cite[Equation~2.5]{CurienKortchemski2014StableLooptrees}), and we rely on the existing proof of the fact that $\dim_H\left( \mathscr L_{\mathbbm x}\right) = \alpha$ (see \cite[Section~3.3]{CurienKortchemski2014StableLooptrees}). Consequently we will refer to \cite{CurienKortchemski2014StableLooptrees} for more details on some definitions and arguments.


\begin{proof}[Proof of Proposition~\ref{prop:Looptree and dimension}.]
    We first prove the existence and properties of the map $p$. Note that by Lemma~\ref{lem:Criterion quotient space}, $\mathscr T_{\mathbbm x} = \bigl([0,1]/\!\!\sim_{\mathbbm x},d_{\mathbbm x}\bigr)$ as isometry classes, thus we may and will work with the quotient space $\bigl([0,1]/\!\!\sim_{\mathbbm x},d_{\mathbbm x}\bigr)$ instead of the isometry class $\mathscr T_{\mathbbm x}$. 
    
    An immediate consequence of \cite[Lemma~2.1]{CurienKortchemski2014StableLooptrees} is that $d_{\mathbbm x}^{\text{loop}} \leq d_{\mathbbm x}$, thus the canonical projection $\pi^{\text{loop}}_{\mathbbm x}:[0,1]\mapsto \mathscr L_{\mathbbm x}$ (defined after \cite[Definition~2.3]{CurienKortchemski2014StableLooptrees}) is such that for $s,t \in [0,1]$
    \begin{equation*}
        d_{\mathbbm x}^{\text{loop}}\bigl(\pi^{\text{loop}}_{\mathbbm x}(s),\pi^{\text{loop}}_{\mathbbm x}(t)\bigr) \leq d_{\mathbbm x}(s,t).
    \end{equation*}
    Thanks to this, $\pi^{\text{loop}}_{\mathbbm x}$ induces a map $p$ from the quotient space $\bigl([0,1]/\!\!\sim_{\mathbbm x},d_{\mathbbm x}\bigr)$ onto $\mathscr L_{\mathbbm x}$ such that $p\circ\pi^{\text{tree}}_{\mathbbm x}=\pi^{\text{loop}}_{\mathbbm x}$, and moreover $p$ is $1$-\textsc{Lipschitz}. Now we prove that in addition, for all $u \in \{\pi^{\text{tree}}_{\mathbbm x}(s), \text{ for } s \in \disc(\mathbbm x)\}$ we have $v \not\in \{\pi^{\text{tree}}_{\mathbbm x}(s), \text{ for } s \in \disc(\mathbbm x)\}$ such that $p(u)=p(v)$. For $t \in [0,1]$, we set $e(t)=\inf\{t'\geq t : \mathbbm x(t') < \mathbbm x(t-)\}$. By the path-properties of $\mathbbm x$ we always have that $e(t)\not\in\disc(\mathbbm x)$ and $d_{\mathbbm x}^{\text{loop}}\bigl(t,e(t)\bigr)=0$. Moreover, for $t \in \disc(\mathbbm x)$ such that $u=\pi^{\text{tree}}_{\mathbbm x}(t)$ we must have that $u$ is a leaf while $e(t)>t$ hence $e(t)$ is a one-sided local minimum and $v=\pi^{\text{tree}}_{\mathbbm x}\bigl(e(t)\bigr)$ is not a leaf. This prevents $v$ from being in $\{\pi^{\text{tree}}_{\mathbbm x}(s), \text{ for } s \in \disc(\mathbbm x)\}$, and $v$ is such that $p(v)=\pi^{\text{loop}}_{\mathbbm x}\bigl(e(t)\bigr)=\pi^{\text{loop}}_{\mathbbm x}\bigl(t\bigr)=p(u)$. We finally prove that $p$ is one-to-one once it is restricted to the complementary of $\{\pi^{\text{tree}}_{\mathbbm x}(s), \text{ for } s \in \disc(\mathbbm x)\}$. We first combine the path-properties of $\mathbbm x$ and \cite[Equations~2.4,~2.5]{CurienKortchemski2014StableLooptrees} to get that for $s<t \in [0,1]$,
    \begin{equation*}
        \text{if } d_{\mathbbm x}^{\text{loop}}(s,t)=0, \text{ then } t=e(s) \text{ and } d_{\mathbbm x}(s,t)=\mathbbm x(s)-\mathbbm x(s-).
    \end{equation*}
    For $u \neq v \in [0,1]/\!\!\sim_{\mathbbm x}$, consider $s,t \in [0,1]$ such that $u=\pi^{\text{tree}}_{\mathbbm x}(s)$ and $v=\pi^{\text{tree}}_{\mathbbm x}(t)$, then $p(u)=p(v)$ means that $d_{\mathbbm x}^{\text{loop}}(s,t)=0$ and $d_{\mathbbm x}^{\text{tree}}(s,t)>0$. According to the previous display it implies that one of $s,t$ belongs to $\disc(\mathbbm x)$, thus one of $u,v$ belongs to $\{\pi^{\text{tree}}_{\mathbbm x}(s), \text{ for } s \in \disc(\mathbbm x)\}$.

    We now turn to the \textsc{Hausdorff} dimension. Since $p:\mathscr T_{\mathbbm x} \mapsto \mathscr L_{\mathbbm x}$ is $1$-\textsc{Lipschitz} and onto, we have
    \begin{equation*}
        \dim_H\left(\mathscr L_{\mathbbm x}\right) \leq \dim_H\left(\mathscr T_{\mathbbm x}\right).
    \end{equation*}
    By \cite[Theorem~1.1]{CurienKortchemski2014StableLooptrees}, $\dim_H\left(\mathscr L_{\mathbbm x}\right)=\alpha$ so it only remains to upper-bound $\dim_H\left(\mathscr T_{\mathbbm x}\right)$ by $\alpha$ to get the desired conclusion. This upper-bound can be proved with the very same argument used for the upper-bound $\dim_H\left(\mathscr L_{\mathbbm x}\right) \leq \alpha$ in \cite[Section~3.3.1]{CurienKortchemski2014StableLooptrees}: one may consider the same partition of $[0,1]$ to build a partition of $\mathscr T_{\mathbbm x}$, and we have the same upper-bound on the diameter of those parts as \cite[Equation~3.26]{CurienKortchemski2014StableLooptrees}. Consequently we get the same conclusion, namely $\dim_H\left(\mathscr T_{\mathbbm x}\right) \leq \alpha$.
\end{proof}

\section{The rotation correspondence}\label{sct:rotation}

The rest of this paper, including this section, is now devoted to the study of the rotation $\rot$, which has been defined in the introduction. In this section, we first give an alternative description of $\rot T$ for a plane tree $T$. We also define some transformations related to the rotation, the co-rotation and the looptree transformation. They are not used to prove Theorem~\ref{thm:MainResult}, but they will give further insight on the large-scale effects of the rotation studied in the following sections.

\paragraph{A geometric description.} 
In Section~\ref{sct:Intro}, we gave a recursive definition of the rotation correspondence $\rot$ illustrated by Figure~\ref{fig:DefRecRotation}.  As mentioned earlier, it is clear that it defines a bijection between plane trees with $n$ vertices and binary plane trees with $n$ leaves (hence $2n-1$ vertices), for all $n \in \N$. However, this recursive definition will only be useful when it comes to induction. An important property that can be obtained this way is that for any plane tree $T$, we always have $$S_{\rot T}=C_T\oplus(-1)\ ,$$ where $\oplus$ is concatenation, \ie the \Luka walk of $\rot T$ is the contour process of $T$ plus a final step down (\textsc{Marckert} also refers to \textsc{Flajolet, Sedgewick} (\cite[exercise~5.42~p.262]{FlajoletSedgewick1996IntroductionAnalysisAlgorithms})  about this fact). For almost everything else, we will rather work with another definition provided by \textsc{Marckert} which involves a true rotation (in the geometric sense). Consider a plane tree $T$ embedded in the upper-half plane with the edges \textit{parent -- first child} being vertical and siblings being on the same horizontal lines. To draw $\rot T$, one only needs 3 steps (illustrated by Figure~\ref{fig:DefGeoRotation}):
\begin{itemize}
\item[]\textbf{step 1:} Remove the root and all edges of $T$, except the vertical edges \textit{parent -- first child}. Then for each vertex which initially had a younger sibling, add a horizontal edge between this vertex and its first sibling on its right (\ie the oldest among its younger siblings).
\item[]\textbf{step 2:} Make a rotation of $\pi/4$, so that the first child of the initial root clearly appears as the new root and each edge, when directed to go away from the new root, goes upward. Note that the former \textit{parent -- first child} edges go upward-left (we call them \textit{left edges}) while the added edges go upward-right (we call them \textit{right edges}).
\item[]\textbf{step 3:} Add 1 or 2 leaves to each vertex in order to get a binary tree (if there is no vertex, just add one in order to get the binary tree $\{\varnothing\}$).  When we add only 1 leaf at a vertex, its position (first or second child) is determined by the following rule: each \textit{left edge} must correspond to a first (\ie left) child and each \textit{right edge} must correspond to a second (\ie right) child.

\end{itemize}
This transformation is indeed the same as the rotation described above since it satisfies the same induction. 

\begin{figure}[h]
    \hspace{-1cm}\includegraphics[scale=0.13]{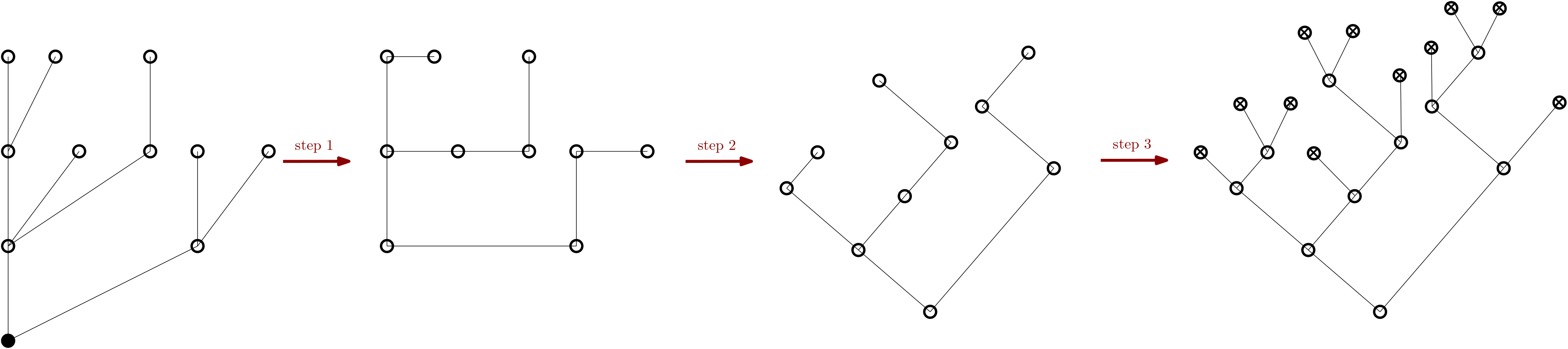}
     \caption{\centering A visual way to construct $\rot T$ from $T$. The root of $T$ is marked as the black vertex while the added leaves are crossed.}
    \label{fig:DefGeoRotation}
\end{figure}

\medskip
\paragraph{The internal subtree.}

The geometric description illustrated by Figure~\ref{fig:DefGeoRotation} enables the definition of a closely related tree that will be crucial to link $\rot T$ and $T$: Starting from $T\neq\{\varnothing\}$, apply \textbf{step 1} and \textbf{step 2} and then stop. The resulting plane tree may and will be identified to the subtree of the internal vertices of $\rot T$. We denote it by $(\rot T)^{\circ}$. Note that we will use both its plane tree structure and its subtree structure (in the graph-theoretic sense) of $\rot T$ without further explicit mention. For instance, a lexicographical enumeration of $(\rot T)^{\circ}$ will also be seen as a lexicographical enumeration of the internal vertices of $\rot T$ (observe that this enumeration is also increasing for the lexicographical order of $\rot T$, so there is no ambiguity).

The subtree $(\rot T)^{\circ}$ is useful to handle the rotation since the geometric description illustrated by Figure~\ref{fig:DefGeoRotation} enables a clear identification between vertices of $T\!\setminus\!\{\varnothing\}$ and vertices of $(\rot T)^{\circ}$  (as the former ones become the latter ones in this construction). For $u \in$  $T\!\setminus\!\{\varnothing\}$, we denote by $\widetilde u \in (\rot T)^{\circ}$  the associated internal vertex of $\rot T$ and call it the \textit{rotated version of $u$}. A simple but crucial observation about this identification is that we can express the height of a rotated vertex:
\begin{equation}\label{eq:decHauteurLargeur}
    \forall u \in T\!\setminus\!\{\varnothing\},\ |\widetilde u|=|u|-1+L(u),
\end{equation}
where $L(u)$ is the number of vertices grafted on $\Zintfo{\varnothing,u}$ on its left side, as defined in Section~\ref{sct:planeTrees}. Indeed, one can see in Figure~\ref{fig:DefGeoRotation} that the height of $\widetilde u$ is the number of edges in the path from $\varnothing$ to this vertex after \textbf{step 1}. Vertical edges amount to the height of $u$, minus $1$ as we have deleted the root, and horizontal edges amount to $L(u)$, hence we have \eqref{eq:decHauteurLargeur}.

A last important remark about the identification between vertices of $T\!\setminus\!\{\varnothing\}$ and vertices of $(\rot T)^{\circ}$ is that it preserves the lexicographical order: $u \prec u'$ in $T\!\setminus\!\{\varnothing\}$ implies $\widetilde u \prec \widetilde u'$ in $(\rot T)^{\circ}$ (see \cite[Lemma~1]{Marckert2004Rotation}). Hence this identification boils down to: 
\begin{center}
   $\widetilde u_1,\ldots,\widetilde u_{\#T-1}$ is the lexicographical enumeration of $(\rot T)^{\circ}$, \ie  for every  $i\in \Zintff{1,\#T-1}$,  $\widetilde u_i$ is the $i$th internal vertex of $\rot T$ in lexicographical order.
\end{center}

Note that extracting the internal subtree of a plane tree is not a one-to-one function. Thus, despite the previous handy identification, working with $(\rot T)^{\circ}$ only as a plane tree means losing information. One can see at \textbf{step 3} that the missing pieces of information given by the leaves are the following: a left leaf in $\rot T$ means that its parent (an internal vertex) corresponds to a leaf in $T$, while a right leaf means that its parent corresponds to a vertex in $T$ which is the last child of its own parent. 

\paragraph{A symmetric counterpart: the co-rotation.} 
We now briefly introduce another closely related transformation, but this one is not needed to understand and prove Theorem~\ref{thm:MainResult}. The reason to mention this transformation is rather to stress that the rotation is not a canonical correspondence between plane trees and binary trees since the ideas underlying the definitions of $\rot$ are the same, up to some symmetries, as the ones underlying this new transformation called here the \textit{co-rotation} and denoted by $\corot$. As this description suggests, these two transformations share many properties and thus most results about the rotation can be transferred to the co-rotation. However when it comes to the lexicographical order, using the rotation is not equivalent to using the co-rotation as the rotation is the only one to preserve this order (in the sense explained above). Because of this, the encoding processes of $\corot T$ for $T$ a large \textsc{Bienaymé} tree have an interesting property that does not hold in the case of the rotation: they jointly converge toward the same process (see Theorem~\ref{thm:Detailed Result Corotation}).

\smallskip

Let us properly define the co-rotation. As for the rotation, we have a recursive definition:
\begin{itemize}
    \item[] \textbf{Base case:} $\corot \{\varnothing\} =\{\varnothing\}$.
    \item[] \textbf{Inductive step:} Consider a plane tree $T$ with $n\geq 2$ vertices. Denote by $T_1,\ldots,T_k$ the $k\geq 1$ subtree(s) grafted on the root. Then $\corot T$ consists of a spine of $k+1$ vertices with $\corot T_k$ grafted on  the right of the $1$st one, \ldots, $\corot T_1$ grafted on  the right of the $k$th one (see Figure~\ref{fig:DefRecCorotation}).
\end{itemize}

\begin{figure}[h!]
    \centering
    \includegraphics[width=0.75\linewidth]{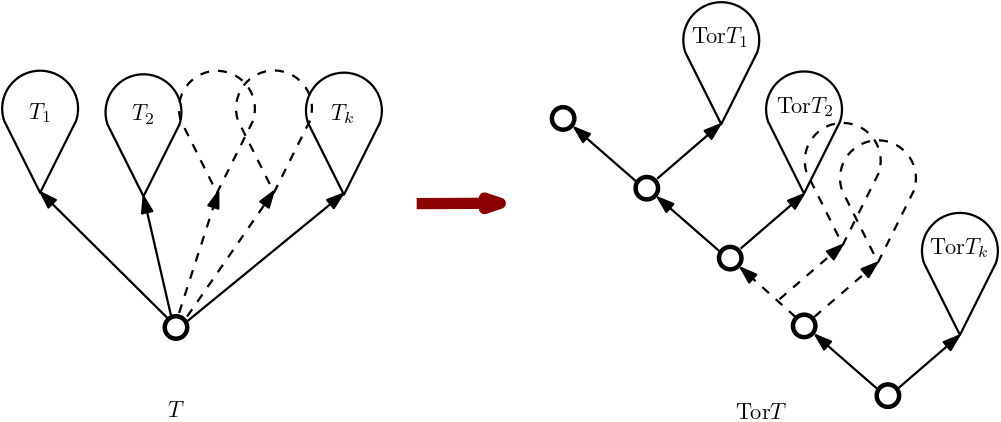}
        \caption{\centering Recursive definition of $\corot$}

    \label{fig:DefRecCorotation}
\end{figure}

There is also a geometric description similar to the one illustrated by Figure~\ref{fig:DefGeoRotation} in the case of the rotation, but one must keep the edges \textit{parent -- last child} instead of \textit{parent -- first child} at \textbf{step~1} and then apply a rotation of $-\pi/4$ instead of $\pi/4$ at \textbf{step 2}. We can thus define the internal subtree $(\corot T)^{\circ}$ in the same way we defined $(\rot T)^{\circ}$, and again there is an identification between vertices of $T\!\setminus\!\{\varnothing\}$ and vertices of $(\corot T)^{\circ}$. However \textit{this identification with $(\corot T)^{\circ}$ does not preserve the lexicographical order}.

$\corot$ lacks another property of $\rot$: there is no simple relation between $S_{\corot\!T}$ and  $C_T$. Actually, we can express $S_{\corot\!T}$ from $S_T$ without much difficulty (we will give some details in Section~\ref{ssct:Luka of Corot}) but this relation is less direct than the one between $S_{\rot\!T}$ and $C_T$.

\smallskip

We conclude this part with an explicit and simple relation between $\rot$ and $\corot$. First, for every plane tree $T$ we introduce its mirror tree $T^{\div}$. It is obtained by embedding $T$ in the upper-half plane, rooted at the origin, and by taking its image by $(x,y)\mapsto (-x,y)$ \ie the reflection symmetry with respect to the vertical line (see Figure~\ref{fig:mirror1}). Equivalently, $T^{\div}$ can be defined by considering $T$ and, for each vertex with some children, by reversing the birth order of its children. The mirror transformation $T\mapsto T^{\div}$ clearly is an involution, and for every plane tree $T$ we have\begin{equation}\label{eq:Link Rot Corot}
    \corot T = \bigl(\rot(T^{\div})\bigr)^{\div}.
\end{equation}

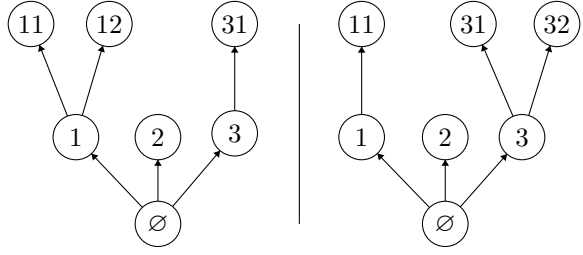
\begin{figure}[h]
    \begin{center}
    \begin{tikzpicture}[scale=0.1]
\tikzstyle{every node}+=[inner sep=0pt]
\draw [black] (24.9,-43.2) circle (3);
\draw (24.9,-43.2) node {$\varnothing$};
\draw [black] (14.1,-31.7) circle (3);
\draw (14.1,-31.7) node {$1$};
\draw [black] (24.9,-31.7) circle (3);
\draw (24.9,-31.7) node {$2$};
\draw [black] (35,-31.2) circle (3);
\draw (35,-31.2) node {$3$};
\draw [black] (8.2,-16.7) circle (3);
\draw (8.2,-16.7) node {$11$};
\draw [black] (18.5,-16.7) circle (3);
\draw (18.5,-16.7) node {$12$};
\draw [black] (35,-16.7) circle (3);
\draw (35,-16.7) node {$31$};
\draw [black] (62.7,-43.2) circle (3);

\draw (62.7,-43.2) node {$\varnothing$};
\draw [black] (62.7,-31.7) circle (3);
\draw (62.7,-31.7) node {$2$};
\draw [black] (51.7,-31.7) circle (3);
\draw (51.7,-31.7) node {$1$};
\draw [black] (72.8,-31.7) circle (3);
\draw (72.8,-31.7) node {$3$};
\draw [black] (66.4,-16.7) circle (3);
\draw (66.4,-16.7) node {$31$};
\draw [black] (77.4,-16.7) circle (3);
\draw (77.4,-16.7) node {$32$};
\draw [black] (51.7,-16.7) circle (3);
\draw (51.7,-16.7) node {$11$};
\draw [black] (26.83,-40.9) -- (33.07,-33.5);
\fill [black] (33.07,-33.5) -- (32.17,-33.79) -- (32.94,-34.43);
\draw [black] (35,-28.2) -- (35,-19.7);
\fill [black] (35,-19.7) -- (34.5,-20.5) -- (35.5,-20.5);
\draw [black] (24.9,-40.2) -- (24.9,-34.7);
\fill [black] (24.9,-34.7) -- (24.4,-35.5) -- (25.4,-35.5);
\draw [black] (22.85,-41.01) -- (16.15,-33.89);
\fill [black] (16.15,-33.89) -- (16.34,-34.81) -- (17.07,-34.13);
\draw [black] (14.94,-28.82) -- (17.66,-19.58);
\fill [black] (17.66,-19.58) -- (16.95,-20.21) -- (17.91,-20.49);
\draw [black] (13,-28.91) -- (9.3,-19.49);
\fill [black] (9.3,-19.49) -- (9.13,-20.42) -- (10.06,-20.05);
\draw [black] (62.7,-40.2) -- (62.7,-34.7);
\fill [black] (62.7,-34.7) -- (62.2,-35.5) -- (63.2,-35.5);
\draw [black] (64.68,-40.95) -- (70.82,-33.95);
\fill [black] (70.82,-33.95) -- (69.92,-34.23) -- (70.67,-34.89);
\draw [black] (60.63,-41.03) -- (53.77,-33.87);
\fill [black] (53.77,-33.87) -- (53.97,-34.79) -- (54.69,-34.1);
\draw [black] (51.7,-28.7) -- (51.7,-19.7);
\fill [black] (51.7,-19.7) -- (51.2,-20.5) -- (52.2,-20.5);
\draw [black] (71.62,-28.94) -- (67.58,-19.46);
\fill [black] (67.58,-19.46) -- (67.43,-20.39) -- (68.35,-20);
\draw [black] (73.68,-28.83) -- (76.52,-19.57);
\fill [black] (76.52,-19.57) -- (75.81,-20.19) -- (76.76,-20.48);

\draw [black] (43.5,-43.2) -- (43.5,-16.7);
\end{tikzpicture}
    \caption{\centering The tree from Figure~\ref{fig:lexicoEnumeration} and its mirror version (both embedded in $\Ul$).}
    \label{fig:mirror1}
    \end{center}
\end{figure}

\paragraph{Spanning trees of looptrees.}
The two correspondences $\rot$ and $\corot$ are also related to the looptree transformation: they produce spanning trees of looptrees. The looptree transformation has been introduced in \cite{CurienKortchemski2014StableLooptrees} in order to understand the associated limiting objects called \textit{stable looptrees}, and we have explained in Section~\ref{ssct:description limit tree} that the scaling limits of rotated \textsc{Bienaymé} trees may also be seen as spanning trees of stable looptrees. We have chosen to present clearly this relation between discrete rotated trees and discrete looptrees, but we will not use it formally, it will only serve as inspiration to understand the link between scaling limits of rotated trees and looptrees.

\smallskip

We first define the looptree $\Loop(T)$ associated to a plane tree $T$: $\Loop(T)$ is the medial graph of $T$, \ie the graph obtained by gluing cycles according to the tree structure of $T$. More precisely, for each edge in $T$ there is a vertex in $\Loop(T)$, and for each couple of edges $(e,e')$ in $T$ incident to a common vertex $v\in T$ such that $e'$ directly comes after $e$ in clockwise order around $v$, there is an edge in $\Loop(T)$ between the vertices corresponding to $e$ and $e'$ (see Figure~\ref{fig:Looptree}).  In the following, such a couple of edges $(e,e')$ is called a \textit{corner} of $T$, so we simply say that an edge in $\Loop(T)$ corresponds to a corner in $T$. Moreover, we distinguish the corner of $T$ going from the edge  \textit{root -- last child of the root} to the edge  \textit{root -- first child of the root}, we call it the root-corner of $T$ and the corresponding edge in $\Loop(T)$ is the root-edge of $\Loop(T)$. We also endow $\Loop(T)$ with the cyclic order on its edges inherited from the cyclic order on the corners of $T$. This cyclic order on corners is formally defined as follows: $(e,e')$ is the successor of $(f,f')$ if and only if $f'=e$ and $f$ and the vertex of incidence of  $(f,f')$ is different from the vertex of incidence of $(e,e')$. A more visual way to define this order is to say that the enumeration of the corners which starts at the root-corner and follows this cyclic order must correspond to the sequence of vertices defining the contour process of $T$.

\begin{figure}[h]
    \centering
\begin{tikzpicture}[scale=0.11]
\tikzstyle{every node}+=[inner sep=0pt]
\draw [black] (31.6,-48) circle (3);
\draw (31.6,-48) node {$u_1$};
\draw [black] (53.5,-47.3) circle (3);
\draw (53.5,-47.3) node {$u_8$};
\draw [black] (23.4,-38.1) circle (3);
\draw (23.4,-38.1) node {$u_2$};
\draw [black] (18,-28.7) circle (3);
\draw (18,-28.7) node {$u_3$};
\draw [black] (27.4,-28.7) circle (3);
\draw (27.4,-28.7) node {$u_4$};
\draw [black] (31.6,-38.1) circle (3);
\draw (31.6,-38.1) node {$u_5$};
\draw [black] (39.5,-38.1) circle (3);
\draw (39.5,-38.1) node {$u_6$};
\draw [black] (39.5,-28.7) circle (3);
\draw (39.5,-28.7) node {$u_7$};
\draw [black] (49.8,-38.1) circle (3);
\draw (49.8,-38.1) node {$u_9$};
\draw [black] (57.8,-38.1) circle (3);
\draw (57.8,-38.1) node {$u_{10}$};
\draw [black] (42.8,-56.9) circle (3);
\draw (42.8,-56.9) node {$u_0$};
\draw [black] (42.8,-56.9) circle (2.4);
\draw [black] (29.69,-45.69) -- (25.31,-40.41);
\draw (26.95,-44.48) node [left] {$e_2$};
\draw [black] (31.6,-45) -- (31.6,-41.1);
\draw (31.1,-43.05) node [left] {$e_5$};
\draw [black] (33.47,-45.66) -- (37.63,-40.44);
\draw (36.11,-44.47) node [right] {$e_6$};
\draw [black] (39.5,-35.1) -- (39.5,-31.7);
\draw (40,-33.4) node [right] {$e_7$};
\draw [black] (21.91,-35.5) -- (19.49,-31.3);
\draw (20.04,-34.63) node [left] {$e_3$};
\draw [black] (24.57,-35.34) -- (26.23,-31.46);
\draw (26.14,-34.35) node [right] {$e_4$};
\draw [black] (52.38,-44.52) -- (50.92,-40.88);
\draw (50.9,-43.59) node [left] {$e_9$};
\draw [black] (54.77,-44.58) -- (56.53,-40.82);
\draw (56.36,-43.74) node [right] {$e_{10}$};
\draw [black] (40.45,-55.03) -- (33.95,-49.87);
\draw (35.83,-52.94) node [below] {$e_1$};
\draw [black] (45.03,-54.9) -- (51.27,-49.3);
\draw (49.53,-52.59) node [below] {$e_8$};
\end{tikzpicture}
\hfill 
\begin{tikzpicture}[scale=0.09]
\tikzstyle{every node}+=[inner sep=0pt]
\draw [black] (40.8,-51.4) circle (3);
\draw (40.8,-51.4) node {$e_1$};
\draw [black] (55.3,-52.8) circle (3);
\draw (55.3,-52.8) node {$e_8$};
\draw [black] (21.3,-48.1) circle (3);
\draw (21.3,-48.1) node {$e_2$};
\draw [black] (12.7,-56.9) circle (3);
\draw (12.7,-56.9) node {$e_3$};
\draw [black] (9.5,-43.4) circle (3);
\draw (9.5,-43.4) node {$e_4$};
\draw [black] (26.3,-32.5) circle (3);
\draw (26.3,-32.5) node {$e_5$};
\draw [black] (40.8,-36.2) circle (3);
\draw (40.8,-36.2) node {$e_6$};
\draw [black] (50.4,-25.5) circle (3);
\draw (50.4,-25.5) node {$e_7$};
\draw [black] (64.1,-38) circle (3);
\draw (64.1,-38) node {$e_9$};
\draw [black] (72.9,-52.1) circle (3);
\draw (72.9,-52.1) node {$e_{10}$};
\draw [black] (42.329,-48.842) arc (137.66934:31.30084:7.504);
\draw [line width=0.3mm, black] (55.165,-55.777) arc (-14.01965:-177.01018:7.519);
\draw [black] (38.613,-53.441) arc (-54.52181:-144.6886:11.4);
\draw [black] (20.13,-45.348) arc (-164.84849:-230.69415:10.905);
\draw [black] (28.946,-31.111) arc (108.77801:42.59234:9.638);
\draw [black] (42.293,-38.794) arc (23.16928:-23.16928:12.723);
\draw [black] (21.24,-51.08) arc (-12.53664:-76.14627:7.55);
\draw [black] (9.746,-56.526) arc (-109.46841:-223.86137:7.01);
\draw [black] (12.296,-42.362) arc (99.58525:36.97946:7.964);
\draw [black] (39.213,-33.68) arc (-159.72456:-284.07211:7.205);
\draw [black] (54.15,-50.042) arc (-166.19468:-255.27629:9.732);
\draw [black] (67.056,-38.44) arc (72.72798:-8.79029:9.76);
\draw [black] (71.028,-54.431) arc (-46.96425:-128.48054:10.483);
\draw [black] (7.339,-41.336) arc (254.04214:-33.95786:2.25);
\draw [black] (14.023,-59.58) arc (54:-234:2.25);
\draw [black] (23.815,-30.84) arc (263.98866:-24.01134:2.25);
\draw [black] (52.339,-27.761) arc (28.69536:-112.49203:7.209);
\draw [black] (51.016,-22.576) arc (195.84277:-92.15723:2.25);
\draw [black] (62.777,-35.32) arc (234:-54:2.25);
\draw [black] (75.887,-52.184) arc (116.12449:-171.87551:2.25);
\end{tikzpicture}
\caption{\centering The tree $T$ from Figure~\ref{fig:DefGeoRotation} on the left and its associated looptree $\Loop(T)$ on the right. Labels indicate the correspondence between the vertices of $T\!\setminus\!\{\varnothing\}$, the edges of $T$ and the vertices of $\Loop(T)$. The root-edge of $\Loop(T)$ is marked in bold. The cyclic order on edges matches the clockwise order.}
    \label{fig:Looptree}
\end{figure}
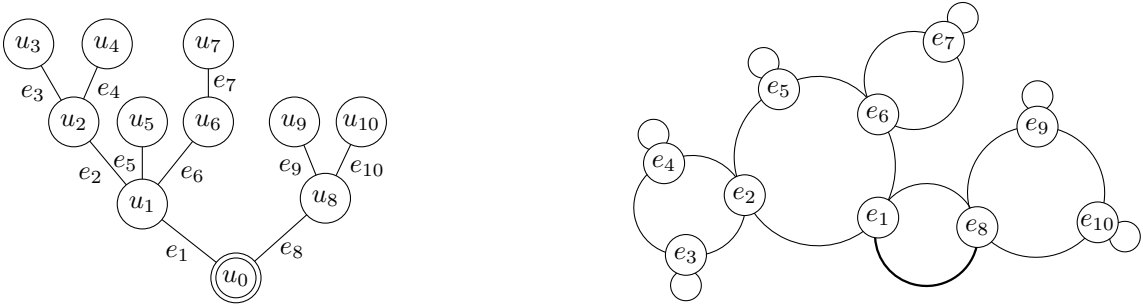

Note that we can identify the vertices of $T\!\setminus\!\{\varnothing\}$ with the edges of $T$ (an edge corresponds to its farthest endpoint from the root) hence we can identify the vertices of $T\!\setminus\!\{\varnothing\}$ with those of $\Loop(T)$. We may now write the relation between $\rot T$ and $\Loop(T)$ as follows:

\begin{lem}\label{lem:discrete spanning tree}
    Consider the sequence $(c_0,\ldots ,c_{2\#T-2})$ of  edges of $\Loop(T)$ which start from the root-edge $c_0$ of $\Loop(T)$ and then follows the cyclic order until it reaches again the root-edge (so that $c_{2\#T-2}=c_0$). We define three subsets of edges with a recursive algorithm. First set $\mathcal E=\mathcal R=\mathcal L= \emptyset$, then for each $k$ between $1$ and $ 2\#T-2$, 
    \begin{itemize}
        \item[-] if $c_k$ is a self-loop then add $c_k$ to $\mathcal L$;
        \item[-] else, if $\mathcal E \cup \{c_k\}$ contains a cycle then add $c_k$ to $\mathcal R$;
        \item[-] otherwise add $c_k$ to $\mathcal E$.
    \end{itemize}

    Then the set of edges $\mathcal E$ forms a spanning tree of $\Loop(T)$. Moreover, $(\rot T)^{\circ}$ is equal to this tree once it is rooted at the vertex of incidence of $c_0$ and $c_1$ and ordered by the cyclic order of $\Loop(T)$ (with the convention that the first child of the root is the endpoint of $c_1$). Finally, if $T\neq \{\varnothing \}$ then $\mathcal L$ is in bijection with the set of left leaves of $\rot T$ and $\mathcal R$ is in bijection with the set of right leaves of $\rot T$.
\end{lem}

As one may expect, $(\corot T)^{\circ}$ can be seen as another spanning tree of $\Loop(T)$, constructed with the same algorithm up to some modifications: one must reverse the enumeration of edges of $\Loop(T)$ (\ie  consider the sequence $c_0,c_{2\#T-3},\dots,c_1,c_0$), declare that the corresponding tree is rooted at the vertex of incidence of $c_0$ and $c_{2\#T-3}$ and declare that $c_{2\#T-3}$ is the \textit{last} edge out of the root. Moreover $\mathcal L$ and $\mathcal R$ exchange their role in the case of the co-rotation. See Figure~\ref{fig:Rot in Looptree} for a drawing of these two specific spanning trees of a looptree. 

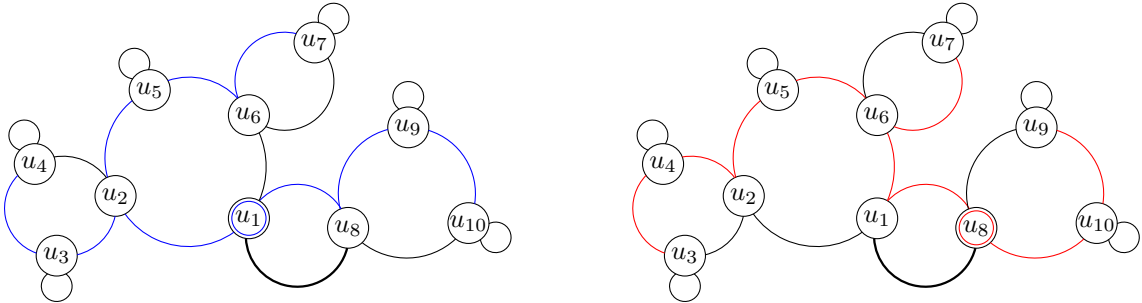
\begin{figure}[h]
\centering
\begin{tikzpicture}[scale=0.09]
\tikzstyle{every node}+=[inner sep=0pt]
\draw [black] (40.8,-51.4) circle (3);
\draw [blue] (40.8,-51.4) circle (2.5);
\draw (40.8,-51.4) node {$u_1$};
\draw [black] (55.3,-52.8) circle (3);
\draw (55.3,-52.8) node {$u_8$};
\draw [black] (21.3,-48.1) circle (3);
\draw (21.3,-48.1) node {$u_2$};
\draw [black] (12.7,-56.9) circle (3);
\draw (12.7,-56.9) node {$u_3$};
\draw [black] (9.5,-43.4) circle (3);
\draw (9.5,-43.4) node {$u_4$};
\draw [black] (26.3,-32.5) circle (3);
\draw (26.3,-32.5) node {$u_5$};
\draw [black] (40.8,-36.2) circle (3);
\draw (40.8,-36.2) node {$u_6$};
\draw [black] (50.4,-25.5) circle (3);
\draw (50.4,-25.5) node {$u_7$};
\draw [black] (64.1,-38) circle (3);
\draw (64.1,-38) node {$u_9$};
\draw [black] (72.9,-52.1) circle (3);
\draw (72.9,-52.1) node {$u_{10}$};
\draw [blue] (42.329,-48.842) arc (137.66934:31.30084:7.504);
\draw [line width = 0.3mm, black] (55.165,-55.777) arc (-14.01965:-177.01018:7.519);
\draw [blue] (38.613,-53.441) arc (-54.52181:-144.6886:11.4);
\draw [blue] (20.13,-45.348) arc (-164.84849:-230.69415:10.905);
\draw [blue] (28.946,-31.111) arc (108.77801:42.59234:9.638);
\draw [black] (42.293,-38.794) arc (23.16928:-23.16928:12.723);
\draw [blue] (21.24,-51.08) arc (-12.53664:-76.14627:7.55);
\draw [blue] (9.746,-56.526) arc (-109.46841:-223.86137:7.01);
\draw [black] (12.296,-42.362) arc (99.58525:36.97946:7.964);
\draw [blue] (39.213,-33.68) arc (-159.72456:-284.07211:7.205);
\draw [blue] (54.15,-50.042) arc (-166.19468:-255.27629:9.732);
\draw [blue] (67.056,-38.44) arc (72.72798:-8.79029:9.76);
\draw [black] (71.028,-54.431) arc (-46.96425:-128.48054:10.483);
\draw [black] (7.339,-41.336) arc (254.04214:-33.95786:2.25);
\draw [black] (14.023,-59.58) arc (54:-234:2.25);
\draw [black] (23.815,-30.84) arc (263.98866:-24.01134:2.25);
\draw [black] (52.339,-27.761) arc (28.69536:-112.49203:7.209);
\draw [black] (51.016,-22.576) arc (195.84277:-92.15723:2.25);
\draw [black] (62.777,-35.32) arc (234:-54:2.25);
\draw [black] (75.887,-52.184) arc (116.12449:-171.87551:2.25);

\end{tikzpicture}
\hfill 
\begin{tikzpicture}[scale=0.09]
\tikzstyle{every node}+=[inner sep=0pt]
\draw [black] (40.8,-51.4) circle (3);
\draw (40.8,-51.4) node {$u_1$};
\draw [black] (55.3,-52.8) circle (3);
\draw [red] (55.3,-52.8) circle (2.5);
\draw (55.3,-52.8) node {$u_8$};
\draw [black] (21.3,-48.1) circle (3);
\draw (21.3,-48.1) node {$u_2$};
\draw [black] (12.7,-56.9) circle (3);
\draw (12.7,-56.9) node {$u_3$};
\draw [black] (9.5,-43.4) circle (3);
\draw (9.5,-43.4) node {$u_4$};
\draw [black] (26.3,-32.5) circle (3);
\draw (26.3,-32.5) node {$u_5$};
\draw [black] (40.8,-36.2) circle (3);
\draw (40.8,-36.2) node {$u_6$};
\draw [black] (50.4,-25.5) circle (3);
\draw (50.4,-25.5) node {$u_7$};
\draw [black] (64.1,-38) circle (3);
\draw (64.1,-38) node {$u_9$};
\draw [black] (72.9,-52.1) circle (3);
\draw (72.9,-52.1) node {$u_{10}$};
\draw [red] (42.329,-48.842) arc (137.66934:31.30084:7.504);
\draw [line width = 0.3mm, black] (55.165,-55.777) arc (-14.01965:-177.01018:7.519);
\draw [black] (38.613,-53.441) arc (-54.52181:-144.6886:11.4);
\draw [red] (20.13,-45.348) arc (-164.84849:-230.69415:10.905);
\draw [red] (28.946,-31.111) arc (108.77801:42.59234:9.638);
\draw [red] (42.293,-38.794) arc (23.16928:-23.16928:12.723);
\draw [black] (21.24,-51.08) arc (-12.53664:-76.14627:7.55);
\draw [red] (9.746,-56.526) arc (-109.46841:-223.86137:7.01);
\draw [red] (12.296,-42.362) arc (99.58525:36.97946:7.964);
\draw [black] (39.213,-33.68) arc (-159.72456:-284.07211:7.205);
\draw [black] (54.15,-50.042) arc (-166.19468:-255.27629:9.732);
\draw [red] (67.056,-38.44) arc (72.72798:-8.79029:9.76);
\draw [red] (71.028,-54.431) arc (-46.96425:-128.48054:10.483);
\draw [black] (7.339,-41.336) arc (254.04214:-33.95786:2.25);
\draw [black] (14.023,-59.58) arc (54:-234:2.25);
\draw [black] (23.815,-30.84) arc (263.98866:-24.01134:2.25);
\draw [red] (52.339,-27.761) arc (28.69536:-112.49203:7.209);
\draw [black] (51.016,-22.576) arc (195.84277:-92.15723:2.25);
\draw [black] (62.777,-35.32) arc (234:-54:2.25);
\draw [black] (75.887,-52.184) arc (116.12449:-171.87551:2.25);
\end{tikzpicture}
    \caption{\centering $(\rot T)^{\circ}$, in blue, and $(\corot T)^{\circ}$, in red, embedded in $\Loop(T)$. Roots are marked by double circles.}
    \label{fig:Rot in Looptree}
\end{figure}
Lemma~\ref{lem:discrete spanning tree} may be proved either by induction or by means of the geometric definition of $\rot$ illustrated at Figure~\ref{fig:DefGeoRotation}. We sketch this latter version. We have seen that there is a correspondence between vertices of $(\rot T)^{\circ}$ and edges of $T$, and \textbf{step 1} in the geometric construction of the rotation reveals that the edges of $(\rot T)^{\circ}$ correspond to some corners of $T$. This makes $(\rot T)^{\circ}$ a spanning tree of $\Loop(T)$, and then it suffices to check that the missing corners of $T$ are those corners corresponding to the leaves of $T$ (which give self-loops in $\Loop(T)$ and left leaves in $\rot T$) plus those corners linking the last edge of a vertex of $T$ to its parent edge -or first edge in the case of the root- (which give edges ending cycle in $\Loop(T)$ and right leaves in $\rot T$).

\smallskip 

To conclude, Lemma~\ref{lem:discrete spanning tree} says a bit more than the fact that $(\rot T)^{\circ}$ is a spanning tree of $\Loop(T)$, as it gives that each edge of $\rot T$ corresponds to an edge of $\Loop(T)$ and vice versa. Actually, we can construct $\Loop(T)$ simply by considering $\rot T$ and by merging each leaf with an internal vertex. With this observation, we explain the simple example of a spanning $\R$-tree given in Section~\ref{ssct:description limit tree}. Consider $\mathscr T$ (resp. $\mathscr L$) the metric space obtained from $\rot T$ (resp. $\Loop(T)$) when edges are seen as unit-length segments, then we can also construct $\mathscr L$ as a quotient from $\mathscr T$ by identifying each leaf (except the root) with a branching point. In particular the canonical projection $p:\mathscr T \mapsto \mathscr L$ is onto and \textsc{Lipschitz}, with  a subset $\mathscr T^*$ of $\mathscr T$ such that $\mathscr T\!\setminus\!\mathcal L(\mathscr T) \subset \mathscr T^*$ and $p$ restricted to $\mathscr T^*$ is one-to-one.

\section{Scaling limits of encoding processes}\label{sct:Encoding Processes}

Thanks to the extension of the encoding of $\R$-trees to càdlàg contour functions, and in particular its continuity property given by Propositions~\ref{prop:GHlipschitzM1} and~\ref{prop:GHP M1}, we can obtain Theorem~\ref{thm:MainResult} by studying the joint scaling limits of some encoding process of $\Tree_n$ and $\rot \Tree_n$ in the setting of \textsc{Skorokhod}'s $M_1$ topology. This section is devoted to this study, we first detail the framework and give some known results about the encoding processes of $\Tree_n$ and then combine these results with the properties of $\rot$ given in Section~\ref{sct:rotation} to determine the desired scaling limits. We also study the co-rotation and show that its encoding processes asymptotically behave in a nicer way than those of the rotation, namely they jointly converge toward the same excursion.

\subsection{Details of the framework}\label{ssct:framework}
\paragraph{Hypothesis on \textsc{Bienaymé} trees.}
For all adequate $n\in \N^*$, $\Tree_n$ still denotes a random tree with law $\BGW\evt{\,\cdot\,\vert\text{\#vertices = $n$}}$, where $\mu$ is a fixed critical offspring distribution, and we suppose that $\mu$ is attracted to a spectrally positive $\alpha$-stable distribution, with $\alpha \in (1,2]$. This means that there exist real sequences $(a_n)_{n\geq 1},(b_n)_{n \geq 1}$ such that if $(Z_n)_{n\geq 1}$ is a sequence of \iid variables with distribution $\mu$ then we have 
\begin{equation}\label{eq:defStable}
    \frac{\sum_{i=1}^n Z_i-b_n}{a_n} \xrightarrow[\tend{n}]{\dstb} Z,
\end{equation}
where $Z$ is a real variable with finite \textsc{laplace} transform satisfying $\E\evt{\exp\cro{-\lambda Z}} = e^{\lambda^{\alpha}}$, for all $\lambda >0$. 
The variable $Z$ is a Gaussian in the case $\alpha=2$, otherwise it is a spectrally positive $\alpha$-stable variable. 
Let us also mention that $(a_n)_{n \geq 1}$ is necessarily of the form $a_n=n^{1/\alpha}\ell(n)$ with $\ell$ a slowly varying function\footnote{In this paper, one only needs to know that $\ell$ is such that for all $\varepsilon >0$, for $x$ large enough we have $x^{-\varepsilon} \leq \ell(x)\leq x^{\varepsilon}$. See \cite{BinghamGoldieTeugels1987RegularVariation} for a proper definition and a general account of slowly varying functions.}. 
We refer to \cite{JansonPreprintStableDistributions} for a review of stable distributions as well as necessary and sufficient conditions on $\mu$ that ensure that it is attracted to an $\alpha$-stable distribution. 

\paragraph{Choice of topology.}
A distribution attracted to a spectrally positive $\alpha$-stable distribution also satisfies a functional version of \eqref{eq:defStable}, but the limiting process is a Brownian motion in the case $\alpha=2$ only, otherwise it is a spectrally positive $\alpha$-stable \textsc{Lévy} process which is càdlàg but not continuous. For this reason, some encoding processes of $\Tree_n$ actually have discontinuous scaling limits and we must work in the space of càdlàg functions $D([0,1])$. In the studies on random trees, the topology used on this space is usually \textsc{Skorokhod}'s $J_1$ topology, however we will see that some encoding processes of rotated trees converge with respect to the weaker \textsc{Skorokhod}'s $M_1$ topology but not with respect to the $J_1$ topology, and in light of Proposition~\ref{prop:GHP M1} this weaker convergence is enough to get the desired result. We thus choose the following convention:
\begin{gather*}
    \textit{In this paper (except in the appendix~\ref{app:MirrorLuka}),}\\ \textit{all convergences of processes are with respect to \textsc{Skorokhod}'s $M_1$ topology.}
\end{gather*}
One only needs the definition \eqref{eq:distance M1} and the few properties of the $M_1$ topology given in section~\ref{ssct:M1DiscontinuousContour} to understand the rest of this paper (see \cite{Whitt2002StochasticProcessLimits} for a more general account). We just stress that the convergences already established with respect to $J_1$ imply the same convergences with respect to $M_1$ but they will be written in a slightly different way, because under $J_1$ it is often required to work with the \textit{discontinuous} interpolation of a sequence to establish its scaling limit while there is no need for this under $M_1$. More precisely, consider $A=(A(k))_{0\leq k\leq p}$ a real sequence. We have defined its time-scaled function $a\in C([0,1])$ by linear interpolation in order to make it continuous, but we could have chosen to embed $A$ in $D([0,1])$ by defining the piecewise constant function $a':t\mapsto A(\lfloor pt\rfloor)$ (see Figure~\ref{fig:Continuous and discontinuous interpolations}). This can make a huge difference for the $J_1$ topology as $d_{J_1}(a,a')\geq \max_{1 \leq k \leq p}\abs{A(k)-A(k-1)}/2$, but one can find parametric representations of $a$ and $a'$ with the same spatial component and close temporal components as depicted in Figure~\ref{fig:Parametric representation of interpolations}, which entails that $d_{M_1}(\lambda a,\lambda a')\leq 1/p$ for any $\lambda \in \R$. As $p$ always goes to $+\infty$ here, working with $a$ or $a'$ makes no difference for scaling limits with respect to the $M_1$ topology, thus we may and will choose to always use the linear interpolation, \ie the time-scaled function defined in Section~\ref{sct:planeTrees}.

\begin{figure}[h]
    \centering
    \includegraphics[width=0.5\linewidth]{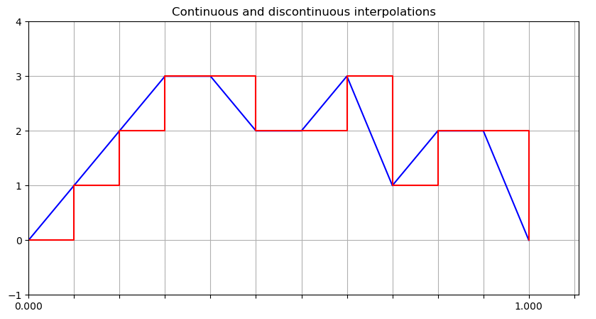}
    \caption{\centering Continuous (in blue) and discontinuous (in red) interpolations of the height process from Figure~\ref{fig:encodingProcesses}.}
    \label{fig:Continuous and discontinuous interpolations}
\end{figure}
\begin{figure}[h]
    \centering
    \includegraphics[width=\linewidth]{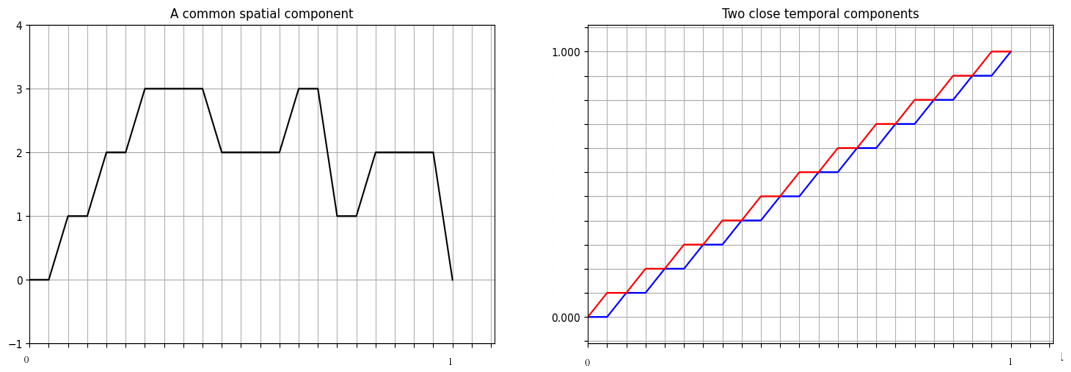}
    \caption{\centering Some parametric representations of the two functions from Figure~\ref{fig:Continuous and discontinuous interpolations}. They share the same spatial component, but have slightly different temporal components (the matching colors indicate which function is represented).}
    \label{fig:Parametric representation of interpolations}
\end{figure}


\subsection{Known scaling limits for \textsc{Bienaymé} trees}\label{ssct:KnownScalingLimits}

\paragraph{\textsc{Duquesne}'s Theorem.}

First we introduce a theorem of \textsc{Duquesne \cite{Duquesne2003Contour}} which generalizes \textsc{Aldous}' theorem by giving the joint scaling limits of all encoding functions of $\Tree_n$ in this general framework. The limiting processes can be defined from a spectrally positive $\alpha$-stable \textsc{Lévy} process in the following way: consider a \textsc{Lévy} process $X$ whose \textsc{Laplace} transform is given by $\E\evt{\exp\cro{-\lambda X_t}} = e^{t\lambda^{\alpha}}$, for all $t,\lambda >0$. The continuous-time height process $H$ associated to $X$ is informally defined from $X$ by means of the relation expressing the height process of a tree from its \Luka walk. More formally, $H$ is the continuous modification of the process $H'$ such that $H'_t$ is the local time, at instant $t$, of $X^{(t)}$ at its supremum, where $X^{(t)}$ is the dual \textsc{Lévy} process of $(X_s)_{0\leq s \leq t}$. The scaling limits obtained by \textsc{Duquesne} are a normalized excursion $\mathbbm x$ of a spectrally positive $\alpha$-stable \textsc{Lévy} process together with its associated height excursion $\mathbbm h$, which can be thought of as a positive excursion $\mathbbm x$ of $X$ away from $0$ conditioned to have a unit length together with the corresponding excursion $\mathbbm h$ of $H$. We refer to \cite[Section~3.1-2]{Duquesne2003Contour} for proper definitions and details, the only thing to keep in mind here is that $\mathbbm x$ is continuous if and only if $\alpha=2$ while $\mathbbm h$ is always continuous.

We can now give a formal statement of \textsc{Duquesne}'s theorem, where we recall that $s_T,h_T$ and $c_T$ refer to the time-scaled function of the \Luka walk, height process and contour process of a tree $T$ as defined in Section~\ref{sct:planeTrees}.

\begin{thm}[\textsc{Duquesne \cite{Duquesne2003Contour}}]\label{thm:Duquesne}
Let $\Tree_n$ be distributed as $\BGW\evt{\,\cdot\,\vert\textup{\#vertices = $n$}}$ where the offspring distribution $\mu$ is critical and has variance $\sigma^2$ possibly infinite. Assuming that $\mu$ is attracted to a stable distribution of index $\alpha \in (1,2]$, there is an increasing slowly varying function $\ell$ that depends on $\mu$ such that
\begin{equation*}
    \left(\frac{1}{\ell(n)n^{1/\alpha}}s_{\Tree_n},\frac{\ell(n)}{n^{1-1/\alpha}}h_{\Tree_n},\frac{\ell(n)}{n^{1-1/\alpha}}c_{\Tree_n}\right) \xrightarrow[\tend{n}]{\dstb} (\mathbbm x,\mathbbm h,\mathbbm h).
\end{equation*}
Furthermore when $\alpha=2$ then $\mathbbm x = \mathbbm h$ and $\ell(n)\cv{n}\frac{\sigma}{\sqrt 2}$ (if $\sigma^2=+\infty$ this means $\ell(n)\cv{n}+\infty$).
\end{thm}

\begin{rmk}
Theorem~\ref{thm:Duquesne} is a reinforcement of \cite[Theorem~3.1]{Duquesne2003Contour} that is not explicitly written but results from the same arguments (see in particular \cite[Remark~3.2, Proposition~4.3, Equation~(30)]{Duquesne2003Contour}).
\end{rmk} 

We stress that the case $\sigma^2<+\infty$ has some specific features. In this case we must have $\alpha=2$ and all three limiting processes are equal, but we also see that $\ell$ has a finite limit hence the three encoding functions of $\Tree_n$ are rescaled by the \textit{same} factor $\sqrt n$ (up to a constant). This case has actually been proved first by \textsc{Marckert \& Mokkadem \cite{MarckertMokkadem2003SameExcursion}} under an additional assumption of finite exponential moment for $\mu$, and its features will make a real difference for the scaling limit of $\rot\Tree_n$, thus we choose to restate this result on its own. In this particular case, we use $\mathbbm e$ a normalized excursion of a \textit{standard} Brownian motion $B$, but we have to consider $X=\sqrt 2 B$ in order to have a \textsc{Lévy} process with \textsc{Laplace} transform $\E\evt{\exp\cro{-\lambda X_t}} = e^{t\lambda^2}$, so one must keep in mind that $\mathbbm e$ has the law of $\mathbbm x/\sqrt 2$ in this situation.

\begin{thm}[\textsc{Marckert \& Mokkadem \cite{MarckertMokkadem2003SameExcursion}, Duquesne \cite{Duquesne2003Contour}}]\label{thm:MMD}
Let $\Tree_n$ be distributed as $\BGW\evt{\,\cdot\,\vert\textup{\#vertices = $n$}}$ where the offspring distribution $\mu$ is critical. Assuming that $\mu$ has variance $\sigma^2 <+\infty$, we have the convergence in distribution
    \begin{equation*}
        \frac{1}{\sqrt n}\left(s_{\Tree_n},h_{\Tree_n},c_{\Tree_n}\right) \xrightarrow[\tend{n}]{\dstb} \left(\sigma\mathbbm e,\frac{2}{\sigma}\mathbbm e,\frac{2}{\sigma}\mathbbm e\right).
    \end{equation*}
\end{thm}

\medskip

\begin{rmk}
There are actually similar scaling limit results for other types of conditioned \textsc{Bienaymé}  trees, such as those conditioned to have \textit{at least} $n$ vertices and those conditioned to have height greater than $n$. We refer to \cite[section~5]{Legall2010ItoRandomTrees} for the joint convergence of all three encoding functions rescaled by $\sqrt n$ towards the same Brownian excursion under a general \textit{regular conditioning} in the case $\sigma^2<+\infty$, and we refer to \cite[section~2.5]{DuquesneLegall2002RandomTreeLevyProcessesBranching}  for the joint convergence of all three encoding functions under some specific conditioning in the general stable case. In the following we will not be concerned with those other conditioned \textsc{Bienaymé} trees but most arguments and results here can be adapted as soon as we have a scaling limit result analogous to Theorem~\ref{thm:Duquesne}.
\end{rmk}

\paragraph{Extension based on the mirror transformation.}

We now turn to an extension of the previous result needed to study the rotation. Recall the mirror transformation $T \mapsto T^{\div}$ introduced in Section~\ref{sct:rotation}. We use the previous result of \textsc{Duquesne} to describe the joint convergence of encoding processes of both $\Tree_n$ and $\Tree_n^{\div}$.

Observe that $C_{\Tree_n^{\div}}$ simply is the reversed contour process of $\Tree_n$, obtained by following the contour of the tree from right to left instead of left to right. This means that $c_{\Tree_n^{\div}}=\widehat  c_{\Tree_n}$,  where for all $x\in D([0,1])$, $\widehat x$ is defined by $\widehat x:t \mapsto x\bigl((1-t)\text{-}\bigr)$. Since $x \mapsto \widehat x$ is continuous with respect to the $M_1$ topology, the asymptotic behaviours of $C_{\Tree_n^{\div}}$ and $C_{\Tree_n}$ are related in a simple way. However, we will rather be interested in $S_{\Tree_n^{\div}}$ and $H_{\Tree_n^{\div}}$ as it will allow us to extract some information about $\rot\Tree_n$. As height processes are asymptotically equal to contour processes for these trees, we can relate the asymptotic behaviours of $H_{\Tree_n^{\div}}$ and $H_{\Tree_n}$ without much difficulty, but the links between $S_{\Tree_n^{\div}}$ and $S_{\Tree_n}$ are more intricate. Nevertheless, one may deduce the following result.

\begin{prop}\label{prop:DuquesneMiroir}
Let $\Tree_n$ be distributed as $\BGW\evt{\,\cdot\,\vert\textup{\#vertices = $n$}}$ where the offspring distribution $\mu$ is critical. Assuming that $\mu$ is attracted to a stable distribution of index $\alpha \in (1,2]$, there is a measurable function $F^{\div}:D([0,1])\rightarrow D([0,1])$ such that 
    \begin{equation*}
        \left(\frac{1}{\ell(n)n^{1/\alpha}}s_{\Tree_n},\frac{\ell(n)}{n^{1-1/\alpha}}h_{\Tree_n},\frac{1}{\ell(n)n^{1/\alpha}} s_{\Tree_n^{\div}},\frac{\ell(n)}{n^{1-1/\alpha}} h_{\Tree_n^{\div}}\right) \xrightarrow[\tend{n}]{\dstb} \left(\mathbbm x,\mathbbm h,F^{\div}(\mathbbm x ),\widehat{\mathbbm h}\right).
    \end{equation*} 
Moreover, when $\alpha=2$ the normalized excursion $\mathbbm x$ and $F^{\div}$ are such that $F^{\div}(\mathbbm x )=\widehat {\mathbbm x}$ almost surely.
\end{prop}
\begin{rmk}
    Assuming $\mu$ has variance $\sigma^2<+\infty$, we may reformulate Proposition~\ref{prop:DuquesneMiroir} as
    \begin{equation*}
        \frac{1}{\sqrt n}\left(s_{\Tree_n},h_{\Tree_n},s_{\Tree_n^{\div}},h_{\Tree_n^{\div}}\right) \xrightarrow[\tend{n}]{\dstb} \left(\sigma\mathbbm e,\frac{2}{\sigma}\mathbbm e,\sigma\widehat{\mathbbm e},\frac{2}{\sigma}\widehat{\mathbbm e}\right).
    \end{equation*}
\end{rmk}
We prove Proposition~\ref{prop:DuquesneMiroir} in Appendix~\ref{app:MirrorLuka}. Let us just stress here that $\Tree_n^{\div}$ is also distributed according to $\BGW\evt{\,\cdot\,\vert\text{\#vertices = $n$}}$, since the mirror transformation clearly preserves the distribution $\BGW$  as well as the set of trees with $n$ vertices. As a consequence $(s_{\Tree_n^{\div}},h_{\Tree_n^{\div}})$ is distributed as $(s_{\Tree_n},h_{\Tree_n})$, whose convergence is given by Theorem~\ref{thm:Duquesne}, and thus the proof of Proposition~\ref{prop:DuquesneMiroir} is just a question of
understanding the intricate dependence between the limiting processes.

We do not have a fully explicit expression for $F^{\div}(\mathbbm x)$ (see the last remark of Appendix~\ref{app:MirrorLuka} for details) but an important property follows from Proposition~\ref{prop:DuquesneMiroir}: since $(s_{\Tree_n^{\div}},h_{\Tree_n^{\div}})$ is distributed as $(s_{\Tree_n},h_{\Tree_n})$, their limits $\bigl(F^{\div}(\mathbbm x ),\widehat{\mathbbm h}\bigr)$ and $(\mathbbm x,\mathbbm h)$ have the same distribution. In particular, $F^{\div}(\mathbbm x )$ is an excursion of a spectrally positive $\alpha$-stable \textsc{Lévy} process, and its associated height excursion is $\widehat{\mathbbm h}$. Moreover as $T \mapsto T^{\div}$ is an involution we clearly have 
\begin{equation*}
    F^{\div}\left( F^{\div}(\mathbbm x)\right) = \mathbbm x \ \as[]
\end{equation*}
The main reason that leads us to state Proposition~\ref{prop:DuquesneMiroir} is that the connection between the scaling limits of $s_{\Tree_n}$ and $s_{\Tree_n^{\div}}$, given by $F^{\div}$, will be used to understand the connection between $s_{\Tree_n}$ and the encoding functions of $\rot \Tree_n$. It will also be useful to understand the connection between the scaling limits of $\rot\Tree_n$ and $\corot\Tree_n$.


\subsection{Encoding processes of a large \textsc{Bienaymé} tree and its rotation}\label{ssct:mainResults}

We are now all set to study the encoding processes of $\rot\Tree_n$ and compare them to those of $\Tree_n$. We provide Table~\ref{tab:knonwScalingLimits} to summarize the notation introduced in the previous section and needed in the following.

\medskip

\begin{table}[htbp]
\caption{Main notation for scaling limits of encoding processes.}
\centering
\begin{tabular}{c p{10cm}}
\toprule

 $\mu$& critical offspring distribution attracted to a spectrally positive $\alpha$-stable distribution with $\alpha \in (1,2]$;\\
 $\sigma^2$&variance of $\mu$, which may be $+\infty$;\\
 $\ell$&increasing slowly varying function related to $\mu$ as in Theorem~\ref{thm:Duquesne};\\

 \hline

 $(\Tree_n)_n$&sequences (with indices $n$ such that $\BGW\evt{\text{\#vertices = $n$}}>0$) of random trees with $\Tree_n$ distributed according to $\BGW\evt{\,\cdot\,\vert\text{\#vertices} = n}$; \\

\hline
$\Tree_n^{\div}$ & mirror tree of $\Tree_n$; \\
 \hline
 
 $\mathbbm e$&normalized excursion of a standard Brownian motion (\ie with \textsc{Laplace} exponent $\lambda \mapsto \lambda^2/2$);\\
 $\mathbbm x$&normalized excursion of a strictly $\alpha$-stable spectrally positive \textsc{Lévy} process with \textsc{Laplace} exponent $\lambda \mapsto \lambda^{\alpha}$;\\
 $\mathbbm h$&continuous-time height excursion associated with $\mathbbm x$;\\

 \hline
$x \mapsto \widehat x$ &endomorphism of $D([0,1])$  defined by  $\widehat x(t) = x\bigl((1-t)\text{-}\bigr)$;\\
$F^{\div}$ & measurable function from $D([0,1])$ to $D([0,1])$ given by Proposition~\ref{prop:DuquesneMiroir}; \\
 \bottomrule
\end{tabular}
\label{tab:knonwScalingLimits}
\end{table}

\paragraph{Statement and consequences.} Our most complete result about the rotation is the following joint convergence of all encoding functions of $\Tree_n$ and of $\rot\Tree_n$.

\begin{thm}\label{thm:DetailedResult}
Let $\Tree_n$ be distributed as $\BGW\evt{\,\cdot\,\vert\textup{\#vertices = $n$}}$ where the offspring distribution $\mu$ is critical, has variance $\sigma^2 \leq +\infty$ and is attracted to a stable distribution with index $\alpha \in (1,2]$.
\begin{itemize}
\item Assuming $\sigma^2 <+\infty$, we have
    \begin{equation*}
        \frac{1}{\sqrt n}\left(s_{\Tree_n},c_{\Tree_n},s_{\rot\!\Tree_n},c_{\rot\!\Tree_n}\right) \xrightarrow[\tend{n}]{\dstb} \left(\sigma \mathbbm e,\frac{2}{\sigma}\mathbbm e,\frac{2}{\sigma}\mathbbm e,\frac{2+\sigma^2}{\sigma}\mathbbm e\right).
    \end{equation*}

\item Assuming $\sigma^2 =+\infty$, we have
    \begin{equation*}
        \left(\frac{1}{\ell(n)n^{1/\alpha}}s_{\Tree_n},\frac{\ell(n)}{n^{1-1/\alpha}}c_{\Tree_n},\frac{\ell(n)}{n^{1-1/\alpha}} s_{\rot\!\Tree_n},\frac{1}{\ell(n)n^{1/\alpha}}c_{\rot\!\Tree_n}\right) \xrightarrow[\tend{n}]{\dstb} \left(\mathbbm x,\mathbbm h,\mathbbm h,\widehat{\mathbbm y}\right)
    \end{equation*} where $\mathbbm y=F^{\div}(\mathbbm x)$.

\end{itemize}
   
In addition, in both cases height functions are asymptotically the same as contour functions:
\begin{equation*}
    \frac{\ell(n)}{n^{1-1/\alpha}}\bigl\Vert c_{\Tree_n}- h_{\Tree_n}\bigr\Vert_{\infty}  \xrightarrow[\tend{n}]{\prob} 0 \text{ and }  d_{M_1}\left(\frac{1}{\ell(n)n^{1/\alpha}}c_{\rot\!\Tree_n},\frac{1}{\ell(n)n^{1/\alpha}}h_{\rot\!\Tree_n}\right)  \xrightarrow[\tend{n}]{\prob} 0.
\end{equation*}

\end{thm}

\begin{rmk}
Observe that when $\alpha=2$, $\widehat{\mathbbm y }=\mathbbm x=\mathbbm h\ \as[]$ and all six encoding processes jointly converge towards the same excursion $\mathbbm x$, which is distributed as $\sqrt 2 \mathbbm e$. Moreover, since a Brownian excursion is continuous, the convergence in distribution can also be seen with respect to $\norm{\cdot}_{\infty}$and we actually have:
\begin{equation*}
    \frac{1}{\ell(n)\sqrt n} \bigl\Vert c_{\rot\!\Tree_n}-h_{\rot\!\Tree_n}\bigr\Vert_{\infty}  \xrightarrow[\tend{n}]{\prob} 0.
\end{equation*}
Another feature appears in the case $\sigma^2 <+\infty$ due to the specific property that all encoding processes of $\Tree_n$ have the same scale $\sqrt n$ in this setting. The functions $\tfrac{1}{\sigma\sqrt n}s_{\Tree_n}$, $\tfrac{\sigma}{2\sqrt n}c_{\Tree_n}$, $\tfrac{\sigma}{2\sqrt n}s_{\rot\!\Tree_n}$, $\tfrac{\sigma}{(2+\sigma^2)\sqrt n}c_{\rot\!\Tree_n}$ jointly converge towards the same Brownian excursion, so here again all scaling factors have order $\sqrt n$, but most importantly in this case the rescaling factor of $c_{\rot\!\Tree_n}$ is \textit{not} asymptotically equivalent to the rescaling factor of $s_{\Tree_n}$. This contrasts with the symmetry of the scaling factors in the case $\sigma^2 =+\infty$.
\end{rmk}

Thanks to Proposition~\ref{prop:GHlipschitzM1}, Theorem~\ref{thm:MainResult}  appears as a consequence of Theorem~\ref{thm:DetailedResult}. 
\begin{proof}[Proof of Theorem~\ref{thm:MainResult}.]
We first restrict our attention to the components corresponding to the contour processes to get a simpler statement.
\begin{itemize}
    \item When $\sigma^2<+\infty$
\begin{equation}\label{eq:OnlyContourFinite}
    \frac{1}{\sqrt n}\bigl(c_{\Tree_n}, c_{\rot\!\Tree_n}\bigr) \xrightarrow[\tend{n}]{\dstb} \left(\frac{2}{\sigma} \mathbbm e,\frac{2+\sigma^2}{\sigma}\mathbbm e\right).
\end{equation}
\item When $\sigma^2=+\infty$
    \begin{equation}\label{eq:OnlyContourInfinite}
    \Bigl(\frac{\ell(n)}{n^{1-1/\alpha}}c_{\Tree_n},\frac{1}{\ell(n)n^{1/\alpha}}c_{\rot\!\Tree_n}\Bigr) \xrightarrow[\tend{n}]{\dstb} \bigl(\widehat {\mathbbm h},\widehat {\mathbbm x}\bigr).
\end{equation}
\end{itemize}
Note that in the second case, we use that $\bigl(F^{\div}(\mathbbm x ),\widehat{\mathbbm h}\bigr)$ is distributed as $(\mathbbm x,\mathbbm h)$ in order to write a limit that does not involve $F^{\div}$. Then we apply the continuous mapping $(x,y)\mapsto (\mathscr T_x,\mathscr T_y)$ (Proposition~\ref{prop:GHlipschitzM1}) and we use \eqref{eq:Tree approximation} to get Theorem~\ref{thm:MainResult}. We can use $\mathbbm h$ and $\mathbbm x$ instead of $\widehat {\mathbbm h}$ and $\widehat {\mathbbm x}$ since for all $x \in D_0([0,1],\R_+)$, $\mathscr{T}_{\widehat x}$ is isometric to $\mathscr{T}_x$.
\end{proof}
\begin{rmk}
    The exact same proof with Proposition~\ref{prop:GHP M1} instead of Proposition~\ref{prop:GHlipschitzM1} gives a convergence with respect to $\dGHP$.
\end{rmk}

\paragraph{Proof strategy for Theorem~\ref{thm:DetailedResult}.} We now discuss the main steps of the proof of Theorem~\ref{thm:DetailedResult} in an informal way. As a preliminary remark, recall that up to a last step down $C_{\Tree_n}$ and $S_{\rot\!\Tree_n}$ are the same process, hence $S_{\rot\!\Tree_n}$ is already fully understood and we must focus on the study of $H_{\rot\!\Tree_n}$ and $C_{\rot\!\Tree_n}$. To do so, our main tool is the internal subtree $(\rot T)^{\circ}$ defined in Section~\ref{sct:rotation}. 

First, we express the height of vertices in $(\rot\Tree_n)^{\circ}$ with known processes, thanks to \eqref{eq:decHauteurLargeur} and the mirror transformation. However, this requires considering an enumeration $\left(\widetilde w_{k}\right)_{1 \leq k\leq n-1}$ of $(\rot\Tree_n)^{\circ}$ different from the lexicographical one. We will formally define this enumeration later in the proof, in short we will first introduce $\left( w_{k}\right)_{0 \leq k\leq n-1}$ the \textit{mirrored enumeration} of $\Tree_n$ (see Figure~\ref{fig:mirroredEnumeration} below) and then we will use the correspondence $u \in \Tree_n\!\setminus\!\{\varnothing\} \mapsto \widetilde u \in (\rot\Tree_n)^{\circ}$ defined in Section~\ref{sct:rotation} to obtain $\left(\widetilde w_{k}\right)_{1 \leq k\leq n-1}$.  This enumeration $\left(\widetilde w_{k}\right)_{1 \leq k\leq n-1}$ is such that the non-standard height process $H_n^*$ based on it, which is defined by  $H^*_n(k)=|\widetilde w_{k}|$ for $1 \leq k \leq n-1$ and $H^*_n(0)=H^*_n(n)=0$, has a scaling limit easily deduced from Theorem~\ref{thm:Duquesne}. This convergence of $H_n^*$ is actually a joint convergence involving $H_n^*$ and the encoding processes of $\Tree_n$, see \eqref{eq:FirstCvInfinite} and \eqref{eq:FirstCvFinite} below.  

Then, we relate this process $H_n^*$ to the height and contour processes of both $(\rot\Tree_n)^{\circ}$ and $\rot\Tree_n$. For this last part, we actually relate these processes through three independent lemmas. They all rely on the same method, namely establishing combinatorial links between the discrete processes and then use them to control the $M_1$ distance between their rescaled functions, and we have chosen to postpone their proofs to Section~\ref{sct:ProofLemmas} which will be dedicated to this method. 

In Section~\ref{ssct:HeightAndContour}, we adapt a classical argument to the $M_1$ setting: 
if a sequence of trees $(T_n)_n$ is such that $\#T_n\rightarrow+\infty$ and $\max H_{T_n}= o\bigl(\#T_n\bigr)$ as $n$ goes to $+\infty$, then when $n$ is large the rescaled height process of $T_n$ is roughly the same as its rescaled contour process. We formalize this in the next lemma.

\begin{lem}\label{lem:HeightIsRoughlyContour}
Let $(T_n)_n$ be a sequence of (possibly random) finite plane trees. Assume that
\begin{align*}
    \#T_n\xrightarrow[\tend{n}]{\prob}+\infty & \text{ and } 
    \frac{\max H_{T_n}}{\# T_n} \xrightarrow[\tend{n}]{\prob} 0.
\end{align*}

Then for any real sequence $(\lambda(n))_n$ such that $\lambda(n)\rightarrow 0$, we have
    \begin{equation*}
        d_{M_1}\bigl(\lambda(n) h_{T_n}, \lambda(n) c_{T_n}\bigr) \xrightarrow[\tend{n}]{\prob} 0 .
    \end{equation*}
\end{lem}
Let us mention that the scaling limit result for $H^*_n$ will yield that we can apply Lemma~\ref{lem:HeightIsRoughlyContour} to $\rot\Tree_n$ and $(\rot\Tree_n)^{\circ}$.

In Section~\ref{ssct:rightmostEnumeration}, we then study the enumeration $\left(\widetilde w_{k}\right)_{1 \leq k\leq n-1}$  and we obtain the following result linking the time-scaled functions of $H_n^*$ and $C_{(\rot\!\Tree_n)^{\circ}}$.

\begin{lem}\label{lem:RightHeightAndRightContourM1}
For any real sequence $(\lambda(n))_n$, we have
     \begin{equation*}
        d_{M_1}\bigl(\lambda(n)\widehat c_{(\rot\!\Tree_n)^{\circ}}, \lambda(n)h^*_n\bigr) \xrightarrow[\tend{n}]{\prob} 0,
    \end{equation*}
where $\widehat c_{(\rot\!\Tree_n)^{\circ}}$ is the \textit{reversed} time-scaled contour function of $(\rot\Tree_n)^{\circ}$, obtained by applying $x\mapsto \widehat x$ to ${c_{(\rot\!\Tree_n)^{\circ}}}$.
\end{lem}

Finally, in Section~\ref{ssct:AddingTheLeaves} we study how the leaves are distributed in $\rot\Tree_n$ in order to relate the height processes $H_{(\rot\!\Tree_n)^{\circ}}$ and $H_{\rot\!\Tree_n}$. Again, we deduce that their rescaled functions are roughly the same.

\begin{lem}\label{lem:AddingTheLeaves}
For any real sequence $(\lambda(n))_n$ such that $\lambda(n)\rightarrow 0$, we have
    \begin{equation*}
        d_{M_1}\bigl(\lambda(n) h_{(\rot\!\Tree_n)^{\circ}}, \lambda(n)h_{\rot\!\Tree_n}\bigr) \xrightarrow[\tend{n}]{\prob} 0.
    \end{equation*}
\end{lem}
Theorem~\ref{thm:DetailedResult} directly follows from these three lemmas combined with the joint convergence of $H_n^*$, the three encoding processes of $\Tree_n$, and $S_{\rot\!\Tree_n}$.

\medskip

We now make formal the above proof strategy, with the assumption that Lemmas~\ref{lem:HeightIsRoughlyContour},~\ref{lem:RightHeightAndRightContourM1} and~\ref{lem:AddingTheLeaves} hold, and we refer to Section~\ref{sct:ProofLemmas} for their proofs. We provide Tables~\ref{tab:knonwScalingLimits} and~\ref{tab:Reminder} to recall some useful notation.
\begin{table}[htbp]
\centering
\caption{Reminders on plane trees and the rotation.}
\label{tab:Reminder}
\begin{tabular}{c p{10cm}}
\toprule
 $\widetilde u$, for $u\in T\!\setminus\!\{\varnothing\}$&vertex of $(\rot T)^{\circ}$ corresponding to $u$ and called the rotated version of $u$ (defined in Section~\ref{sct:rotation});\\
 $L(u)$ (resp. $R(u)$), for $u \in T$&number of edges of $T$ grafted on $\Zintfo{\varnothing,u}$ on its \textit{left} (resp. \textit{right}) side (defined in Section~\ref{sct:planeTrees});\\
$T \mapsto T^{\div}$ &  the mirror transformation (introduced in Section~\ref{sct:rotation}); \\
 \bottomrule
\end{tabular}

\end{table}

\begin{proof}[Proof of Theorem~\ref{thm:DetailedResult}.]
Recall \eqref{eq:decHauteurLargeur} which states that the height of a rotated vertex $\widetilde u \in (\rot\Tree_n)^{\circ}$ is made of two contributions, namely $|u|$ and $L(u)$, and also recall that the \Luka walk $S_{\Tree_n}$ of $\Tree_n$ is such that $S_{\Tree_n}(k)=R(u_k)$ for all $k<n$, where $u_0,\ldots u_{n-1}$ is the lexicographical enumeration of $\Tree_n$. In order to exchange \textit{right} and \textit{left}, we consider a new enumeration of $\Tree_n$ based on its mirror version $\Tree_n^{\div}$. We identify each vertex of $\Tree_n$ with its mirror image in $\Tree_n^{\div}$ (\ie its image by the reflection symmetry, see Figure~\ref{fig:mirroredEnumeration}). Thanks to this identification, a lexicographical enumeration of $\Tree_n^{\div}$ gives a \textit{mirrored enumeration} of $\Tree_n$, which will be systematically denoted by $w_0=\varnothing, w_1, \ldots, w_{n-1}$ to avoid confusion with the lexicographic enumeration (again, see Figure~\ref{fig:mirroredEnumeration}). As the right of the mirror image of some vertex $u$ corresponds to the left of $u$ in the initial tree, while the height is preserved, we get that 
\begin{equation}\label{eq:MirrorLukaIsLeft}
    \text{for } 0\leq k \leq n-1,\ L(w_k)=S_{\Tree_n^{\div}}(k) \text{ and } |w_k|=H_{\Tree_n^{\div}}(k).
\end{equation}
\begin{figure}[h]
    \begin{center}
    \begin{tikzpicture}[scale=0.1]
\tikzstyle{every node}+=[inner sep=0pt]
\draw [black] (24.9,-43.2) circle (3);
\draw (24.9,-43.2) node {$w_0$};
\draw [black] (14.1,-31.7) circle (3);
\draw (14.1,-31.7) node {$w_4$};
\draw [black] (24.9,-31.7) circle (3);
\draw (24.9,-31.7) node {$w_3$};
\draw [black] (35,-31.2) circle (3);
\draw (35,-31.2) node {$w_1$};
\draw [black] (8.2,-16.7) circle (3);
\draw (8.2,-16.7) node {$w_6$};
\draw [black] (18.5,-16.7) circle (3);
\draw (18.5,-16.7) node {$w_5$};
\draw [black] (18.5,-16.7) circle (2.4);
\draw [black] (35,-16.7) circle (3);
\draw (35,-16.7) node {$w_2$};
\draw [black] (62.7,-43.2) circle (3);
\draw (62.7,-43.2) node {$u_0^{\div}$};
\draw [black] (62.7,-31.7) circle (3);
\draw (62.7,-31.7) node {$u_3^{\div}$};
\draw [black] (51.7,-31.7) circle (3);
\draw (51.7,-31.7) node {$u_1^{\div}$};
\draw [black] (72.8,-31.7) circle (3);
\draw (72.8,-31.7) node {$u_4^{\div}$};
\draw [black] (66.4,-16.7) circle (3);
\draw (66.4,-16.7) node {$u_5^{\div}$};
\draw [black] (66.4,-16.7) circle (2.4);
\draw [black] (77.4,-16.7) circle (3);
\draw (77.4,-16.7) node {$u_6^{\div}$};
\draw [black] (51.7,-16.7) circle (3);
\draw (51.7,-16.7) node {$u_2^{\div}$};
\draw [black] (26.83,-40.9) -- (33.07,-33.5);
\fill [black] (33.07,-33.5) -- (32.17,-33.79) -- (32.94,-34.43);
\draw [black] (35,-28.2) -- (35,-19.7);
\fill [black] (35,-19.7) -- (34.5,-20.5) -- (35.5,-20.5);
\draw [black] (24.9,-40.2) -- (24.9,-34.7);
\fill [black] (24.9,-34.7) -- (24.4,-35.5) -- (25.4,-35.5);
\draw [black] (22.85,-41.01) -- (16.15,-33.89);
\fill [black] (16.15,-33.89) -- (16.34,-34.81) -- (17.07,-34.13);
\draw [black] (14.94,-28.82) -- (17.66,-19.58);
\fill [black] (17.66,-19.58) -- (16.95,-20.21) -- (17.91,-20.49);
\draw [black] (13,-28.91) -- (9.3,-19.49);
\fill [black] (9.3,-19.49) -- (9.13,-20.42) -- (10.06,-20.05);
\draw [black] (62.7,-40.2) -- (62.7,-34.7);
\fill [black] (62.7,-34.7) -- (62.2,-35.5) -- (63.2,-35.5);
\draw [black] (64.68,-40.95) -- (70.82,-33.95);
\fill [black] (70.82,-33.95) -- (69.92,-34.23) -- (70.67,-34.89);
\draw [black] (60.63,-41.03) -- (53.77,-33.87);
\fill [black] (53.77,-33.87) -- (53.97,-34.79) -- (54.69,-34.1);
\draw [black] (51.7,-28.7) -- (51.7,-19.7);
\fill [black] (51.7,-19.7) -- (51.2,-20.5) -- (52.2,-20.5);
\draw [black] (71.62,-28.94) -- (67.58,-19.46);
\fill [black] (67.58,-19.46) -- (67.43,-20.39) -- (68.35,-20);
\draw [black] (73.68,-28.83) -- (76.52,-19.57);
\fill [black] (76.52,-19.57) -- (75.81,-20.19) -- (76.76,-20.48);
\draw [black] (43.5,-43.2) -- (43.5,-16.7);
\end{tikzpicture}
    \caption{\centering A tree (left) and its mirror tree (right). The lexicographical enumeration of the mirror tree is denoted by $u_0^{\div}, \ldots, u_6^{\div}$. The double circles indicate a marked vertex and its mirror image. The mirrored enumeration of the initial tree is denoted by $w_0, \ldots, w_6$.}
    \label{fig:mirroredEnumeration}
    \end{center}
\end{figure}
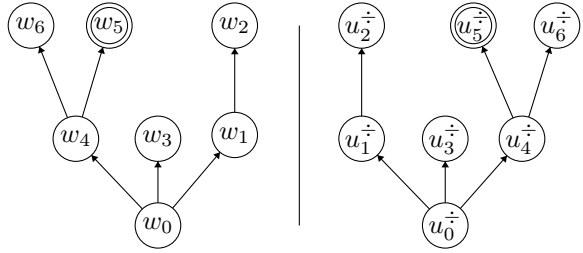

    Let us now consider the enumeration of  $(\rot\Tree_n)^{\circ}$ obtained by applying the correspondence $u \mapsto \widetilde u$ to $\left( w_{k}\right)_{1 \leq k\leq n-1}$, and let $H^*_n$ be such that $H^*_n(k)=|\widetilde w_{k}|$ for $1 \leq k \leq n-1$ and $H^*_n(0)=H^*_n(n)=0$. As usual $h^*_n$ denotes the associated time-scaled function. By  \eqref{eq:decHauteurLargeur} and \eqref{eq:MirrorLukaIsLeft}, we have that for all $1 \leq k \leq n-1$
    \begin{equation}\label{eq:ExactHeightPlusWidth}
        H^*_n(k)=H_{\Tree_n^{\div}}(k) + S_{\Tree_n^{\div}}(k)-1.
    \end{equation}
As a consequence, $\Vert h^*_n-h_{\Tree_n^{\div}} - s_{\Tree_n^{\div}}\Vert_{\infty}\leq 2$ and we deduce a scaling limit for $h^*_n$ (jointly with the encoding processes of $\Tree_n$) thanks to Proposition~\ref{prop:DuquesneMiroir}:

\begin{itemize}
    
\item When $\sigma^2=+\infty$, $h_{\Tree_n^{\div}}$ becomes negligible in front of $s_{\Tree_n^{\div}}$ since $n^{1-1/\alpha}/\ell(n)=o\bigl(n^{1/\alpha}\ell(n)\bigr)$ as $n\rightarrow+\infty$, hence we have
    \begin{equation}\label{eq:FirstCvInfinite}
    \Bigl(\frac{1}{\ell(n)n^{1/\alpha}} s_{\Tree_n}, \frac{\ell(n)}{n^{1-1/\alpha}}c_{\Tree_n},\frac{1}{\ell(n)n^{1/\alpha}} h^*_n \Bigr) \xrightarrow[\tend{n}]{\dstb} (\mathbbm x, \mathbbm h,\mathbbm y),
\end{equation}
where $\mathbbm y=F^{\div}(\mathbbm x)$.
\item When $\sigma^2<+\infty$, $h_{\Tree_n^{\div}}$ becomes proportional to $s_{\Tree_n^{\div}}$, hence we have
\begin{equation}\label{eq:FirstCvFinite}
    \frac{1}{\sqrt n}\Bigl(s_{\Tree_n},c_{\Tree_n},h^*_n\Bigr) \xrightarrow[\tend{n}]{\dstb} (\sigma \mathbbm e, \frac{2}{\sigma} \mathbbm e,\frac{2+\sigma^2}{\sigma}\widehat {\mathbbm e}).
\end{equation}
\end{itemize}
Note that we use the following facts: $\ell(n)/n^{1-1/\alpha}\times\Vert c_{\Tree_n}- h_{\Tree_n}\Vert_{\infty}  \rightarrow 0$ in probability (Theorem~\ref{thm:Duquesne}), and the addition of functions is not continuous in general for $d_{M_1}$ but is continuous at any point $(x,y)$ such that $x$ or $y$ is in $C([0,1])$.

We now relate $h_n^*$ to the contour function of $\rot\Tree_n$. First, Lemma~\ref{lem:RightHeightAndRightContourM1} gives that once rescaled $\widehat {h^*_n}$ is roughly the same as $c_{(\rot\!\Tree_n)^{\circ}}$. But according to Lemma~\ref{lem:HeightIsRoughlyContour}, since $\max H_{(\rot\!\Tree_n)^{\circ}}=\max H_n^*$ is negligible in front of $n$, we have that $c_{(\rot\!\Tree_n)^{\circ}}$ is almost the same as $h_{(\rot\!\Tree_n)^{\circ}}$. By Lemma~\ref{lem:AddingTheLeaves} this last process is also roughly the same as $h_{\rot\!\Tree_n}$, so it is roughly the same as $c_{\rot\!\Tree_n}$ by  Lemma~\ref{lem:HeightIsRoughlyContour} again. We finally get
\begin{equation*}
    d_{M_1}\left(\frac{1}{\ell(n)n^{1/\alpha}} c_{\rot\!\Tree_n}, \frac{1}{\ell(n)n^{1/\alpha}} \widehat {h^*_n}\right) \xrightarrow[\tend{n}]{\prob} 0.
\end{equation*}

Thus \eqref{eq:FirstCvInfinite} and \eqref{eq:FirstCvFinite} still hold if we replace $h^*_n$ with $c_{\rot\!\Tree_n}$, as long as we apply $x\mapsto \widehat x$ to the corresponding limit. Finally, we also know that, up to a last step down, $C_{\Tree_n}$ and $S_{\rot\!\Tree_n}$ coincide, so by convergence of $c_{\Tree_n}$ we get
\begin{equation*}
    \frac{\ell(n)}{n^{1-1/\alpha}}\bigl\Vert c_{\Tree_n}-s_{\rot\!\Tree_n}\bigr\Vert_{\infty} \xrightarrow[\tend{n}]{\prob} 0.
\end{equation*}
As a consequence we get from \eqref{eq:FirstCvInfinite} and \eqref{eq:FirstCvFinite} the joint convergence of $s_{\Tree_n},c_{\Tree_n},s_{\rot\!\Tree_n}, c_{\rot\!\Tree_n}$ (rescaled as above). Lemma~\ref{lem:HeightIsRoughlyContour} already gives that their height and contour processes are roughly the same, so Theorem~\ref{thm:DetailedResult} is proved.
\end{proof}

\paragraph{Comments on the proof.}
\begin{itemize}
    \item If one is not interested in the dependence between the encoding processes of $\rot\Tree_n$ and $s_{\Tree_n}$ (for instance, if one simply wants to get \eqref{eq:OnlyContourFinite} and \eqref{eq:OnlyContourInfinite} to establish Theorem~\ref{thm:MainResult}), then there is no need to use $F^{\div}$ as \eqref{eq:FirstCvInfinite} and \eqref{eq:FirstCvFinite} without $s_{\Tree_n}$ may be deduced from Theorem~\ref{thm:Duquesne} applied to $\bigl(\Tree_n^{\div}\bigr)_n$ instead of its extension Proposition~\ref{prop:DuquesneMiroir}.
    
    \item  In the special case $\mu=\Geom(1/2)$, a form of Theorem~\ref{thm:DetailedResult} has already been proved by \textsc{Marckert} in \cite{Marckert2004Rotation}. Let us stress the similarities and differences between his proof and ours. \textsc{Marckert} presents in his paper the geometric construction of $\rot$ which gives the handy identification between $T\!\setminus\!\{\varnothing\}$ and the internal subtree $(\rot T)^{\circ}$, and he also introduces the key idea of using $h_{(\rot\! \Tree_n)^{\circ}}$ as an intermediate process to link $h_{\Tree_n}$ and $h_{\rot\! \Tree_n}$. Our strategy for a general proof of Theorem~\ref{thm:DetailedResult} is based on this idea, with some adaptations such as replacing $h_{(\rot\! \Tree_n)^{\circ}}$ by $h_n^*$ as we actually have to link $h_{\rot\! \Tree_n}$ with $s_{\Tree_n^{\div}}$ as soon as $\sigma^2=+\infty$. However, we use quite different methods to prove that the several processes are asymptotically close to each other. Indeed, in this special case $\rot\Tree_n$ is uniform over binary trees with $n$ leaves and thus is distributed as a conditioned \textsc{Bienaymé} tree with offspring distribution $(\delta_0+\delta_2)/2$. This enables \textsc{Marckert} to use some previous results from \cite{MarckertMokkadem2003SameExcursion,MarckertMokkadem2003InternalStructureBGW} which describe in details the internal structure of conditioned critical \textsc{Bienaymé} trees (with some exponential moment assumption), and it immediately gives that the internal height process of $\rot\Tree_n$, namely $h_{(\rot\! \Tree_n)^{\circ}}$, is close to the global height process $h_{\rot\! \Tree_n}$, as well as the fact that the "right part" of $h_{(\rot\! \Tree_n)^{\circ}}$ is close to its "left part" which is $h_{\Tree_n}$. Unfortunately, $\mu=\Geom(1/2)$ is the only case where $\rot\Tree_n$ is distributed as a conditioned \textsc{Bienaymé} tree. We thus use another approach, as explained in Section~\ref{sct:Intro}. Its main specific features are that we express all quantities of interest with encoding processes of $\Tree_n$ (since the behaviour of this tree only is well known) and then we look for combinatorial relations between processes that hold for every plane tree in order to link the asymptotic behaviours, with respect to $d_{M_1}$, of those processes.
\end{itemize}

\subsection{Comparison with the co-rotation}\label{ssct:Comparison co-rotation}

Recall the companion correspondence of the rotation, $\corot$, introduced in Section~\ref{sct:rotation}. We explore briefly the similarities and mostly the differences between $\rot$ and $\corot$ applied on large \textsc{Bienaymé} trees in this section, and in particular we prove that the encoding processes of $\corot\Tree_n$ asymptotically behave in a nicer way. 

By \eqref{eq:Link Rot Corot}, which gives that $\corot$ is linked to $\rot$ through the mirror transformation, we see that $\corot \Tree_n$ is distributed as $(\rot \Tree_n)^{\div}$, hence it has the same distributional scaling limit. Thanks to Proposition~\ref{prop:DuquesneMiroir}, we can even use \eqref{eq:Link Rot Corot} to get a joint convergence as a corollary of Theorem~\ref{thm:DetailedResult}. We state this result as convergence of contour functions but it directly translates into convergence of trees. Recall the notation of Table~\ref{tab:knonwScalingLimits}.

\begin{cor}\label{cor:Comparison Rot Corot}
Let $\Tree_n$ be distributed as $\BGW\evt{\,\cdot\,\vert\textup{\#vertices = $n$}}$ where the offspring distribution $\mu$ is critical and attracted to a stable distribution of index $\alpha\in(1,2]$. 

\begin{itemize}
    \item When $\alpha = 2$,
 \begin{equation*}
\frac{1}{\ell(n)\sqrt{n}}\bigl\Vert c_{\rot\! \Tree_n}- c_{\corot\! \Tree_n}\bigr\Vert_{\infty}  \xrightarrow[\tend{n}]{\prob} 0.
    \end{equation*}

\item When $\alpha < 2$, we have the convergence in distribution
\begin{equation*}
    \left(\frac{\ell(n)}{n^{1-1/\alpha}}c_{\Tree_n},\frac{1}{\ell(n)n^{1/\alpha}}c_{\rot\! \Tree_n}, \frac{1}{\ell(n)n^{1/\alpha}}c_{\corot \! \Tree_n}\right) \xrightarrow[\tend{n}]{\dstb} \bigl(\mathbbm h,\widehat{\mathbbm y},\mathbbm x\bigr)
\end{equation*}
with $\mathbbm y = F^{\div}(\mathbbm x)$.
\end{itemize}
\end{cor}

\begin{proof}
    First, we assume that $\alpha=2$. We first get from Theorem~\ref{thm:DetailedResult} that
    \begin{equation*}
        \frac{1}{\ell(n)\sqrt{n}}\bigl\Vert c_{\rot\! \Tree_n}-(1+\frac{2}{\sigma^2})s_{\Tree_n}\bigr\Vert_{\infty} \xrightarrow[\tend{n}]{\prob} 0,
    \end{equation*}
    with $2/(+\infty)=0$ by convention. The same holds when one replaces $\Tree_n$ by $\Tree_n^{\div}$, and by \eqref{eq:Link Rot Corot} we have $c_{\rot( \Tree_n^{\div})} = \widehat c_{\corot\! \Tree_n}$. The result thus follows from Proposition~\ref{prop:DuquesneMiroir} as it gives
    \begin{equation*}
\frac{1}{\ell(n)\sqrt{n}}\bigl\Vert \widehat s_{\Tree_n^{\div}}-s_{\Tree_n}\bigr\Vert_{\infty} \xrightarrow[\tend{n}]{\prob} 0.
    \end{equation*}

\smallskip

    We now deal with the case $\alpha<2$. The argument is rather similar: it is sufficient to prove that 
    \begin{equation}\label{eq:Contour Rot et Luka Miroir}
        d_{M_1}\left(\frac{1}{\ell(n)n^{1/\alpha}}\widehat c_{\rot\! \Tree_n},\frac{1}{\ell(n)n^{1/\alpha}}s_{\Tree_n^{\div}} \right) \xrightarrow[\tend{n}]{\prob} 0.
    \end{equation}
    Indeed, \eqref{eq:Contour Rot et Luka Miroir}  applied to $\Tree_n^{\div}$ instead of $\Tree_n$ gives
    \begin{equation}\label{eq:Contour corot et Luka}
        d_{M_1}\left(\frac{1}{\ell(n)n^{1/\alpha}}c_{\corot\! \Tree_n},\frac{1}{\ell(n)n^{1/\alpha}}s_{\Tree_n} \right) \xrightarrow[\tend{n}]{\prob} 0.
    \end{equation}
     By combining Proposition~\ref{prop:DuquesneMiroir} with \eqref{eq:Contour Rot et Luka Miroir}, \eqref{eq:Contour corot et Luka} and the fact that $\ell(n)/n^{1-1/\alpha}\times \Vert c_{\Tree_n}- h_{\Tree_n}\Vert_{\infty}  \xrightarrow{\prob} 0$, we get the desired result.

     Notice that \eqref{eq:Contour Rot et Luka Miroir} has already been implicitly proved during the proof of Theorem~\ref{thm:DetailedResult} (both terms in \eqref{eq:Contour Rot et Luka Miroir} have been compared to $h_n^*$). We may also deduce it directly from Theorem~\ref{thm:DetailedResult} and Proposition~\ref{prop:DuquesneMiroir}: According  to these results, the sequence $$\left(\frac{\ell(n)}{n^{1-1/\alpha}}c_{\Tree_n},\frac{1}{\ell(n)n^{1/\alpha}}\widehat c_{\rot\! \Tree_n}, \frac{1}{\ell(n)n^{1/\alpha}}s_{\Tree_n^{\div}}\right)_n $$ is tight and any subsequential distributional limit must be distributed as $\bigl(\mathbbm x,F^{\div}(\mathbbm x),F^{\div}(\mathbbm x)\bigr)$. This concludes the proof.
    \end{proof}

\begin{rmk}
    Simulations seem to indicate that the corresponding joint scaling limits for rotated and co-rotated trees, namely $\bigl(\mathscr T_{F^{\div}(\mathbbm x)},\mathscr T_{\mathbbm x}\bigr)$, is such that almost surely $\mathscr T_{F^{\div}(\mathbbm x)} \neq \mathscr T_{\mathbbm x}$ when $\alpha<2$, but we do not have a proof for this claim. In any case, these two trees are still closely related, as $\mathscr L_{F^{\div}(\mathbbm x)}= \mathscr L_{\mathbbm x}$ because of \cite[Theorem~4.1]{CurienKortchemski2014StableLooptrees} and the fact that $\Loop(T)=\Loop(T^{\div})$ once we see them as metric spaces (that is, once we forget about the ordered structure of these graphs). According to the discussion in Section~\ref{ssct:description limit tree}, $\mathscr T_{\mathbbm x}$ and $\mathscr T_{F^{\div}(\mathbbm x)}$ are two spanning $\R$-trees of $\mathscr L_{\mathbbm x}$. We may think of this as a continuous analogue of Figure~\ref{fig:Rot in Looptree}. 
\end{rmk}

So far, $\rot \Tree_n$ and $\corot \Tree_n$ have a symmetric behaviour at large scale, as their contour processes satisfy a kind of duality relation at the limit. But this symmetry breaks down once we take into account their \Luka walks. Indeed, there is a combinatorial relation between $S_{\corot \Tree_n}$ and $S_{\Tree_n}$ that will be studied and used in Section~\ref{ssct:Luka of Corot}, with the same method previously used to compare some encoding processes, and it appears that these processes are asymptotically the same for $d_{M_1}$.

\begin{lem}\label{lem:ComparisonLukaCorot}
For any real sequence $(\lambda(n))_n$ such that $\lambda(n)\rightarrow 0$, we have
    \begin{equation*}
        d_{M_1}\bigl(\lambda(n) s_{\corot\!\Tree_n}, \lambda(n)s_{\Tree_n}\bigr) \xrightarrow[\tend{n}]{\prob} 0.
    \end{equation*}
\end{lem}

As a consequence, we are able to give an analogue of Theorem~\ref{thm:DetailedResult} \ie a joint convergence result of all encoding processes of $\Tree_n$ and $\corot\Tree_n$, but in this situation \textit{all encoding processes of  $\corot\Tree_n$ have the same order and converge towards the same limit, namely the limit of the \Luka walk of $\Tree_n$}.

\begin{thm}\label{thm:Detailed Result Corotation}
Let $\Tree_n$ be distributed as $\BGW\evt{\,\cdot\,\vert\textup{\#vertices = $n$}}$ where the offspring distribution $\mu$ is critical, has variance $\sigma^2 \leq +\infty$ and is attracted to a stable distribution with index $\alpha \in (1,2]$.
\begin{itemize}
\item Assuming $\sigma^2 <+\infty$, we have
    \begin{equation*}
        \frac{1}{\sqrt n}\left(s_{\Tree_n},c_{\Tree_n},s_{\corot\!\Tree_n},c_{\corot\! \Tree_n}\right) \xrightarrow[\tend{n}]{\dstb} \left(\sigma \mathbbm e,\frac{2}{\sigma}\mathbbm e, \sigma \mathbbm e,\frac{2+\sigma^2}{\sigma}\mathbbm e\right).
    \end{equation*}

\item Assuming $\sigma^2 =+\infty$, we have
    \begin{equation*}
        \left(\frac{1}{\ell(n)n^{1/\alpha}}s_{\Tree_n},\frac{\ell(n)}{n^{1-1/\alpha}}c_{\Tree_n}, \frac{1}{\ell(n)n^{1/\alpha}} s_{\corot\!\Tree_n},\frac{1}{\ell(n)n^{1/\alpha}}c_{\corot\! \Tree_n}\right) \xrightarrow[\tend{n}]{\dstb} \left(\mathbbm x,\mathbbm h,\mathbbm x,\mathbbm x\right).
    \end{equation*}
\end{itemize}
In addition, in both cases height functions are asymptotically the same as contour functions.
\end{thm}

\begin{proof}
    In the case $\sigma^2=+\infty$, the result follows from Theorem~\ref{thm:Duquesne} together with Lemma~\ref{lem:ComparisonLukaCorot}  and \eqref{eq:Contour corot et Luka} (which is valid as soon as $\sigma^2=+\infty$).

    In the case $\sigma^2<+\infty$, the argument is the same but \eqref{eq:Contour corot et Luka} does not hold anymore. We use instead Corollary~\ref{cor:Comparison Rot Corot} and Theorem~\ref{thm:DetailedResult} to get
   \begin{equation*}
        \frac{1}{\sqrt{n}}\bigl\Vert c_{\corot\! \Tree_n}-(1+\frac{2}{\sigma^2})s_{\Tree_n}\bigr\Vert_{\infty} \xrightarrow[\tend{n}]{\prob} 0.
    \end{equation*}
    Finally, height and contour processes are asymptotically close according to Lemma~\ref{lem:HeightIsRoughlyContour}.
\end{proof}

\begin{rmk}
    Theorem~\ref{thm:Detailed Result Corotation} gives an example where two families of trees, $(\Tree_n)_n$ and $(\corot\Tree_n)_n$, have asymptotically the same \Luka walk (for $d_{M_1}$) but significantly different scaling limits. On the other hand, Theorems~\ref{thm:DetailedResult} and~\ref{thm:Detailed Result Corotation} give a family of trees, $(\rot \Tree_n)_n$, whose lexicographic \Luka walk and mirrored \Luka walk have  dramatically different behaviours: $S_{\rot\! \Tree_n}$ is of order $n^{1-1/\alpha}/\ell(n)$ and converges in distribution towards the continuous function $\mathbbm h$ while $S_{(\rot\! \Tree_n)^{\div}}=S_{\corot (\Tree_n^{\div})}$ is of order $\ell(n)n^{1/\alpha}$ and converges towards the discontinuous function $F^{\div}(\mathbbm x)$. This comes from the presence of spines with a specific orientation (with respect to the planar order) in $\rot \Tree_n$.
\end{rmk}

\subsection{Comments on the method used in this paper}\label{ssct:comments}

\paragraph{On the use of \textsc{Skorokhod}'s $M_1$ topology.} As soon as we consider the case $\alpha < 2$, the use of \textsc{Skorokhod}'s $M_1$ topology is crucial for several steps. In particular, in Lemmas~\ref{lem:RightHeightAndRightContourM1} and~\ref{lem:ComparisonLukaCorot}, we compare the processes $H_n^*$ and $S_{\Tree_n}$ to the processes $C_{(\rot\!\Tree_n)^{\circ}}$ and $S_{\corot\!\Tree_n}$, and this shows that they have the same limits with respect to the $M_1$ topology. However the former processes have macroscopic jumps (\ie jumps comparable to the scale of the whole process) while the latter only have increments $\pm1$, thus we cannot match their jumps and those processes cannot converge toward the same limits with respect to the $J_1$ topology. In short, the $M_1$ topology gives sense to convergences with unmatched jumps, such as in Theorem~\ref{thm:DetailedResult} where the convergence of $c_{\rot\!\Tree_n}$ (whose increments are $\pm1$) toward a discontinuous function comes from series of microscopic jumps that merge together to form macroscopic jumps at the limit.

\smallskip

On the contrary, when $\alpha=2$ we can stick with the uniform convergence in $C([0,1])$, and in the case $\sigma^2<+\infty$ it is possible to simplify further the proof of Theorem~\ref{thm:DetailedResult}. The key observation is that in this setting $L(u)$ (as well as $R(u)$) is asymptotically proportional to $\abs{u}$, uniformly in $u \in \Tree_n$, hence the height of a rotated vertex is also proportional to the height of the corresponding initial vertex.
\begin{prop}\label{prop:HeightAfterRotationIsUniformlyProportionalToHeightBefore}
Suppose $\sigma^2<+\infty$, then
\begin{equation*}\label{eq:leftProportionalToHeight}
        \frac{1}{\sqrt n}\max_{u \in \Tree_n} \abs{L(u)-\frac{\sigma^2}{2}|u|}  \xrightarrow[\tend{n}]{\prob} 0.
    \end{equation*}
\end{prop}
\begin{proof}
Recall that we can exchange right and left by considering the mirror tree. It gives that
\begin{equation*}
    \max_{u \in \Tree_n} \abs{L(u)-\frac{\sigma^2}{2}|u|} = \max_{u' \in \Tree_n^{\div}}\abs{ R(u')-\frac{\sigma^2}{2} |u'|} \leq \big\Vert s_{\Tree_n^{\div}}-\frac{\sigma^2}{2} h_{\Tree_n^{\div}}\big\Vert_{\infty}.
\end{equation*}
Then we apply Theorem~\ref{thm:MMD} to $\bigl(\Tree_n^{\div}\bigr)_n$ to get the desired result.
\end{proof}
As a consequence of Proposition~\ref{prop:HeightAfterRotationIsUniformlyProportionalToHeightBefore}, \eqref{eq:decHauteurLargeur} and the fact that $\rot$ preserves the lexicographical order, we get
\begin{equation}\label{eq:Alternative proof finite variance}
        \frac{1}{\sqrt n}\bigl\Vert h_{(\rot\!\Tree_n)^{\circ}}-(1+\frac{\sigma^2}{2})h_{\Tree_n}\bigr\Vert_{\infty} \xrightarrow[\tend{n}]{\prob} 0.
    \end{equation}
Let us stress that one has to use the convergence in distribution of $(h_{\Tree_n})_n$  towards a \textit{continuous} function to get the above display with respect to $\norm{\cdot}_{\infty}$, and a proof that does not involve $d_{M_1}$ also requires the use of the modulus of continuity of this limit to control some perturbations induced by a change of time. Comparing processes with respect to $d_{M_1}$ enables one to get rid of those considerations, but requires parametric representations.

Theorem~\ref{thm:MMD} combined with \eqref{eq:Alternative proof finite variance} and Lemmas~\ref{lem:HeightIsRoughlyContour} and~\ref{lem:AddingTheLeaves} (which can also be proved without $d_{M_1}$ in this particular setting, with the technicalities mentioned above) lead to the joint convergence of $S_{\Tree_n}, H_{\Tree_n}, C_{\Tree_n}, H_{\rot\!\Tree_n}, C_{\rot\!\Tree_n}$. Since $S_{\rot\!\Tree_n}$ is (almost) $C_{\Tree_n}$ we get Theorem~\ref{thm:DetailedResult} in the case $\sigma^2<+\infty$.

One can also prove Lemma~\ref{lem:ComparisonLukaCorot} without $d_{M_1}$ in this particular case and thus get a proof of Theorem~\ref{thm:Detailed Result Corotation}, case $\sigma^2<+\infty$, that does not use $d_{M_1}$.

\paragraph{On possible generalizations.} The approach used here is not restricted to \textsc{Bienaymé} trees. We essentially need a family of random trees $(T_n)_n$ such that we have a joint scaling limit for $\bigl(s_{T_n},h_{T_n}\bigr)_n$ and in addition \begin{equation*}
     \frac{\max H_{T_n}+\max S_{T_n}+1}{\# T_n} \xrightarrow[\tend{n}]{\prob} 0.
\end{equation*}

To be precise, these assumptions are sufficient to study the joint behaviour of $T_n$ and $\corot T_n$ and get an analogue of Theorem~\ref{thm:Detailed Result Corotation}. In the case of the rotation, we rather need that $(T_n^{\div})_n$ satisfies these assumptions to study $(\rot T_n)_n$ and we must additionally understand the joint behaviour of the \Luka walks of  $T_n$ and $T_n^{\div}$ to get an analogue of Theorem~\ref{thm:DetailedResult}. In this paper, we relied on $T_n^{\div}$ for both the rotation and the co-rotation, hence we briefly sketch here how to study the joint behaviour of $T_n$ and $\corot T_n$ without any assumption on $T_n^{\div}$.

First, $(\corot T)^{\circ}$ and $T\!\setminus\!\{\varnothing\}$ are linked by an analogue of \eqref{eq:decHauteurLargeur} involving $R(u)$ instead of $L(u)$, thus in the very same way used to prove Theorem~\ref{thm:DetailedResult} we may define a non-standard height sequence $H_n^{**}$ for $(\corot T_n)^{\circ}$, based on the order corresponding to the lexicographical one of $T_n\!\setminus\!\{\varnothing\}$, such that $H_n^{**}=S_{T_n}+H_{T_n}$. It implies a joint scaling limit for $h_n^{**}$ and all encoding processes of $T_n$, without any consideration about $T_n^{\div}$. Moreover this order on $(\corot T_n)^{\circ}$ is the \textit{left} equivalent of the rightmost order defined for $(\rot T_n)^{\circ}$, thus Lemma~\ref{lem:RightHeightAndRightContourM1} can be directly adapted to get that $h_n^{**}$ and $c_{(\corot \!T_n)^{\circ}}$ are roughly the same. Lemmas~\ref{lem:HeightIsRoughlyContour} and~\ref{lem:ComparisonLukaCorot}  do not require any assumption about $T_n^{\div}$, the only part that must be adapted is the link between $(\corot T_n)^{\circ}$ and $\corot T_n$. Indeed, Lemma~\ref{lem:AddingTheLeaves} relies on the fact that $\rot$ preserves the lexicographical order, but $\corot$ only preserves the mirrored lexicographical order. However, since $\max H_{T_n^{\div}}=\max H_{T_n}$, we may apply Lemma~\ref{lem:AddingTheLeaves} (without modification) to $T_n^{\div}$, and thanks to \eqref{eq:Link Rot Corot} and Lemma~\ref{lem:HeightIsRoughlyContour} we get that $c_{(\corot \!T_n)^{\circ}}$ and $c_{\corot \!T_n}$ are roughly the same.

We may deduce, for instance, that the rotation acts as a dilation again on rotated \textsc{Bienaymé} trees: when $\sigma^2=+\infty$ we get 
\begin{equation*}
    \dGH\left(\frac{1}{\ell(n)n^{1/\alpha}} \rot^{\circ k} \Tree_n,\frac{k}{\ell(n)n^{1/\alpha}} \rot \Tree_n \right) \xrightarrow[\tend{n}]{\prob} 0.
\end{equation*}

\section{From combinatorics to comparisons under $M_1$}\label{sct:ProofLemmas}

In this section, we prove the several lemmas stating that some encoding processes, once rescaled, are close according to $d_{M_1}$. They all rely on the same method. We will first establish some combinatorial relations between those processes and then use this to build some convenient parametric representations that enable us to control the $d_{M_1}$ distance between them. Sections~\ref{ssct:HeightAndContour} and~\ref{ssct:Luka of Corot} are the simplest illustrations of this method, while Sections~\ref{ssct:rightmostEnumeration} and~\ref{ssct:AddingTheLeaves} require some additional work.

As the combinatorial relations will require manipulating different trees at the same time, we fix some notation: $u_0=\varnothing, u_1, \ldots, u_{n-1}$ will always denote the lexicographical enumeration of $\Tree_n$ while we will use $v_0=\varnothing, v_1, \ldots, v_{2n-2}$ for the lexicographical enumeration of $\rot\Tree_n$. Thus $\widetilde u_1,\ldots,\widetilde u_{n-1}$ is the lexicographical enumeration of $\bigl(\rot\Tree_n\bigr)^{\circ}$. 

We provide Table~\ref{tab:ProofPart1} to complete Table~\ref{tab:knonwScalingLimits} with a recap of the notation used throughout Section~\ref{sct:ProofLemmas}.

\begin{table}[htbp]
\caption{Table of important notation in Section~\ref{sct:ProofLemmas}.}
\centering
\begin{tabular}{c p{10cm}}
\toprule
 $\widetilde u$, for $u\in T\!\setminus\!\{\varnothing\}$&vertex of $(\rot T)^{\circ}$ corresponding to $u$ and called the rotated version of $u$ (defined in Section~\ref{sct:rotation})\\
 $L(u)$ (resp. $R(u)$), for $u \in T$&number of edges of $T$ grafted on $\Zintfo{\varnothing,u}$ on its \textit{left} (resp. \textit{right}) side (defined in Section~\ref{sct:planeTrees})\\

 \hline
 
 $u_0=\varnothing, u_1, \ldots, u_{n-1}$&lexicographical enumeration of $\Tree_n$\\
 $v_0=\varnothing, v_1, \ldots, v_{2n-2}$&lexicographical enumeration of $\rot\Tree_n$ \\
 $w_0=\varnothing, w_1, \ldots, w_{n-1}$&mirrored enumeration of $\Tree_n$ (defined in Section~\ref{ssct:mainResults})\\
 $r_1,r_2,\ldots,r_{n-1}$&rightmost enumeration of $\bigl(\rot\Tree_n\bigr)^{\circ}$ (defined in Section~\ref{ssct:rightmostEnumeration})\\

\hline
 $H^*_n$ and $h^*_n$&non-standard height process of $\bigl(\rot\Tree_n\bigr)^{\circ}$ and its time-scaled function (defined in Section~\ref{ssct:mainResults})\\
 \bottomrule
\end{tabular}
\label{tab:ProofPart1}
\end{table}



\subsection{Lemma~\ref{lem:HeightIsRoughlyContour}: height and contour processes are roughly the same}\label{ssct:HeightAndContour}

We prove here Lemma~\ref{lem:HeightIsRoughlyContour}, which will illustrate our method to control the $M_1$ distance between several encoding functions. 

Recall that Lemma~\ref{lem:HeightIsRoughlyContour} translates the following general argument into the framework of \textsc{Skorokhod}'s $M_1$ topology: for any sequence of trees $(T_n)_n$ such that $\#T_n\rightarrow+\infty$ and $\max H_{T_n}$ is negligible in front of $\#T_n$ as $n$ goes to $+\infty$, and for any renormalization $(\lambda(n))_n$ such that $\lambda(n)\rightarrow 0$, the rescaled height process $\lambda(n) h_{T_n}$ is roughly the same as the rescaled contour process $\lambda(n) c_{T_n}$ when $n$ is large.

\begin{proof}[Proof of Lemma~\ref{lem:HeightIsRoughlyContour}.]
We fix $(\lambda(n))_n$ such that $\lambda(n)\rightarrow 0$. First, we state a simple combinatorial relation between these processes. For $0 \leq k < \#T_n$, let $j_n(k)$ be the time needed to reach $u_k$ for the first time while following the contour of $T_n$ at unit speed. By counting edges used twice or once we see that 
\begin{equation}\label{eq:ExtractionHeightContour}
    \text{for } 0 \leq k < \#T_n,\ j_n(k)=2k-H_{T_n}(k).
\end{equation} 
We also set $j_n(\#T_n)=2\#T_n-2$ the total time needed to perform the contour (note that $j_n$ remains strictly increasing). By construction we have $C_{T_n}\bigl(j_n(k)\bigr)=H_{T_n}(k)$ for all $0 \leq k \leq \#T_n$, and the contour process $C_{T_n}$ is easily described in between those times as depicted by Figure~\ref{fig:LinkHeightContour}.  More precisely, for $0 \leq k \leq \#T_n-2$  either $H_{T_n}(k+1)=H_{T_n}(k)+1$ or $H_{T_n}(k+1) \leq  H_{T_n}(k)$ and we have:
\begin{itemize}
    \item In the first case, $j_n(k+1) = j_n(k)+1$ hence $C_{T_n}$ is already fully described on $\Zintff{j_n(k),j_n(k+1)}$.
    \item In the second case, $j_n(k+1) > j_n(k)+1$ and $C_{T_n}$ is affine on $\Zintff{j_n(k),j_n(k+1)-1}$ with $C_{T_n}\bigl(j_n(k)\bigr)=H_{T_n}(k)$ and slope $-1$. This entails that $C_{T_n}\bigl(j_n(k+1)-1\bigr)=H_{T_n}(k+1)-1$, so after this $C_{T_n}$ makes one step $+1$ to satisfy $C_{T_n}\bigl(j_n(k+1)\bigr)=H_{T_n}(k+1)$.
    \item Finally, $C_{T_n}$ is affine on $\Zintff{j_n(\#T_n-1),j_n(\#T_n)}$ with slope $-1$.
\end{itemize}  

\begin{figure}[h]\label{LinkHeightContour}
    \centering
    \includegraphics[width=\linewidth]{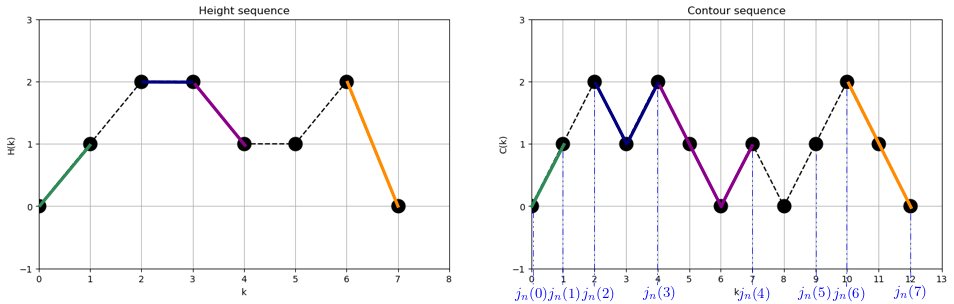}
    \caption{\centering Description of $C_{T_n}$ (on the right) based on $H_{T_n}$ (on the left) and the times $(j_n(k), 0\leq k\leq \#T_n)$, with $T_n$ being the tree from Figure~\ref{fig:lexicoEnumeration}. Several parts of the two processes have been identified by matching colours to illustrate the case $H_{T_n}(k+1)=H_{T_n}(k)+1$ (green part), the case $H_{T_n}(k+1) \leq  H_{T_n}(k)$ (blue and magenta parts) and the case of the last step of $H_{T_n}$ (orange part).}
    \label{fig:LinkHeightContour}
\end{figure}

Based on this, we now introduce two parametric representations of $h_{T_n}$ and $c_{T_n}$ that can be compared easily. For $c_{T_n}$ we simply take $(c_{T_n}, \id_{[0,1]})$. For $h_{T_n}$, we introduce two auxiliary sequences $\bigl(A_n(m)\bigr)_{0 \leq m\leq 2\#T_n-2}$ and $\bigl(B_n(m)\bigr)_{0 \leq m\leq 2\#T_n-2}$ by defining them on each interval $\Zintff{j_n(k),j_n(k+1)}$:

\begin{figure}[h]
    \centering
    \includegraphics[width=\linewidth]{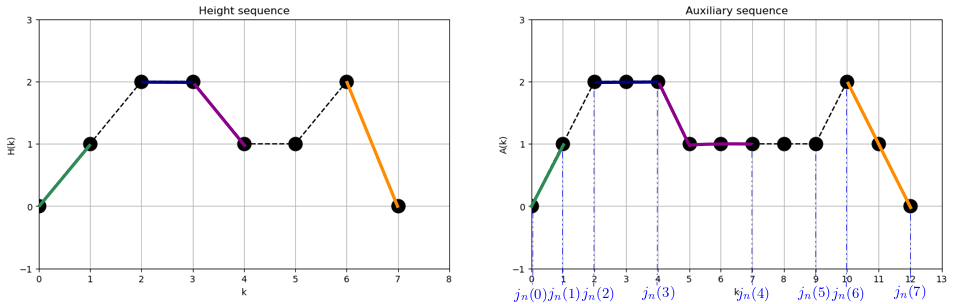}
    \caption{\centering Description of $A_n$ (on the right) based on $H_{T_n}$ (on the left) and the times $(j_n(k), 0\leq k\leq \#T_n)$, with $T_n$ being the tree from Figure~\ref{fig:lexicoEnumeration}. Several parts of the two processes have been identified by matching colours to illustrate the case $H_{T_n}(k+1)=H_{T_n}(k)+1$ (green part), the case $H_{T_n}(k+1) =  H_{T_n}(k)$ (blue part), the case $H_{T_n}(k+1) < H_{T_n}(k)$ (magenta part) and the case of the last step of $H_{T_n}$ (orange part).}
    \label{fig:AuxiliaryProcess}
\end{figure}
\begin{itemize}
    \item For $0 \leq k \leq \#T_n-2 \st H_{T_n}(k+1)\geq H_{T_n}(k)$, $A_n$ and $B_n$ are affine on $\Zintff{j_n(k),j_n(k+1)}$ with: 
    \begin{align*}
        & A_n\bigl(j_n(k)\bigr)=H_{T_n}(k);&\ & B_n\bigl(j_n(k)\bigr)= \frac{k}{\#T_n}; \\  
        & A_n\bigl(j_n(k+1)\bigr)=H_{T_n}(k+1);&\ & B_n\bigl(j_n(k+1)\bigr)= \frac{k+1}{\#T_n}.
    \end{align*}
    \item For $0 \leq k \leq \#T_n-2 \st H_{T_n}(k+1) <  H_{T_n}(k)$, $A_n$ and $B_n$ are affine on $\Zintff{j_n(k),j_n(k+1)-2}$ with:
    \begin{align*}
       & A_n\bigl(j_n(k)\bigr)=H_{T_n}(k);&\ & B_n\bigl(j_n(k)\bigr)= \frac{k}{\#T_n}; \\  
        & A_n\bigl(j_n(k+1)-2\bigr)=H_{T_n}(k+1);&\ & B_n\bigl(j_n(k+1)-2\bigr)= \frac{k+1}{\#T_n};
    \end{align*}
    and $A_n$ and $B_n$ are constant on $\Zintff{j_n(k+1)-2,j_n(k+1)}$.
    \item On $\Zintff{j_n(\#T_n-1),j_n(\#T_n)}$, $A_n$ and $B_n$ are affine with:
    \begin{align*}
        & A_n\bigl(j_n(\#T_n-1)\bigr)=H_{T_n}(\#T_n-1);&\ & B_n\bigl(j_n(\#T_n-1)\bigr)= \frac{\#T_n-1}{\#T_n}; \\  
        & A_n\bigl(j_n(\#T_n)\bigr)=H_{T_n}(\#T_n)=0;&\ & B_n\bigl(j_n(\#T_n)\bigr)= 1.
        \end{align*}
\end{itemize}
By construction of $A_n$ and $B_n$, their time-scaled functions $a_n,b_n$ are such that $(a_n,b_n)$ is a parametric representation of $h_{T_n}$ (Figure~\ref{fig:AuxiliaryProcess} illustrates the fact that $a_n$ is a valid spatial component). Moreover, $A_n(m) \neq C_{T_n}(m)$ if and only if $m=j_n(k+1)-1$ for some  $k \leq \#T_n-2$ such that $j_n(k+1) > j_n(k)+1$, and in this situation $A_n(m)= C_{T_n}(m)+1$ (compare Figures~\ref{fig:LinkHeightContour} and~\ref{fig:AuxiliaryProcess} to see this). Thus $\norm{a_n- c_{\Tree_n}}_{\infty}\leq1$ and we get \begin{equation*}
    d_{M_1}\bigl(\lambda(n) h_{T_n}, \lambda(n) c_{T_n}\bigr) \leq \lambda(n) \vee \norm{\id_{[0,1]}-b_n}_{\infty}.
\end{equation*}
Now we control the temporal components thanks to the construction of $B_n$ and \eqref{eq:ExtractionHeightContour}:
\begin{equation*}
    \norm{\id_{[0,1]}-b_n}_{\infty} \leq \max_{1 \leq k\leq \#T_n-1}\abs{\frac{k}{\#T_n}-\frac{j_n(k)}{2\#T_n-2}}\vee\abs{\frac{k}{\#T_n}-\frac{j_n(k)-2}{2\#T_n-2}}  
    \leq \frac{\max H_{T_n} + O(1)}{\#T_n}.
\end{equation*}
By assumption, this last term goes to $0$ with high probability as $n$ goes to $+\infty$, so the result holds.
\end{proof}

\subsection{Lemma~\ref{lem:RightHeightAndRightContourM1} and the rightmost enumeration}\label{ssct:rightmostEnumeration}

In this section, we turn to the study of the order on $(\rot\Tree_n)^{\circ}$ given by $\widetilde w_1, \ldots, \widetilde w_{n-1}$ that arises in the proof of Theorem~\ref{thm:DetailedResult}. Our main goal is to establish Lemma~\ref{lem:RightHeightAndRightContourM1} in order to extract information from the scaling limit of  $h_n^*$. 

As in the previous section, we start with some combinatorial relations between $H_n^*$ and some encoding function of $(\rot\Tree_n)^{\circ}$. Unfortunately, the rotation does not preserve the mirrored order and we do not have any direct description of $H_{(\rot\!\Tree_n)^{\circ}}$ based on $H_n^*$. But we can still describe this new order given by $\widetilde w_1, \ldots, \widetilde w_{n-1}$ and use it to link $H^*_n$ with the contour process of $(\rot\Tree_n)^{\circ}$. More precisely, let $\widehat C_{(\rot\!\Tree_n)^{\circ}}$ be the \textit{reversed contour process} of $(\rot\Tree_n)^{\circ}$, obtained by following its contour from right to left instead of left to right (\ie $k \mapsto \widehat C_{(\rot\!\Tree_n)^{\circ}}(2n-4-k)$ is its classical contour process). Lemma~\ref{lem:RightHeightAndRightContour} gives a full description of $\widehat C_{(\rot\!\Tree_n)^{\circ}}$ based on $H_n^*$, illustrated by Figure~\ref{fig:SpecialHeightAndReversedContour}.
\begin{lem}\label{lem:RightHeightAndRightContour}
For all $0\leq k\leq n$, set 
    \begin{equation*}
        j_n(k)=\sum_{i=1}^k \abs{H^*_n(i)- H^*_n(i-1)}.
    \end{equation*}
First, $j_n$ is increasing from $j_n(0)=0$ to $j_n(n)=2n-4$. Moreover, for all $(k,m)$ such that $0\leq k\leq n-1$ and $j_n(k)\leq m\leq j_n(k+1)$, we have
\begin{equation*}
    \widehat C_{(\rot\!\Tree_n)^{\circ}}(m)=H^*_n(k)+\bigl(m-j_n(k)\bigr)\sgn\bigl(H^*_n(k+1)-H^*_n(k)\bigr).
\end{equation*}
In particular $H^*_n(k)=\widehat C_{(\rot\!\Tree_n)^{\circ}}(j_n(k))$ for all $0\leq k\leq n$.
\end{lem}

\begin{figure}[h]
    \centering
    \includegraphics[width=\linewidth]{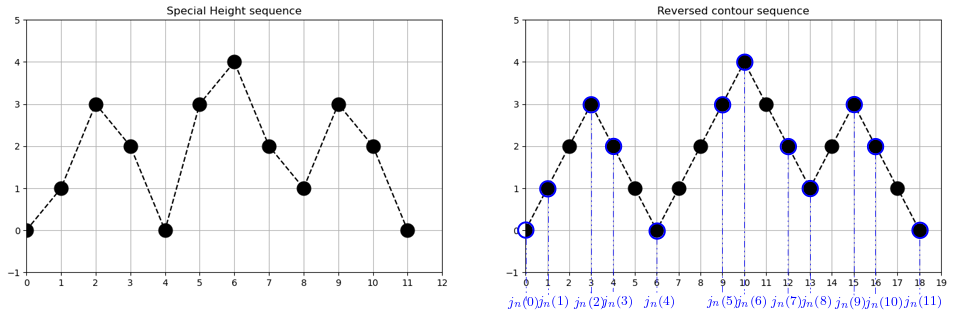}
    \caption{\centering The non-standard height process $H_n^*$ (on the left) and the reversed contour process $\widehat C_{(\rot\!\Tree_n)^{\circ}}$ (on the right), in the particular case of the internal subtree from Figure~\ref{fig:DefGeoRotation}. The blue extraction (on the right) shows the relation between these processes.}
    \label{fig:SpecialHeightAndReversedContour}
\end{figure}
To prove Lemma~\ref{lem:RightHeightAndRightContourM1}, we still have to control the change of time $j_n$ introduced in Lemma~\ref{lem:RightHeightAndRightContour} to control the $M_1$ distance between $h_n^*$ and $\widehat c_{(\rot\!\Tree_n)^{\circ}}$. Lemma~\ref{lem:SimplerExpressionforIndex} provides a more amenable expression for this change of time, based on further combinatorial considerations.

\begin{lem}\label{lem:SimplerExpressionforIndex}
For all $1\leq k\leq n-1$, we have
\begin{equation*}
        \sum_{i=1}^k \abs{H^*_n(i)- H^*_n(i-1)}=2k-1+S_{\Tree_n^{\div}}(k)-H_{\Tree_n^{\div}}(k).
    \end{equation*}
\end{lem}

With Lemmas~\ref{lem:RightHeightAndRightContour} and~\ref{lem:SimplerExpressionforIndex}, we are ready to prove the main result of this section, Lemma~\ref{lem:RightHeightAndRightContourM1}.
\begin{proof}[Proof of Lemma~\ref{lem:RightHeightAndRightContourM1}.]
We use the simple parametric representation $(h^*_n,\id_{[0,1]})$ for $h^*_n$. Recall the sequence $\bigl(j_n(k)\bigr)_{0 \leq k \leq n}$ from Lemma~\ref{lem:RightHeightAndRightContour} and let $\phi_n\in C([0,1])$ be such that 
\begin{equation*}
    \phi_n\!\left(\frac{k}{n}\right)=\frac{j_n(k)}{2n-4} \text{ for } 0 \leq k\leq n \text{ and } \phi_n \text{ affine on each segment } \left[\frac{k}{n},\frac{k+1}{n}\right].
\end{equation*}
According to Lemma~\ref{lem:RightHeightAndRightContour}, $\phi_n$ is non-decreasing and onto, and $\widehat c_{(\rot\!\Tree_n)^{\circ}}$ is affine on each segment $[j_n(k)/(2n-4),j_n(k+1)/(2n-4)]$ with $\widehat c_{(\rot\!\Tree_n)^{\circ}}\bigl(j_n(k)/(2n-4)\bigr)=h^*_n(k/n)$ for $0\leq k \leq n$, hence the (completed) graph of $\widehat c_{(\rot\!\Tree_n)^{\circ}}$ is the union of the segments in $\R^2$ linking $\bigl(h^*_n(k/n),j_n(k)/(2n-4)\bigr)$ to $\bigl(h^*_n((k+1)/n),j_n(k+1)/(2n-4)\bigr)$. Even though some of these segments may be reduced to a single point, it entails that $(h^*_n,\phi_n)$ is a parametric representation of $\widehat c_{(\rot\!\Tree_n)^{\circ}}$. For any real sequence $(\lambda(n))_n$, we can thus control $d_{M_1}\bigl(\lambda(n)\widehat c_{(\rot\!\Tree_n)^{\circ}}, \lambda(n)h^*_n\bigr)$ by controlling the temporal component $\phi_n$, as we have equal spatial components. This boils down to control $j_n$ and in light of Lemma~\ref{lem:SimplerExpressionforIndex} we get
\begin{align*}
    d_{M_1}\bigl(\lambda(n)\widehat c_{(\rot\!\Tree_n)^{\circ}}, \lambda(n)h^*_n\bigr) &\leq \norm{\id_{[0,1]}-\phi_n}_{\infty} \leq \max_{1 \leq k\leq n-1}\abs{\frac{k}{n}-\frac{j_n(k)}{2n-4}} \\
    &\leq \frac{\max S_{\Tree_n^{\div}} + \max H_{\Tree_n^{\div}} + O(1)}{n}.
\end{align*}

Since $n^{1-1/\alpha}/\ell(n)$ and $ \ell(n)n^{1/\alpha}$ are negligible in front of $n$,  Theorem~\ref{thm:Duquesne} applied to $\bigl(\Tree_n^{\div}\bigr)_n$ gives that the last term goes to $0$ in probability as $n$ goes to $+\infty$, so the result holds.
\end{proof}

We now establish the combinatorial relations given by Lemmas~\ref{lem:RightHeightAndRightContour} and~\ref{lem:SimplerExpressionforIndex}. Before giving any proof, let us introduce the adequate order on $(\rot\Tree_n)^{\circ}$. First, we equip the unary-binary tree $(\rot\Tree_n)^{\circ}$ with an additional structure inherited from its definition as a subtree of $\rot \Tree_n$: each edge will be either a \textit{left edge} or a \textit{right edge}. As seen in the geometric definition of  $\rot \Tree_n$ (see Figure~\ref{fig:DefGeoRotation}),  $(\rot\Tree_n)^{\circ}$ is naturally equipped with such a structure. More explicitly, identify $(\rot\Tree_n)^{\circ}$ with the internal vertices of the (full) binary tree $\rot \Tree_n$ and say that an edge is a \textit{left edge} if it points towards a first child in $\rot \Tree_n$ and a \textit{right edge} if it points towards a second child.

We now define the \textit{rightmost enumeration} $r_1,\ldots,r_{n-1}$ of $(\rot\Tree_n)^{\circ}$ in a recursive way, with a particle following the reversed contour of $(\rot\Tree_n)^{\circ}$ (see Figure~\ref{fig:rightmostenumeration}): 
\begin{itemize}
    \item \textit{Step $1$}: The particle starts at the root and travels every right edge it encounters (but no left edge). The last vertex reached is $r_1$.
    \item \textit{Step $k+1$} (where $k \leq n-2$): Suppose that we have defined $r_i$ for all $i \leq k$ using the particle which now sits at $r_k$.  If $r_k$ has a left child, then the particle goes through this left edge then travels every right edge it encounters (but no more left edge), and the last vertex reached is $r_{k+1}$. Otherwise, the particle goes back up the ancestral line of $r_{k}$ until it finds a vertex that does not belong to $\{r_1,\ldots,r_k\}$, and this vertex becomes $r_{k+1}$.
\end{itemize}

\begin{figure}[h]
    \centering
    \includegraphics[width=0.5\linewidth]{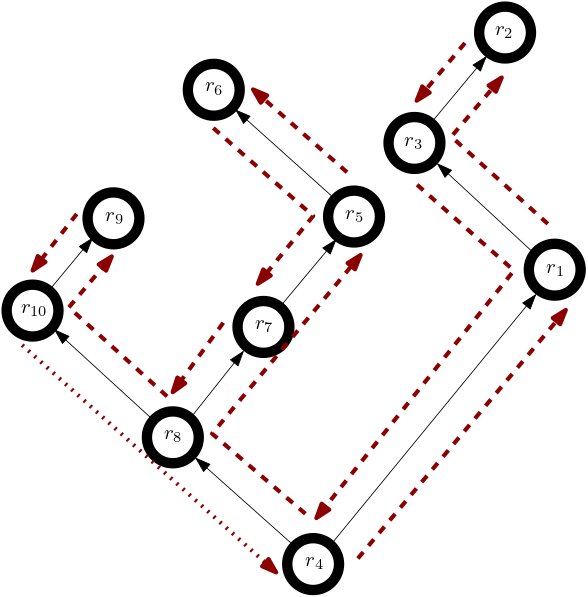}
    \caption{\centering The rightmost enumeration of the internal subtree from Figure~\ref{fig:DefGeoRotation}. The dashed arrows indicate the several steps for the particle used in the definition, and the dotted one is the additional step needed to complete the reversed contour.}
    \label{fig:rightmostenumeration}
\end{figure}
One can prove, by induction on the height of the tree, that the particle indeed follows the reversed contour of $(\rot\Tree_n)^{\circ}$ when those steps are realized (with an additional step to come back to the root) as depicted in Figure~\ref{fig:rightmostenumeration}. The same induction could prove that $r_1,\ldots,r_{n-1}$ truly is an enumeration of $(\rot\Tree_n)^{\circ}$, but this fact also follows from our main observation:
\begin{lem}\label{lem:RotatedMirrorIsRighmost}
    For all $1\leq i\leq n-1,\ \widetilde w_i = r_i$.
\end{lem}
\begin{proof}
      Lemma~\ref{lem:RotatedMirrorIsRighmost} can be proved by induction on $i$. Indeed, when one applies $u\mapsto \tilde u$, "go from $w_i$ to its last child (if there is any)" translates into "starting from $\widetilde w_i$, go through its left edge (if there is any) then cross every right edge encountered (but no more left edge)" in the rotated tree while "go from $w_i$ to its closest elder sibling if it exists, else go to its parent" translates into "go from $\widetilde w_i$ to its parent". Formal details are left to the reader.
\end{proof}

We can now use this link between the rightmost enumeration of $(\rot\Tree_n)^{\circ}$ and its reversed contour process $\widehat C_{(\rot\!\Tree_n)^{\circ}}$ to prove Lemma~\ref{lem:RightHeightAndRightContour}.

\begin{proof}[Proof of Lemma~\ref{lem:RightHeightAndRightContour}.] 
Let $x_0=\varnothing,x_1,\ldots,x_{2n-4}=\varnothing$ be the sequence of adjacent vertices obtained by following the reversed contour. In the definition of the rightmost enumeration, for all $1\leq k \leq n-1$ the particle  either goes straight up or straight down at step $k$, hence it crosses $\abs{|r_k|-|r_{k-1}|}$ edges during this step (with $r_0=\varnothing$ to include the case $k=1$). Consider that each edge is split into two directed edges (one upward and one downward). By the previous Lemma~\ref{lem:RotatedMirrorIsRighmost}, $j_n(k)$ simply is the number of directed edges crossed by the particle after performing \textit{Steps $1$ to $k$}. Since the particle follows the reversed contour, we get that  $r_k=x_{j_n(k)}$ hence $H^*_n(k)=\widehat C_{(\rot\!\Tree_n)^{\circ}}(j_n(k))$ for all $1\leq k \leq n-1$. This also holds for $k=0$ as $j_n(0)=0$. By considering the additional step that goes from $r_{n-1}$ to $\varnothing$, the particle crosses $|r_{n-1}|=|H^*_n(n)- H^*_n(n-1)|$ edges to complete the contour thus $j_n(n)$ is indeed $2n-4$ \ie the total number of edges crossed by the particle and $H^*_n(n)=0=\widehat C_{(\rot\!\Tree_n)^{\circ}}(j_n(n))$. The proposed expression for $\widehat C_{(\rot\!\Tree_n)^{\circ}}(m)$ follows from the observation that $\widehat C_{(\rot\!\Tree_n)^{\circ}}$ is monotone on each interval $\Zintff{j_n(k),j_n(k+1)}$ (because it corresponds to \textit{Step $k+1$}), and on this interval it goes from $H^*_n(k)$ to $H^*_n(k+1)$ with increments of constant absolute size $1$. Note that this holds even if $\Zintff{j_n(k),j_n(k+1)}$ is degenerate \ie $j_n(k)=j_n(k+1)$, since this happens if and only if $ H^*_n(k+1)= H^*_n(k)$ (and it may only happen for $k=0$ or $k=n-1$). 
\end{proof}

A refinement of the same argument, based on counting the number of edges crossed in the upward direction by two different means, also yields Lemma~\ref{lem:SimplerExpressionforIndex}:

\begin{proof}[Proof of Lemma~\ref{lem:SimplerExpressionforIndex}.]
As explained above, $j_n(k)=\sum_{i=1}^k \bigl|H^*_n(i)- H^*_n(i-1)\bigr|$ is the number of directed edges crossed by the particle after \textit{Steps $1$ to $k$} of the construction of the rightmost enumeration of $(\rot\Tree_n)^{\circ}$. Let us focus on the edges crossed in the upward direction (\ie in the direction going away from the root). Since the particle either goes straight up or straight down at each step, the number of edges crossed in the upward direction after \textit{step $k$} is $\sum_{i=1}^k \bigl(H^*_n(i)- H^*_n(i-1)\bigr)_+$. However we also notice that at step $k+1$ with $k\geq 1$, the particle goes up if and only if its initial vertex $r_k=\widetilde w_k$ has a left child, which is equivalent to $d_{w_k}(\Tree_n)=S_{\Tree_n^{\div}}(k+1)-S_{\Tree_n^{\div}}(k)+1 \geq 1$, and then it crosses  $1$ left edge and $S_{\Tree_n^{\div}}(k+1)-S_{\Tree_n^{\div}}(k)$ right edges. At the first step, the particle just crosses $S_{\Tree_n^{\div}}(1)$ right edges. When the particle goes straight down, it still crosses $0=S_{\Tree_n^{\div}}(k+1)-S_{\Tree_n^{\div}}(k)+1$ edges in the upward direction, hence we get the following for $k\geq 1$
\begin{equation*}
    \sum_{i=1}^k \bigl(H^*_n(i)- H^*_n(i-1)\bigr)_+=k+ S_{\Tree_n^{\div}}(k)-1.
\end{equation*}
We can now derive our simpler expression for $j_n(k)$ for $1\leq k\leq n-1$ thanks to \eqref{eq:ExactHeightPlusWidth}:
\begin{equation*}
\begin{split}
    \sum_{i=1}^k \abs{H^*_n(i)- H^*_n(i-1)}
    &=2\sum_{i=1}^k \bigl(H^*_n(i)- H^*_n(i-1)\bigr)_+ -  H^*_n(k) \\
    &=2k-1+S_{\Tree_n^{\div}}(k)-H_{\Tree_n^{\div}}(k).
\end{split}
\end{equation*}

\end{proof}

\subsection{Lemma~\ref{lem:AddingTheLeaves}: adding the leaves}\label{ssct:AddingTheLeaves}
The purpose of this section is to prove Lemma~\ref{lem:AddingTheLeaves}, which states that the height processes of $\rot\Tree_n$ and $(\rot\Tree_n)^{\circ}$, once rescaled, are also roughly the same. As before, we first need a combinatorial lemma to relate these processes, illustrated by Figure~\ref{fig:HeightWithWithoutLeaves}.

\begin{lem}\label{lem:internalExtraction}
For all $0\leq k \leq n-1$, set $j_n(k)=2k-H_{\Tree_n}(k+1)+1$. It is strictly increasing from $j_n(0)=0$ to $j_n(n-1)=2n-1$, and for all $0\leq k \leq n-2$, $H_{\rot\!\Tree_n}$ satisfies on $\Zintff{j_n(k),j_n(k+1)}$:
\begin{align*}
    \text{\textbullet When } j_n(k+1) = j_n(k)+1: & \\
    H_{\rot\!\Tree_n}\bigl(j_n(k)\bigr)=H_{(\rot\!\Tree_n)^{\circ}}(k) &\text{ and } H_{\rot\!\Tree_n}\bigl(j_n(k)+1\bigr)=H_{(\rot\!\Tree_n)^{\circ}}(k+1); \\
    \text{\textbullet When } j_n(k+1) = j_n(k)+2: & \\
    H_{\rot\!\Tree_n}\bigl(j_n(k)\bigr)=H_{(\rot\!\Tree_n)^{\circ}}(k) &\text{ and } H_{\rot\!\Tree_n}\bigl(j_n(k)+1\bigr)=H_{\rot\!\Tree_n}\bigl(j_n(k)+2\bigr)=H_{(\rot\!\Tree_n)^{\circ}}(k+1); \\
    \text{\textbullet When } j_n(k+1) > j_n(k)+2: & \\
    H_{\rot\!\Tree_n}\bigl(j_n(k)\bigr)=H_{(\rot\!\Tree_n)^{\circ}}(k) &,\  H_{\rot\!\Tree_n}\bigl(j_n(k)+1\bigr)=H_{\rot\!\Tree_n}\bigl(j_n(k)+2\bigr)=H_{(\rot\!\Tree_n)^{\circ}}(k)+1, \\
    \text{then } H_{\rot\!\Tree_n} \text{ strictly decreases} &\text{ on } \Zintff{j_n(k)+2,j_n(k+1)} \text{ with } H_{\rot\!\Tree_n}\bigl(j_n(k+1)\bigr)=H_{(\rot\!\Tree_n)^{\circ}}(k+1).
\end{align*}
\end{lem}

\begin{figure}[h]
    \centering
    \includegraphics[width=\linewidth]{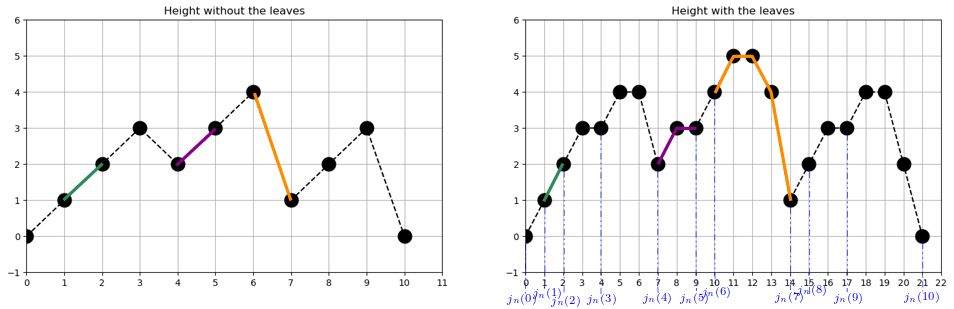}
    \caption{\centering Description of $H_{\rot\!\Tree_n}$ (on the right) based on $H_{(\rot\!\Tree_n)^{\circ}}$ (on the left) and the times $(j_n(k), 0\leq k\leq n-1)$, in the particular case of the tree from Figure~\ref{fig:DefGeoRotation}. Several parts of the two processes have been identified by matching colours to illustrate the case $j_n(k+1) = j_n(k)+1$ (green part), the case $j_n(k+1) = j_n(k)+2$ (magenta part) and finally the case $j_n(k+1) > j_n(k)+2$ (orange part).}
\label{fig:HeightWithWithoutLeaves}
\end{figure}

\begin{proof}
Recall from Section~\ref{sct:rotation} that $u\mapsto \widetilde u$ preserves the lexicographical order and that we have the following property for the leaves in $\rot\Tree_n$: a vertex is a left leaf in $\rot\Tree_n$ if and only if its parent is $\widetilde u$ for $u$ a leaf in $\Tree_n$, while it is a right leaf if and only if its parent is $\widetilde u$ for $u$ in $\Tree_n$ which is the last child of its own parent. 
We now prove that 
\begin{equation}\label{eq:ExtractionInternalNodes}
    \text{For } 0\leq k \leq n-2, \widetilde u_{1+k}=v_{j_n(k)},
\end{equation} 
by analysing the number of leaves in $\rot\Tree_n$ between two consecutive internal vertices, say $\widetilde u_{m}$ and $\widetilde u_{m+1}$: 
\begin{itemize}
    \item When $H_{\Tree_n}(m+1)=H_{\Tree_n}(m)+1$,  $u_{m+1}$ is the first child of $u_{m}$ in $\Tree_n$, so $\widetilde u_{m+1}$ is also the first child of $\widetilde u_{m}$ in $\rot\Tree_n$ and there is no leaf in between.
    \item When $H_{\Tree_n}(m+1)\leq H_{\Tree_n}(m)$, $u_{m}$ is a leaf and there are $H_{\Tree_n}(m)-H_{\Tree_n}(m+1)$ last children among $u_{m}$ and its ancestors that are not ancestors of $u_{m+1}$, so there is one left leaf and then $H_{\Tree_n}(m)-H_{\Tree_n}(m+1)$ right leaves between $\widetilde u_{m}$ and $\widetilde u_{m+1}$.
\end{itemize}

As a consequence, there are always $1+H_{\Tree_n}(m)-H_{\Tree_n}(m+1)\geq 0$ leaves between $\widetilde u_{m}$ and $\widetilde u_{m+1}$. Since $\widetilde u_1=v_0=v_{j_n(0)}$ and for all $0\leq k \leq n-2$, $j_n(k+1)-j_n(k)=2+H_{\Tree_n}(k+1)-H_{\Tree_n}(k+2)$, $j_n$ is strictly increasing and \eqref{eq:ExtractionInternalNodes} holds. This entails that $H_{\rot\!\Tree_n}\bigl(j_n(k)\bigr)=H_{(\rot\!\Tree_n)^{\circ}}(k)$ for those $k$, and it also holds for $k=n-1$. 

Moreover, $j_n(k+1)=j_n(k)+2$ if and only if $H_{\Tree_n}(k+1)=H_{\Tree_n}(k+2)$, which means that in this situation the left child of $\widetilde u_{1+k}$ is a leaf while its right child is $\widetilde u_{2+k}$. In particular we have $H_{\rot\!\Tree_n}\bigl(j_n(k)+1\bigr)=H_{\rot\!\Tree_n}\bigl(j_n(k)+2\bigr)=H_{(\rot\!\Tree_n)^{\circ}}(k+1)$.

Finally, $j_n(k+1)>j_n(k)+2$ if and only if $H_{\Tree_n}(k+1) > H_{\Tree_n}(k+2)$. In this case, both children of $\widetilde u_{1+k}$ are leaves, hence $H_{\rot\!\Tree_n}\bigl(j_n(k)+1\bigr)=H_{\rot\!\Tree_n}\bigl(j_n(k)+2\bigr)=H_{(\rot\!\Tree_n)^{\circ}}(k)+1$, and $\widetilde u_{2+k}$ must be the right child of an ancestor of $\widetilde u_{1+k}$ (except in the case $k=n-2$ where $\widetilde u_{1+k}$ is the last internal vertex). The $H_{\Tree_n}(k+1) - H_{\Tree_n}(k+2)-1$ remaining right leaves must be grafted on  vertices among  $\{$ancestors~of~$\widetilde u_{1+k}\}\!\setminus\!\{$ancestors~of~$\widetilde u_{2+k}\}$ (or simply among $\{$ancestors~of~$\widetilde u_{1+k}\}$ in the special case $k=n-2$), so they have distinct height and since they are consecutive leaves  the lexicographical order sort them in decreasing order for their height. In addition, in the case $k<n-2$, $\widetilde u_{2+k}$ is the right child of a vertex $v\in\{$ancestors~of~$\widetilde u_{1+k}\}$ while those leaves have the left child of $v$ as an ancestor, hence their height must be larger than $|v|+1=H_{(\rot\!\Tree_n)^{\circ}}(k+1)$. This is also satisfied for $k=n-2$ as $H_{(\rot\!\Tree_n)^{\circ}}(k+1)=0$ in this situation.
\end{proof}

We now prove Lemma~\ref{lem:AddingTheLeaves} with  the usual method: we build closely related parametric representations based on the combinatorial description given by Lemma~\ref{lem:internalExtraction}.

\begin{proof}[Proof of Lemma~\ref{lem:AddingTheLeaves}.]
Fix a sequence $(\lambda(n))_n$ such that $\lambda(n)\rightarrow 0$. First, recall the sequence $j_n$ of Lemma~\ref{lem:internalExtraction} and for simplicity let $g_n(k)=j_n(k)/(2n-1)$. We work with the  parametric representation $(h_{\rot\!\Tree_n},\id_{[0,1]})$ for $h_{\rot\!\Tree_n}$, and we build a piecewise affine parametric representation $(\chi_n,\tau_n)$ for $h_{(\rot\!\Tree_n)^{\circ}}$. More precisely, on each segment $\bigl[g_n(k),g_n(k+1)\bigr]:$ 
    \begin{itemize}
    \item For $ k \st j_n(k+1)=j_n(k)+1$, $(\chi_n,\tau_n)$ is affine on $\bigl[g_n(k),g_n(k+1)\bigr]$ with 
    \begin{align*}
        & \chi_n\bigl(g_n(k)\bigr)=H_{(\rot\!\Tree_n)^{\circ}}(k);&\ & \tau_n\bigl(g_n(k)\bigr)= \frac{k}{n-1}; \\  
        & \chi_n\bigl(g_n(k+1)\bigr)=H_{(\rot\!\Tree_n)^{\circ}}(k+1);&\ & \tau_n\bigl(g_n(k+1)\bigr)= \frac{k+1}{n-1}.
    \end{align*}
    \item For $ k \st j_n(k+1)=j_n(k)+2$, $(\chi_n,\tau_n)$ is affine on $\bigl[g_n(k),(j_n(k)+1)/(2n-1)\bigr]$ with
    \begin{align*}
        & \chi_n\bigl(g_n(k)\bigr)=H_{(\rot\!\Tree_n)^{\circ}}(k);&\ & \tau_n\bigl(g_n(k)\bigr)= \frac{k}{n-1}; \\  
        & \chi_n\left(\frac{j_n(k)+1}{2n-1}\right)=H_{(\rot\!\Tree_n)^{\circ}}(k+1);&\ & \tau_n\left(\frac{j_n(k)+1}{2n-1}\right)= \frac{k+1}{n-1},
    \end{align*}
and $(\chi_n,\tau_n)$ is constant on $\bigl[(j_n(k)+1)/(2n-1),g_n(k+1)\bigr]$.

    \item For $ k \st j_n(k+1)>j_n(k)+2$, let $t^*\in\bigl(g_n(k),g_n(k+1)\bigr)$ be the unique solution of $h_{\rot\!\Tree_n}(t^*)=H_{(\rot\!\Tree_n)^{\circ}}(k)$, $(\chi_n,\tau_n)$ is constant on $\bigl[g_n(k),t^*\bigr]$ with  \begin{align*}
        & \chi_n\bigl(g_n(k)\bigr)=H_{(\rot\!\Tree_n)^{\circ}}(k);&\ & \tau_n\bigl(g_n(k)\bigr)= \frac{k}{n-1}. 
    \end{align*}
Now observe that $h_{\rot\!\Tree_n}$ induces a decreasing one-to-one map from $\bigl[t^*,g_n(k+1)\bigr]$ to $\bigl[H_{(\rot\!\Tree_n)^{\circ}}(k+1),H_{(\rot\!\Tree_n)^{\circ}}(k)\bigr]$. On $\bigl[t^*,g_n(k+1)\bigr]$, $(\chi_n,\tau_n)$ is given by:
\begin{align*}
    \chi_n(t)=h_{\rot\!\Tree_n}(t); &\ &\tau_n(t)=\frac{k}{n-1}+\frac{1}{n-1}\frac{H_{(\rot\!\Tree_n)^{\circ}}(k)-h_{\rot\!\Tree_n}(t)}{H_{(\rot\!\Tree_n)^{\circ}}(k)-H_{(\rot\!\Tree_n)^{\circ}}(k+1)}.
\end{align*}

\end{itemize}
Clearly, $\tau_n$ is continuous non-decreasing and one can check that $\chi_n=h_{(\rot\!\Tree_n)^{\circ}}\circ\tau_n$ on each segment $\bigl[g_n(k),g_n(k+1)\bigr]$, hence $(\chi_n,\tau_n)$ is a parametric representation of $h_{(\rot\!\Tree_n)^{\circ}}$. Moreover, in light of Lemma~\ref{lem:internalExtraction} we see that $\chi_n$ and $h_{\rot\!\Tree_n}$ may only differ on intervals of the form $\bigl(g_n(k),t^*\bigr)$. Consequently $\Vert\chi_n-h_{\rot\!\Tree_n}\Vert_{\infty} \leq 1$ and we get 
\begin{equation*}
        d_{M_1}\bigl(\lambda _n h_{(\rot\!\Tree_n)^{\circ}}, \lambda(n) h_{\rot\!\Tree_n}\bigr) \leq \lambda(n) \vee \norm{\tau_n-\id_{[0,1]}}_{\infty}.
    \end{equation*}
    But on each segment $\bigl[g_n(k),g_n(k+1)\bigr]$ we have
    \begin{equation*}
        |\tau_n(t)-t|\leq \frac{1}{n-1} + |\frac{j_n(k)}{2n-1}-\frac{k}{n-1}|\vee |\frac{j_n(k+1)}{2n-1}-\frac{k}{n-1}|. 
    \end{equation*}
    We can thus control the temporal component thanks to the expression of $j_n$ given in Lemma~\ref{lem:internalExtraction}.
 \begin{equation*}
    \norm{\id_{[0,1]}-\tau_n}_{\infty}  \leq \max_{1 \leq k\leq n-2}\abs{\frac{k}{n-1}-\frac{j_n(k)}{2n-1}}+O\left(\frac{1}{n}\right)
     \leq \frac{\max H_{\Tree_n} + O(1)}{n}.
\end{equation*}
By Theorem~\ref{thm:Duquesne}, this last term goes to $0$ as $n$ goes to $+\infty$, so the result is proved.
\end{proof}

\subsection{Lemma~\ref{lem:ComparisonLukaCorot}: the \Luka path of $\corot\Tree_n$}\label{ssct:Luka of Corot}
We deal in this section with the last combinatorial lemma, Lemma~\ref{lem:ComparisonLukaCorot}, which compares $S_{\Tree_n}$ and $S_{\corot\!\Tree_n}$. The same method applies without difficulty, we first give a combinatorial relation between these processes (illustrated by Figure~\ref{fig:ComparisonLukaCorot}) and then translate it into the desired control for the $M_1$ distance between them.

\begin{lem}\label{lem:Link Luka Corot}
    For all $0\leq k \leq n$, set $j_n(k)=2k+S_{\Tree_n}(k)$. Then $j_n$ is strictly increasing from $j_n(0)=0$ to $j_n(n)=2n-1$, and for all $0\leq k \leq n-1$, $S_{\corot\!\Tree_n}$ satisfies on $\Zintff{j_n(k),j_n(k+1)}$:
    \begin{itemize}
        \item if $j_n(k+1)=j_n(k)+1$,
        \begin{equation*}
            S_{\corot\!\Tree_n}\bigl(j_n(k)\bigr) = S_{\Tree_n}(k)\ \text{ and }\ S_{\corot\!\Tree_n}\bigl(j_n(k+1)\bigr) = S_{\Tree_n}(k+1);
        \end{equation*}
        \item else, $S_{\corot\!\Tree_n}$ is affine on $\Zintff{j_n(k),j_n(k+1)-1}$ with
        \begin{equation*}
            S_{\corot\!\Tree_n}\bigl(j_n(k)\bigr) = S_{\Tree_n}(k)\ \text{ and }\ S_{\corot\!\Tree_n}\bigl(j_n(k+1)-1\bigr) = S_{\Tree_n}(k+1)+1,
        \end{equation*}
        and then $S_{\corot\!\Tree_n}\bigl(j_n(k+1)\bigr) = S_{\Tree_n}(k+1)$.
    \end{itemize}
\end{lem}

\begin{figure}[h]
    \centering
    \includegraphics[width=\linewidth]{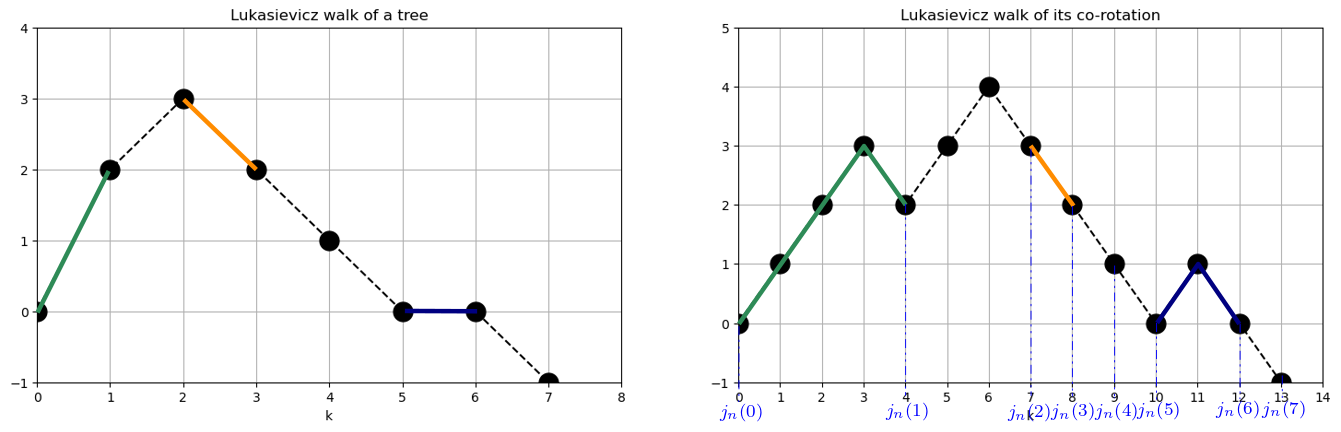}
    \caption{\centering Description of $S_{\corot\!\Tree_n}$ (on the right) based on $S_{\Tree_n}$ (on the left) and the times $(j_n(k), 0\leq k\leq n)$, where $\Tree_n$ is the tree from Figure~\ref{fig:lexicoEnumeration}. Several parts of the two processes have been identified by matching colours to illustrate the case $j_n(k+1) = j_n(k)+1$ (yellow part) and the case $j_n(k+1) \geq j_n(k)+2$ (green and blue parts)}
    \label{fig:ComparisonLukaCorot}
\end{figure}
\begin{proof}
    We prove this lemma for every finite plane tree $T$, instead of $\Tree_n$. We thus consider $j_T(k)=2k+S_T(k)$ for $0 \leq k \leq \#T$.

    We argue by induction on the height of $T$. First, for $T=\{\varnothing\}$, both processes are just one step down and the result is valid. Next, consider $T$ a non-trivial tree and let $T_1,\ldots,T_p$ denote the $p\geq 1$ subtrees grafted on the root of $T$. Notice that by construction of the $\Luka$ walk,
    \begin{equation*}
        S_T = (0,p-1) \boxplus (p-1 + S_{T_1}) \boxplus (p-2 + S_{T_2}) \boxplus \ldots \boxplus S_{T_p},
    \end{equation*}
    where $\boxplus$ means \textit{concatenation with fusion} \ie $(x_1,\dots x_m,a)\boxplus (a,y_1,\ldots y_q) = (x_1,\ldots,x_m,a,y_1,\dots y_q) $. Since $j_T(k)=\sum_{i=1}^k 2+S_T(i)-S_T(i-1)$, it implies that 
    \begin{equation*}
        j_T=(0,p+1)\boxplus (p+1 +j_{T_1}) \boxplus  (p+2\#T_1 +j_{T_2})\boxplus \ldots \boxplus (2+2\#T_1+\ldots+2\#T_{p-1} +j_{T_p}).
    \end{equation*}
    
    Finally, by mean of the recursive definition of $\corot T$ (Figure~\ref{fig:DefRecCorotation}) we also have 
    \begin{equation*}
        S_{\corot T} = (0,1,2,\ldots,p-2,p-1,p,p-1) \boxplus (p-1 + S_{\corot (T_1)}) \boxplus (p-2 + S_{\corot (T_2)}) \boxplus \ldots \boxplus S_{\corot(T_p)}.
    \end{equation*}
    As a consequence, $S_{\corot T}$ satisfies the lemma on $\Zintff{j_T(0),j_T(1)}=\Zintff{0,p+1}$, and it also satisfies the lemma on the next segments since by the induction hypothesis $S_{\corot T_1},\ldots,  S_{\corot T_p}$ already fully satisfy the lemma. 
\end{proof}

Thanks to this relation, we can easily mimic Section~\ref{ssct:HeightAndContour} and introduce some convenient parametric representations of $s_{\Tree_n}$ and $s_{\corot \Tree_n}$  to prove Lemma~\ref{lem:ComparisonLukaCorot}.

\begin{proof}[Proof of Lemma~\ref{lem:ComparisonLukaCorot}.]
We fix $(\lambda(n))_n$ such that $\lambda(n)\rightarrow 0$. 

    We simply take $(s_{\corot \Tree_n},\id_{[0,1]})$ as a parametric representation of $s_{\corot \Tree_n}$, and we introduce two auxiliary sequences $\bigl(A_n(m)\bigr)_{0 \leq m\leq 2n-1}$ (depicted in Figure~\ref{fig:ComparisonLukaAuxiliary}) and $\bigl(B_n(m)\bigr)_{0 \leq m\leq 2n-1}$ to define a parametric representation of $s_{\Tree_n}$. 

    \begin{figure}[h]
    \centering
    \includegraphics[width=\linewidth]{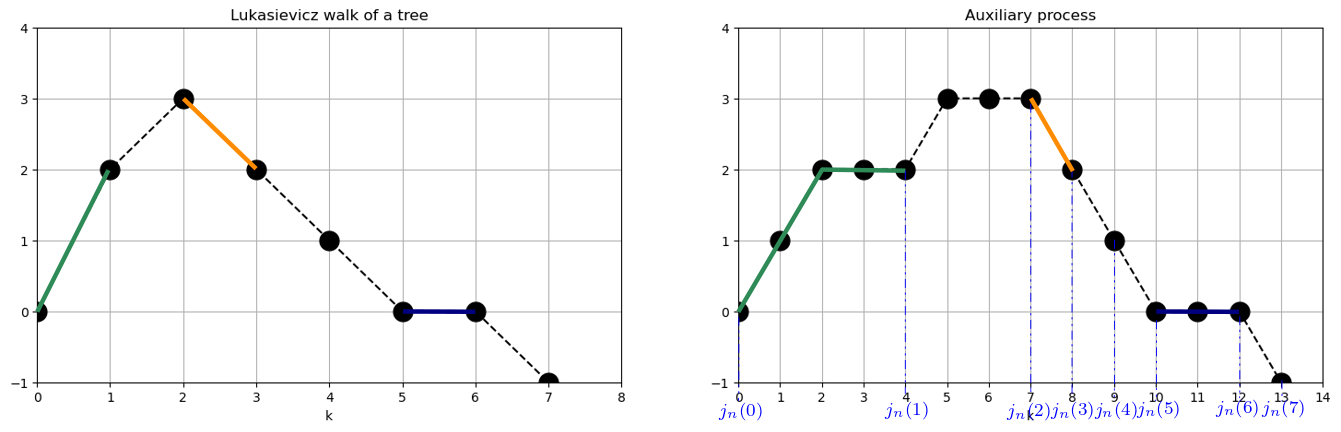}
    \caption{\centering An auxiliary process $A_n$ that approximates $S_{\corot\!\Tree_n}$ (on the right) based on $S_{\Tree_n}$ (on the left) and the times $(j_n(k), 0\leq k\leq n)$, where $\Tree_n$ is the tree from Figure~\ref{fig:lexicoEnumeration}. Several parts of the two processes have been identified by matching colours to illustrate the case $j_n(k+1) = j_n(k)+1$ (yellow part), the case $j_n(k+1) = j_n(k)+2$ (blue part) and the case $j_n(k+1)>j_n(k)+2$ (green part).}
    \label{fig:ComparisonLukaAuxiliary}
\end{figure}
Recall the change of time $j_n$ introduced in Lemma~\ref{lem:Link Luka Corot}.  On each interval $\Zintff{j_n(k),j_n(k+1)}$,
\begin{itemize}
    \item When $j_n(k+1)\leq j_n(k)+2$, $A_n$ and $B_n$ are affine on $\Zintff{j_n(k),j_n(k+1)}$ with: 
    \begin{align*}
        & A_n\bigl(j_n(k)\bigr)=S_{T_n}(k);&\ & B_n\bigl(j_n(k)\bigr)= \frac{k}{n}; \\  
        & A_n\bigl(j_n(k+1)\bigr)=S_{T_n}(k+1);&\ & B_n\bigl(j_n(k+1)\bigr)= \frac{k+1}{n}.
    \end{align*}
    \item When $j_n(k+1)>j_n(k)+2$, $A_n$ and $B_n$ are affine on $\Zintff{j_n(k),j_n(k+1)-2}$ with:
    \begin{align*}
       & A_n\bigl(j_n(k)\bigr)=S_{T_n}(k);&\ & B_n\bigl(j_n(k)\bigr)= \frac{k}{n}; \\  
        & A_n\bigl(j_n(k+1)-2\bigr)=S_{T_n}(k+1);&\ & B_n\bigl(j_n(k+1)-2\bigr)= \frac{k+1}{n};
    \end{align*}
    and then $A_n$ and $B_n$ are constant on $\Zintff{j_n(k+1)-2,j_n(k+1)}$.
   
\end{itemize}

By construction, their time-scaled functions $(a_n,b_n)$ form a parametric representation of $s_{\Tree_n}$. Moreover, it is clear from Lemma~\ref{lem:Link Luka Corot} that $\norm{a_n- s_{\corot \Tree_n}}_{\infty}\leq1$ (compare Figures~\ref{fig:ComparisonLukaCorot} and~\ref{fig:ComparisonLukaAuxiliary} to see this). Thus we get \begin{equation*}
    d_{M_1}\bigl(\lambda(n) s_{\corot \Tree_n}, \lambda(n) s_{\Tree_n}\bigr) \leq \lambda(n) \vee \norm{\id_{[0,1]}-b_n}_{\infty}.
\end{equation*}
Now we control the temporal components thanks to the construction of $B_n$ and Lemma~\ref{lem:Link Luka Corot}:
\begin{equation*}
    \norm{\id_{[0,1]}-b_n}_{\infty} \leq \max_{1 \leq k\leq n-1}\abs{\frac{k}{n}-\frac{j_n(k)}{2n-1}}\vee\abs{\frac{k}{n}-\frac{j_n(k)-2}{2n-1}}  
    \leq \frac{\max S_{\Tree_n} + O(1)}{n}.
\end{equation*}
By Theorem~\ref{thm:Duquesne}, this last term goes to $0$ with high probability as $n$ goes to $+\infty$, so the result holds.
\end{proof}
\begin{rmk}
    We have chosen to state Lemma~\ref{lem:ComparisonLukaCorot} with $\Tree_n$ but it actually holds for every sequence of random trees $(T_n)_n$ such that 
    \begin{align*}
    \#T_n\xrightarrow[\tend{n}]{\prob}+\infty & \text{ and } 
    \frac{\max S_{T_n}}{\# T_n} \xrightarrow[\tend{n}]{\prob} 0.
\end{align*}
\end{rmk}

\newpage

\appendix

\section{The \Luka walk and the mirror transformation}\label{app:MirrorLuka}

In this appendix, we give the full details of the joint convergence stated in Proposition~\ref{prop:DuquesneMiroir}.  It is a simple consequence of Theorem~\ref{thm:Duquesne} in the special case $\alpha=2$, but additional arguments are needed when $\alpha < 2$. Indeed, in the case $\alpha=2$ we have that all encoding processes of $\Tree_n$ converge towards the same Brownian excursion, and the same is true for $\Tree_n^{\div}$. Since the contour process behaves nicely when we apply the mirror transformation, we easily get a joint convergence result for $\bigl(\ell(n)/\sqrt n\times(c_{\Tree_n},c_{\Tree_n^{\div}})\bigr)_n$ and then deduce the desired result. When $\alpha < 2$, the \Luka walks and contour processes no longer converge towards the same excursion and we must study the \Luka walks on their own. The comparison of these two \Luka walks will require keeping track of their jumps in a precise way, which can be done thanks to some specific properties of the \textsc{Skorokhod}'s $J_1$ topology, hence we state the next result with this topology.
 
\begin{prop}\label{prop:LukaMirror}
 Let $\Tree_n$ be distributed as $\BGW\evt{\,\cdot\,\vert\textup{\#vertices = $n$}}$ where the offspring distribution $\mu$ is critical. Assuming that $\mu$ is attracted to a stable distribution of index $\alpha \in (1,2]$, there is a measurable function $F^{\div}:D([0,1])\rightarrow D([0,1])$ such that
    \begin{equation*}
        \frac{1}{\ell(n)n^{1/\alpha}}\bigl(S_{\Tree_n}(\lfloor nt\rfloor),S_{\Tree_n^{\div}}(\lfloor nt\rfloor)\bigr)_{t \in [0,1]} \xrightarrow[\tend{n}]{\dstb} \bigl(\mathbbm x, F^{\div}(\mathbbm x )\bigr)
    \end{equation*}
with respect to \textsc{Skorokhod}'s $J_1$ topology.
\end{prop}

Note that it implies a similar result for the weaker $M_1$ topology and the corresponding continuous time-scaled function $s_{\Tree_n}$ and $s_{\Tree_n^{\div}}$. Moreover, as we still have that the height process is asymptotically the same as the contour process and this last process behaves nicely when we apply the mirror transformation, we eventually get Proposition~\ref{prop:DuquesneMiroir} as a consequence of Proposition~\ref{prop:LukaMirror}.

\begin{proof}[Proof of Proposition~\ref{prop:DuquesneMiroir}.] We apply Theorem~\ref{thm:Duquesne} to $\bigl(\Tree_n\bigr)_n$ and $\bigl(\Tree_n^{\div}\bigr)_n$ to immediately get tightness of the sequence 
\begin{equation*}
    \left(\frac{1}{\ell(n)n^{1/\alpha}}s_{\Tree_n},\frac{\ell(n)}{n^{1-1/\alpha}}h_{\Tree_n},\frac{1}{\ell(n)n^{1/\alpha}} s_{\Tree_n^{\div}},\frac{\ell(n)}{n^{1-1/\alpha}} h_{\Tree_n^{\div}}\right)_n.
\end{equation*}
Now, let $(x_1,x_2,x_3,x_4)$ denote a subsequential distributional limit. By Proposition~\ref{prop:LukaMirror}, $(x_1,x_3)$ has the law of $\bigl(\mathbbm x,F^{\div}(\mathbbm x )\bigr)$ hence  $x_3=F^{\div}(x_1)$ almost surely. In the same way, Theorem~\ref{thm:Duquesne}  gives that $(x_1,x_2,x_4)$ has the law of $\bigl(\mathbbm x,\mathbbm h,\widehat{\mathbbm h}\bigr)$. As there is a measurable function $G:D([0,1])\rightarrow D([0,1])$ such that $\mathbbm h = G(\mathbbm x) \as[ ]$  we must have $x_2=G(x_1)$ and $x_4=\widehat x_2$ almost surely. Since $x_1$ is distributed as $\mathbbm x$, we can conclude that $(x_1,x_2,x_3,x_4)$ is distributed as $\bigl(\mathbbm x,\mathbbm h, F^{\div}(\mathbbm x ),\widehat{\mathbbm h}\bigr)$ and by uniqueness of the limiting distribution we get the desired convergence.
\end{proof}
Let us now discuss the proof of Proposition~\ref{prop:LukaMirror} in the case $\alpha < 2$. By Theorem~\ref{thm:Duquesne} (actually we use its version with respect to the $J_1$ topology, see \eg \cite[Equation~1]{Marzouk2020ScalingLimitStableSnake} for a precise statement), the sequence at stake is tight and it is enough to study the uniqueness in distribution of its subsequential limits. Based on the relation of the two \Luka walks in the discrete setting and properties of the $J_1$ topology, a subsequential limit is a pair of discontinuous functions with some constraints on their discontinuities. To make this formal, for $x\in D([0,1])$ we set $\Delta x(t)=x(t)-x(t-), \disc(x)=\{t\in[0,1] \st x(t-)\neq x(t)\}$ and $\disc_{\varepsilon}(x)=\{t\in[0,1] \st \abs{\Delta x(t)}>\varepsilon\}$. We have the following lemma.

\begin{lem}\label{lem:Description link Luka mirror}
    Let $(X,Y)$ be a subsequential limit in distribution (with respect to $J_1$) of 
    \begin{equation*}
        \left(\frac{1}{\ell(n)n^{1/\alpha}}\bigl(S_{\Tree_n}(\lfloor nt\rfloor),S_{\Tree_n^{\div}}(\lfloor nt\rfloor)\bigr)_{t \in [0,1]} \right)_n .
    \end{equation*} For all $s\leq t \in [0,1]$, let $e(t)=\inf\{ t'\geq t: X_{t'} < X_{t-}\}$ (with $\inf\emptyset =1$) and $I_{s,t}=\inf_{[s,t]} X$. Finally, introduce $(W_t)_{t \in [0,1]}$ such that 
    \begin{equation*}
        W_t=\sum_{0 \leq s < t: X_{s-} < I_{s,t}} X_s-I_{s,t},
    \end{equation*} where the sum is countable since $(X_s-I_{s,t})\1_{X_{s-} \leq I_{s,t}}=0$ for $s\not\in\disc(X)$. Note that $W$, simply seen as a collection of positive variables, is measurable with respect to $X$.
    
    Then almost surely:
    \begin{itemize}
        \item $\disc(X)$ and $\disc(Y)$ are dense in $[0,1]$;
        \item $\disc(Y)=\{1-e(t), \forall t \in \disc(X)\}$ and  $\forall t \in \disc(X), \Delta Y_{1-e(t)}=\Delta X_t$;
        \item $\forall t\in [0,1]\cap\Q,\ Y_{1-t} = W_t$.

    \end{itemize}
\end{lem}

\begin{proof}
    First we know by Theorem~\ref{thm:Duquesne} that both $X$ and $Y$ are distributed as an excursion of an $\alpha$-stable \textsc{Lévy} process hence we have that $\disc(X)$ and $\disc(Y)$ are dense in $[0,1]$ almost surely (see \cite[Proposition~2.10]{Kortchemski2014StableLaminations}). The other assertions are inherited from relations between the discrete processes $S_{\Tree_n}$ and $S_{\Tree_n^{\div}}$ that we describe now.

    Consider $u_0,\ldots,u_{n-1}$ the lexicographical enumeration of $\Tree_n$ as well as $w_0,\ldots,w_{n-1}$ its mirrored enumeration. For $0\leq k < n$, set $E_n(k)=\min\{m\geq k: S_{\Tree_n}(m) < S_{\Tree_n}(k)\}$, so that $E_n(k)-k$ is the size of the subtree of descendants of $u_k$ in $\Tree_n$. The set of vertices that are greater than or equal to $u_k$ for the mirrored lexicographical order contains those descendants as well as the $k-H_{\Tree_n}(k)$ vertices with an ancestor which is a younger sibling of an ancestor of $u_k$, hence we have $u_k = w_{n-E_n(k)+H_{\Tree_n}(k)}$. As a consequence, $k \mapsto n-E_n(k)+H_{\Tree_n}(k)$  is a permutation of $\Zintff{0,n-1}$ that exchanges the jumps $\Delta S_{\Tree_n^{\div}}$ and $\Delta S_{\Tree_n}$ (where $\Delta A(k)=A(k+1)-A(k)$):
    \begin{equation}\label{eq:same jumps}
        \Delta S_{\Tree_n^{\div}}\bigl(n-E_n(k)+H_{\Tree_n}(k)\bigr) = \Delta S_{\Tree_n}(k).
    \end{equation}
Moreover, for a fixed $k$ and $k'= n-E_n(k)+H_{\Tree_n}(k)$, we get that $S_{\Tree_n^{\div}}(k') =L(u_k)$ (recall that $L,R$ are defined in Section~\ref{sct:planeTrees}), and this may be expressed with $S_{\Tree_n}$ as follows: 
\begin{equation}\label{eq:Left from Standard Luka}
    S_{\Tree_n^{\div}}(k')=\sum_{j:u_j \in \Zintof{\varnothing,u_k}} L(u_j)-L(\text{parent}(u_j))=\sum_{j<k:S_{\Tree_n}(j)=\min_{\Zintff{j,k}}  S_{\Tree_n}} S_{\Tree_n}(j+1)-\min_{\Zintff{j+1,k}}S_{\Tree_n} .
\end{equation}
Notice that we could argue directly in $\Tree_n^{\div}$, since the mirror transformation preserves the ancestral line and simply exchanges $L$ and $R$, to get another expression of the same sum:
\begin{equation}\label{eq:Right from Luka}
    S_{\Tree_n^{\div}}(k')=\sum_{j<k':S_{\Tree_n^{\div}}(j)=\min_{\Zintff{j,k'}}S_{\Tree_n^{\div}}} \left(\min_{\Zintff{j+1,k'}}S_{\Tree_n^{\div}}-S_{\Tree_n^{\div}}(j)\right)
\end{equation}
where additionally, each term in the right sum of \eqref{eq:Left from Standard Luka} corresponds to and is equal to a term in the right sum of \eqref{eq:Right from Luka}.

Sending $n$ to $+\infty$ will yield the desired assertions. To see this, first notice that the lemma only concerns the law of $(X,Y)$ (indeed all three assertions correspond to a measurable events), thus by \textsc{Skorokhod}'s representation theorem we may further assume that we have an almost sure convergence towards $(X,Y)$ with respect to \textsc{Skorokhod} $J_1$ topology. It means that for all $n$ we have some processes $X^n,Y^n$ and some increasing bijections  $\phi_n,\psi_n:[0,1]\mapsto[0,1]$ such that 
\begin{equation*}
    \forall u\in [0,1],\ X^n\circ\phi_n(u)=\frac{1}{\ell(n)n^{1/\alpha}}S_{\Tree_n}(\lfloor nu\rfloor),\  Y^n\circ\psi_n(u)=\frac{1}{\ell(n)n^{1/\alpha}}S_{\Tree^{\div}_n}(\lfloor nu\rfloor),
\end{equation*}
and \begin{equation*}
    (X^n,Y^n) \xrightarrow[\tend{n}]{\norm{\cdot}_{\infty}} (X,Y);\  (\phi_n,\psi_n) \xrightarrow[\tend{n}]{\norm{\cdot}_{\infty}} (\id,\id).
\end{equation*}
Since we have $\{t\in[0,1] \st \Delta X(t)>\varepsilon\}=\{t\in[0,1] \st \Delta X(t) \geq \varepsilon\}$ for all $\varepsilon>0$ almost surely, this convergence implies that for any given $\varepsilon>0$, when $n$ is large enough the change of time $\phi_n$ maps the jumps of size $> \varepsilon$ of $X$ onto those of $S_{\Tree_n}/\ell(n)n^{1/\alpha}$:
\begin{equation*}
    \disc_{\varepsilon}(X)=\left\{\phi_n\left(\frac{k+1}{n}\right), \text{for } k<n \st \Delta S_{\Tree_n}(k) > \varepsilon \ell(n)n^{1/\alpha}\right\} \text{for $n$ large enough}.
\end{equation*}
An analogue statement holds for $Y$ and $S_{\Tree_n^{\div}}$,  which may be combined to \eqref{eq:same jumps} to give
\begin{equation*}
    \disc_{\varepsilon}(Y)=\left\{\psi_n\left(\frac{n-E_n(k)+H_{\Tree_n}(k)+1}{n}\right), \text{for } k<n \st \Delta S_{\Tree_n}(k) > \varepsilon \ell(n)n^{1/\alpha}\right\}.
\end{equation*}
Let us study the asymptotic of the right-hand side terms to obtain the desired link between $\disc(X)$  and $\disc(Y)$. For all $t \in (0,1)$ and for all $n$, let $k_n(t)$ be such that $t \in \bigl(\phi_n\bigl(k_n(t)/n\bigr),\phi_n\bigl((k_n(t)+1) / n\bigr)\bigr]$ and set $m_n(t)= n-E_n(k_n(t))+H_{\Tree_n}(k_n(t))$. Note that if $t\in \disc(X)$ then we must have $t=\phi_n\bigl((k_n(t)+1) / n\bigr)$ for $n$ large. We first prove that almost surely, for all $t\in [0,1], \psi_n \bigl( (m_n(t)+1)/n\bigr)\rightarrow 1-e(t)$ as $n$ goes to $+\infty$. It is clear that $k_n(t) \sim tn$ and  $H_{\Tree_n}(k_n(t))=o(n)$, so it is sufficient to prove that $ E_n(k_n(t))\sim e(t)n$. To see this, observe that by construction we have 
\begin{equation*}
    \phi_n\left( \frac{E_n(k_n(t))}{n}\right)=\inf\{t' \geq t \st X^n_{t'} < X^n_{t-}\}.
\end{equation*} 
For all $t$ such that $e(t)=t$, this quantity clearly converges towards $t$. Moreover, properties of $\alpha$-stable excursions implies that, almost surely, for all $t$ such that $e(t)>t$ and for all $\delta >0$ small enough we have
\begin{equation*}
    \inf_{[t,e(t)-\delta]} X > X_{t-} \text{ and } \inf_{[t,e(t)+\delta]} X < X_{t-}.
\end{equation*}
(see \cite[Proposition~2.10]{Kortchemski2014StableLaminations} again). Thus we always have $\phi_n \bigl( E_n(k_n(t))/n\bigr)\rightarrow e(t)$ and it gives that $\forall t\  E_n(k_n(t))\sim e(t)n$.  

As an almost sure consequence, for all $\varepsilon > 0$, we see that $\forall t \in \disc_{\varepsilon}(X),\  \psi_n \bigl( (m_n(t)+1)/n\bigr)\rightarrow 1-e(t)$ and eventually every point of this sequence is in the finite set $\disc_{\varepsilon}(Y)$, hence $1-e(t)\in \disc_{\varepsilon}(Y)$. Conversely, for $s \in \disc_{\varepsilon}(Y)$, for $n$ large enough we have $k_n$ such that $s=\psi_n\bigl((n-E_n(k_n)+H_{\Tree_n}(k_n)+1)/n\bigr)$, and since $\disc_{\varepsilon}(X)$ is finite we may assume (by extraction) that there is $t\in \disc_{\varepsilon}(X)$ such that $k_n=k_n(t)$ for large $n$. The previous argument directly gives that $s=1-e(t)$ (and $1-e(t)=\psi_n \bigl( (m_n(t)+1)/n\bigr)$ for $n$ large) thus we have
\begin{equation*}
    \disc_{\varepsilon}(Y)=\{1-e(t), \forall t \in \disc_{\varepsilon}(X)\}.
\end{equation*}
As this holds for all $\varepsilon > 0$, we get that $\disc(Y)=\{1-e(t), \forall t \in \disc(X)\}$ and that  $\forall t \in \disc(X)$, $ \Delta Y_{1-e(t)}=\Delta X_t$, almost surely.

Finally, we prove that almost surely $\forall t\in [0,1]\cap\Q,\ Y_{1-t} = W_t$. First, we fix $t\in (0,1)$. By properties of $\alpha$-stable excursions, almost surely $1-t \not\in \disc(Y)$, $t \not \in \disc(X)$ and the equation $e(y)=t$ has only one solution, namely $y=t$, because for all $\varepsilon >0$ small enough we have $\inf_{[t,t\pm \varepsilon]} X < X_t$. Consequently $Y_{1-t}=Y_{1-e(t)}$ almost surely, and by the previous discussion this term can be approximated by $S_{\Tree_n^{\div}}\bigl(m_n(t)\bigr)/\ell(n)n^{1/\alpha}$. We will introduce closely related quantities and prove that they converge both toward $Y_{1-e(t)}$ and $W_t$ to conclude. 

We start by rewriting \eqref{eq:Left from Standard Luka} as 
\begin{equation*}
    S_{\Tree_n^{\div}}\bigl(m_n(t)\bigr)= \sum_{j<k_n(t):S_{\Tree_n}(j)=\min_{\Zintff{j,k_n(t)}}S_{\Tree_n}}\!\!\!\!\!\!\!\!\!\!\!\!\!\!\!\!\!\!\!\!\!\!\!\!\!\!\! \left(S_{\Tree_n}(j+1)-\min_{\Zintff{j+1,k_n(t)}}S_{\Tree_n}\right)=\ell(n)n^{1/\alpha} \!\!\!\!\!\!\!\!\!\sum_{0 \leq s < t: X^n_{s-} \leq I^n_{s,t}} \!\!\!\!\!\!\!(X^n_s-I^n_{s,t}),
\end{equation*}
where $I^n_{u,v}=\inf_{[u,v]}X^n$. We also introduce $\delta_{s,t}(x)= \bigl(x(s)-\inf_{[s,t]}x\bigr)\wedge \bigl(\inf_{[s,t]}x-x(s-)\bigr)$ for $x\in D([0,1])$ and $s<t$, and for all $\varepsilon >0$, we define $$r_n(\varepsilon)=\sum_{0 \leq s < t: X^n_{s-} \leq I^n_{s,t}} (X^n_s-I^n_{s,t})\1_{\delta_{s,t}(X^n) > \varepsilon}\ .$$ Since we only take into account a finite number of jumps, we can deduce that
\begin{equation*}
   r_n(\varepsilon) \cv{n} r_{\infty}(\varepsilon):=\sum_{0 \leq s < t:X_{s-} \leq I_{s,t}} (X_s-I_{s,t})\1_{\delta_{s,t}(X) > \varepsilon}\ .
\end{equation*}
Moreover $r_{\infty}(\varepsilon)$ converges towards $\sum_{0 \leq s < t: X_{s-} < I_{s,t}} X_s-I_{s,t}=W_t$ as $\varepsilon$ goes to $0$. But we can also express $r_n(\varepsilon)$ thanks to \eqref{eq:Right from Luka} and its link with \eqref{eq:Left from Standard Luka} as it enables us to write
\begin{equation*}
    \ell(n)n^{1/\alpha} r_n(\varepsilon)= \sum_{j<m_n(t):  S_{\Tree_n^{\div}}(j)=\min_{\Zintff{j,m_n(t)}}S_{\Tree_n^{\div}}}\left(\min_{\Zintff{j+1,m_n(t)}}S_{\Tree_n^{\div}}-S_{\Tree_n^{\div}}(j)\right)\1_{\Delta_{j,m_n(t)} > \varepsilon \ell(n)n^{1/\alpha}}\ ,
\end{equation*}
where $\Delta_{j,m}=\bigl(\min_{\Zintff{j+1,m}}S_{\Tree_n^{\div}}-S_{\Tree_n^{\div}}(j)\bigr) \wedge \bigl(S_{\Tree_n^{\div}}(j+1)- \min_{\Zintff{j+1,m}}S_{\Tree_n^{\div}}\bigr)$. By expressing the right-hand side term with $Y^n$, which converges toward $Y$ with respect to $J_1$, we deduce another expression of the limit $r_{\infty}(\varepsilon)$. 
\begin{equation*}
r_{\infty}(\varepsilon) = \sum_{0 \leq s < 1-e(t): Y_{s-} \leq \inf_{[s,1-e(t)]}Y} \left(\inf_{[s,1-e(t)]}Y - Y_{s-} \right)\1_{\delta_{s,t}(Y) > \varepsilon}.
\end{equation*}
We conclude thanks to the following property of the $\alpha$-stable excursion $Y$, established in \cite[Corollary~3.4]{CurienKortchemski2014StableLooptrees}:
\begin{equation*}
    Y_{1-e(t)}=\sum_{0 \leq s \leq 1-e(t): Y_{s-} \leq \inf_{[s,1-e(t)]}Y} \left(\inf_{[s,1-e(t)]}Y - Y_{s-} \right).
\end{equation*}
Since $1-t=1-e(t) \not\in \disc(Y)$, we deduce from this that $r_{\infty}(\varepsilon)$ converges toward $Y_{1-t}$ as $\varepsilon$ goes to $0$, hence $Y_{1-t}=W_t$.
\end{proof}

We now end this section with a proof Proposition~\ref{prop:LukaMirror} based on the explicit description of a subsequential limit given by Lemma~\ref{lem:Description link Luka mirror}.

\begin{proof}
    Let $(X,Y)$ be a subsequential limit in distribution (with respect to $J_1$) of 
    \begin{equation*}
        \left(\frac{1}{\ell(n)n^{1/\alpha}}\bigl(S_{\Tree_n}(\lfloor nt\rfloor),S_{\Tree_n^{\div}}(\lfloor nt\rfloor)\bigr)_{t \in [0,1]} \right)_n.
    \end{equation*}
    
    Recall from Lemma~\ref{lem:Description link Luka mirror} that almost surely $\forall t \in [0,1]\cap\Q, Y_{1-t} = W_t$. This shows that $\sigma(Y) = \sigma(Y_s, s\in[0,1]\cap\Q)  \subset \sigma(W)\vee\sigma(\mathcal N) \subset \sigma(X)\vee\sigma(\mathcal N)$ where $\mathcal N$ is the set of $\P$-negligible events. Since $D([0,1])$ is a polish space, \textsc{Doob-Dynkin} lemma guarantees the existence of a measurable function $F^{\div}$ such that $Y=F^{\div}(X)$ almost surely. In particular, almost surely $F^{\div}(X)$ is the unique element  in $D([0,1])$ such that $\forall t \in [0,1]\cap\Q, F^{\div}(X)_{1-t}=W_t$.
    Let us now consider another subsequential limit in distribution $(X',Y')$. Then $X$ and $X'$ both are $\alpha$-stable excursions, hence almost surely $F^{\div}(X')$ is the unique element  in $D([0,1])$ such that $\forall t \in [0,1]\cap\Q, F^{\div}(X')_{1-t}=W'_t$. Moreover, the consequence of Lemma~\ref{lem:Description link Luka mirror} also apply to $(X',Y')$, hence we have $Y'=F^{\div}(X')$ almost surely. This shows that $(X',Y')$ is always distributed as $(X,F^{\div}(X))$, and the desired convergence follows.
\end{proof}
\begin{rmk}
    There may be a more explicit expression for $F^{\div}$: Define $(\widetilde W_t)_{t \in [0,1]}$ by 
    \begin{equation*}
        \widetilde W_t=\sum_{0 \leq s < t: X_{s-} \leq I_{s,t}} X_s-I_{s,t}.
    \end{equation*}
The only difference with $W_t$ is that we also consider the jumps such that $X_{s-} = I_{s,t}$. For a given $t\in[0,1]$, almost surely $\widetilde W_t =W_t$. However, it is conjectured that $t \mapsto \widetilde W_t$ is left-continuous. If this is true then almost surely $F^{\div}(X) = t \mapsto \widetilde W_{1-t}$, but we have not been able to prove this claim.
\end{rmk}

\bibliographystyle{alpha}
\bibliography{mabiblio}
\end{document}